\documentclass[10pt,a4paper,leqno]{amsart}
\usepackage{amsmath}
\usepackage{amssymb}
\usepackage{amsfonts}
\usepackage{amsthm}
\usepackage{mathrsfs}
\usepackage{dsfont}
\usepackage{mathtools}
\usepackage{bm}
\usepackage{todonotes}
\usepackage{nicefrac}
\usepackage{esint}
\usepackage{bbm}		
\usepackage{enumitem}
\usepackage[backref=page]{hyperref}
\usepackage{color}
\usepackage[normalem]{ulem}   

\definecolor{darkgreen}{rgb}{0.5,0.25,0}
\definecolor{darkblue}{rgb}{0,0,1}
\definecolor{answerblue}{rgb}{0,0,0.75}

\hypersetup{colorlinks,breaklinks,
linkcolor=darkblue,urlcolor=darkblue,
anchorcolor=darkblue,citecolor=darkblue}

\newcommand*{\mailto}[1]{\href{mailto:#1}{\nolinkurl{#1}}}
\newcommand{\ep}{\varepsilon}
\newcommand{\eps}{\varepsilon}
\DeclareMathOperator*{\Ex}{\mathbb{E}}

\renewcommand{\d}{\mathrm{d}}
\newcommand{\average}{{\square}}

\newcommand{\loc}{\mathrm{loc}}

\newcommand{\cell}{{\mathcal{C}}}
\DeclareMathOperator{\diag}{diag}

\DeclareMathOperator{\Div}{div}

\newcommand{\R}{\mathbb{R}}
\newcommand{\T}{{\mathbb{T}}}
\newcommand{\N}{\mathbb{N}}
\newcommand{\Z}{\mathbb{Z}}
\newcommand{\U}{\mathfrak{U}}

\newcommand{\bk}[1]{ \left(  #1 \right)}
\newcommand{\abs}[1]{ \left|  #1 \right|}
\newcommand{\norm}[1]{ \left\|  #1 \right\|}
\newcommand{\one}[1]{\mathds{1}_{#1}}
\newcommand{\Dt}{{\Delta t}}
\newcommand{\Dx}{{\Delta x}}

\newcommand{\hf}{{\nicefrac12}}
\newcommand{\thf}{{\nicefrac32}}
\renewcommand{\le}{\leqslant}
\renewcommand{\leq}{\leqslant}
\renewcommand{\ge}{\geqslant}
\renewcommand{\geq}{\geqslant}
\renewcommand{\phi}{\varphi}
\newcommand{\from}{\colon}
\newcommand{\nqquad}{\hspace{-2em}}
\newcommand{\dconv}{\mathbin{\bm{*}}} 

\newcommand{\relspace}{\hphantom{{}={}}}

\newtheorem{theorem}{Theorem}[section]
\newtheorem{proposition}[theorem]{Proposition}
\newtheorem{lemma}[theorem]{Lemma}

\theoremstyle{definition}
\newtheorem{definition}[theorem]{Definition}

\newtheorem{remark}[theorem]{Remark}
\newtheorem{assumption}[theorem]{Assumption}

\allowdisplaybreaks
\numberwithin{equation}{section}

\title[Numerical scheme for stochastic transport equation]{Convergent finite difference schemes for stochastic transport equations}

\author[Fjordholm]{U. S. Fjordholm}
\address[Ulrik S. Fjordholm]{Department of Mathematics\\
   University of Oslo\\
  NO-0316 Oslo\\ Norway}
\email{\mailto{ulriksf@math.uio.no}}

\author[Karlsen]{K. H. Karlsen}
\address[Kenneth H. Karlsen]{Department of Mathematics\\
   University of Oslo\\
  NO-0316 Oslo\\ Norway}
\email{\mailto{kennethk@math.uio.no}}

\author[Pang]{P.H.C. Pang}
\address[Peter H.C. Pang]{Department of Mathematics\\
   University of Oslo\\
  NO-0316 Oslo\\ Norway}
\email{\mailto{ptr@math.uio.no}}



\thanks{We gratefully acknowledge the support of the Research
Council of Norway through the project \textit{INICE} (301538). We thank Welid Alebady for many helpful comments on the manuscript.}

\date{\today}

\begin{document}
\begin{abstract}
We present difference schemes for stochastic transport equations with low-regularity velocity fields. We establish $L^2$ stability and convergence of the difference approximations under conditions that are less strict than those required for deterministic transport equations. The $L^2$ estimate, crucial for the analysis, is obtained through a discrete duality argument and a comprehensive examination of a class of backward parabolic difference schemes.
\end{abstract}

\maketitle

\setcounter{tocdepth}{2}
\tableofcontents

\section{Introduction}\label{sec:intro}
In this paper we design and analyse finite difference schemes
for transport equations with gradient noise.
We allow for data of lower regularity than what is
conventionally assumed for deterministic equations.
Our approach builds upon and extends the
well-posedness theory of weak solutions
which was initially introduced and extensively studied
in prior works such as \cite{Attanasio:2011fj,Beck:2019aa}.

For some $T>0$ and $d\in\N$, we consider
the stochastic transport equation
\begin{equation}\label{eq:stratonovich_eq}
\begin{aligned}
	&\d u +  V \cdot \nabla u \,\d t
	+ \sigma  \nabla u \circ \d W = 0,
	&& (\omega, t,x)
	\in \Omega \times [0,T] \times \R^d, \\
	&u(0,x) = u^0(x), && x\in\R^d,
\end{aligned}
\end{equation}
where the velocity field $V\from\R^d\to\R^d$,
the noise coefficient $\sigma\from\R^d\to\R_+$,
and the initial data $u^0\from\R^d\to\R$
are all given, and 
\[
\sigma\nabla u\circ dW = \sum_{i=1}^d \sigma(x) \frac{\partial u}{\partial x^j}(x,t) \circ dW^j(t),
\]
where $W^1,\dots,W^d$ are independent Brownian motions in $\R^d$ and ``$\circ\ dW^j$'' signifies
temporal white noise in Stratonovich form.
Here and in the remainder, we make the
following assumptions.

\begin{assumption}\label{ass:main_assumption}
The noise coefficient $\sigma$ lies in the
Sobolev space $W^{3,\infty}(\R^d)$ and
satisfies $\sigma(x)\geq\sigma_0$ for all
$x\in\R^d$ for some number $\sigma_0>0$. The
velocity field $V$ lies in $L^2\cap L^\infty(\R^d)$
with $\partial_{x^1}V^1,\ldots,\partial_{x^d}V^d \in L^p(\R^d)$
for some $p > d$ for $d \ge 2$ and $p \ge 2$ for $d = 1$. The initial data $u^0$
lies in $L^2(\R^d)$. 
\end{assumption}

Note that the velocity field is allowed to be rather irregular:
We do \emph{not} require that $\Div V\in L^\infty$, as is typically
done in deterministic transport theory\footnote{For 
technical reasons (in particular, Lemma \ref{lem:V_DV_discretelimits1}(ii)), 
we were unable to relax our condition $\partial_{x^1}V^1,	
	\dots,\partial_{x^d}V^d \in L^p(\R^d)$ to the more natural 
$\Div V = \partial_{x^1}V^1+\dots+\partial_{x^d}V^d \in L^p(\R^d)$.}. 
The objective of this
paper is to establish the convergence of
carefully constructed finite difference schemes for
approximating weak $L^2$ solutions of \eqref{eq:stratonovich_eq}.

Our approach to analysing our difference schemes is
rooted in the classical Lax--Richtmeyer principle
of proving \textit{consistency} and \textit{stability}.
While consistency with \eqref{eq:stratonovich_eq}
will be rather straightforward, the crux of
our challenge resides in fulfilling the stability
criterion---here construed as the
ability to bound the $L^2$ norm of the
difference approximations, independently
of the grid size.  This challenge is linked to
the irregular velocity $V(x)$ as well as the
gradient noise $\sigma(x)\nabla u \circ \d W$
of \eqref{eq:stratonovich_eq}. The inherent
irregularity of $V$ (in particular, $\Div V\notin L^\infty$)
precludes the customary proof of discrete
$L^2$ estimates via the Gronwall inequality.
We will explore  the specific difference scheme
that we propose and examine how we effectively
tackle the challenges later in this introduction.

\subsection{Overview of related works}
Deterministic transport equations ($\sigma \equiv 0$)
play a significant role in modelling the evolution
of quantities being transported by a given
velocity field $V$. These equations find
widespread applications in various areas of
fluid mechanics. In addition to the
transport equations themselves, the closely
related continuity equations $\partial_t u + \Div(Vu) = 0$
and their non-linear counterparts, known as
(hyperbolic) conservation laws, represented
by $\partial_t u + \Div(Vf(u)) = 0$ (where $f$
is a non-linear function), have been intensively
studied from both mathematical
and numerical perspectives, resulting in a wealth
of knowledge and techniques.
When considering Lipschitz velocities
$V$, the books \cite{Eymard:2000fr,
	Holden:2015aa,Kroner:1997lq,Perthame:2002qy}
offer comprehensive introductions to the mathematical
and numerical theory, providing stability and
convergence results, as well as $L^1$ error estimates
for various classes of numerical schemes. For some
works specifically dedicated to linear equations
and finite difference/volume schemes, see
\cite{Delarue:2011aa,Despres:2004aa,Merlet:2007aa,Vila:2003aa}.

For deterministic linear equations with
non-smooth velocity $V$, characterised
by $V\in W^{1,p}$ (for some $p\ge 1$)
and $\Div V\in L^\infty$, various studies
have investigated the convergence properties
of discontinuous Galerkin, finite difference,
and finite volume schemes. Interested
readers are referred to \cite{Boyer:2012ab,
	Schlichting:2017aa,Schlichting:2018aa,Walkington:2005jl}
for more details on these works (this list
is by no means exhaustive). In addition,
\cite{Ben-Belgacem:2019aa} provides
insights into difference approximations of
non-linear conservation laws involving such
velocity fields. Several of these investigations
leverage the (DiPerna--Lions) renormalisation
property of weak solutions to establish
stability and convergence. Others rely on
stability estimates measured through
Kantorovich--Rubinstein type distances.
However, a shared requirement in all of these
studies is that (the negative part of) $\Div V$
belongs to $L^\infty$. Moreover, there exists
a wealth of related literature dedicated to schemes
designed for handling non-smooth velocities
originating from separate equations within the
framework of Navier--Stokes type systems.
To provide some illustrative examples, we
refer to \cite{Karper:2013aa,Liu:2007fe}.

For stochastic transport, the total velocity
field is subject to randomness and takes
the form $V + \sigma \dot{W}$, where
$\dot{W}$ represents temporal white noise,
introducing uncertainty into the gradient
part of \eqref{eq:stratonovich_eq}. We are
not aware of any rigorous convergence
analyses pertaining to numerical approximations
of \eqref{eq:stratonovich_eq}
by finite differences, irrespective
of the smoothness of $V$. We
mention the paper \cite{Hoel:2018aa},
which explores the somewhat related equation
$\partial_t u + \partial_x f(u)\circ \dot{W}=0$.
Notably, the analysis of this particular equation
includes a non-linear flux component, but excludes
a velocity field. Besides, the examination is
carried out pathwise, eliminating the need
for stochastic calculus. We also point
out that using the connection of transport noise
to multiplicative perturbation of characteristics,
a scheme based on SDE results with very rough $V$
was studied in \cite[Section 8]{LL2022}, provided
that the initial datum is smooth. Fully discrete
schemes based on Galerkin approximations for
equations with transport noise have also been
proposed, such as in \cite{CCDPS2019} and references
there cited, where they were numerically investigated.

Despite the scarcity
of research on numerical schemes for stochastic
transport equations, a significant body of
work is available on different classes of SPDEs.
These SPDEs, often of parabolic nature, are
driven by a lower-order stochastic forcing
term $\Phi \, dW$. For an overview and
additional resources, we refer to \cite{Banas:2014aa,
	Jentzen:2011aa,Lord:2014aa,Ondrejat:2022aa}.
Historically, numerical analysis of these
SPDEs focused on finite element methods
and their error estimates. Our research is
more aligned with recent studies that
explore finite volume schemes for non-linear
conservation laws under the influence of
stochastic forcing. These studies establish
convergence using compactness arguments.
There is a growing body of work in this field,
and while it would be impossible to mention
all of them, we highlight a few examples
\cite{Bauzet:2016ab,Bauzet:2016aa,Dotti:2020aa,Funaki:2018aa,
Kroker:2012fk,Majee:2018aa}. We also note that
the $L^2$ stability of discontinuous Galerkin
methods for such SPDEs was confirmed in 
\cite{Li:2020aa}.

\subsection{Outline of techniques}\label{sec:tech_outline}
Let us now give a more detailed exposition of
the primary contributions of this paper.
To facilitate an uncomplicated introduction,
we assume $d=1$ and consider a semi-discrete
difference approximation $u_i(t)\approx u(t,i\Delta x)$.
Here, $i\in \Z$ and $\Delta x>0$ represents the grid size.
We use $u_{\Delta x}$ to symbolise the conventional
piecewise constant interpolation (in $x$) of
the values $(u_i)_{i\in \Z}$. The deterministic
portion of \eqref{eq:stratonovich_eq} will be
discretised using a standard upwind scheme.
This process approximates $V \cdot \nabla u$ as
$V_i^+ \cdot \nabla_- u_i-V_i^-\cdot \nabla_+ u_i$,
where $\nabla_\pm$ signify the common
forward/backward difference operators,
and where $V_i^\pm = \max(\pm V_i,0)$.
For more details and further notations, see Section
\ref{sec:nummethod}. Deterministic upwind
schemes are recognised for their $L^2$ stability,
at least if $\Div V\in L^\infty$. Furthermore,
``upwinding'' leads to weak $H^1$ estimates:
$\iint \abs{\partial_x u_\Dx}^2\, \d x \, \d t
	\lesssim 1/\Dx$,
which are frequently used in the study of upwind
schemes for transport equations and conservation
laws (see e.g.~\cite{Eymard:2000fr}). However, we
emphasise that our analysis does not make use of
weak $H^1$ estimates.

Let us now shift our attention to the discretisation of the
noise component of \eqref{eq:stratonovich_eq}.
First, in order to utilise the martingale property
of stochastic integrals, we will
convert $\sigma \nabla u \circ \d W$ into the It\^o form:
\begin{equation}\label{eq:hyp-op-intro}
	\sigma \nabla u\cdot \d W- \frac12 \sigma \nabla
	\cdot\bk{\sigma \nabla u} \,\d t.
\end{equation}
This transformation elucidates the need for precise
discretisation of the variable-coefficient
gradient term $\sigma\nabla u$ and
the variable-coefficient second order parabolic operator
$\frac12 \sigma\nabla\cdot\bk{\sigma\nabla u}$.
Importantly, these discretisations must be interrelated in
a way that respects the hyperbolic character
of \eqref{eq:stratonovich_eq}. Indeed, applying the
It\^{o} chain rule to $\bigl(u_i(t)\bigr)^2$ unveils a term
$\mathcal{E}^\sigma_i$ (see \eqref{eq:scheme2-intro} below,
and precisely defined following \eqref{eq:ito_eq2})
triggered by the martingale
component of the difference scheme, alongside
dissipation linked to the chosen discretisation of
the second-order (Stratonovich--It\^{o}) parabolic
term. Control over this particular term ($\mathcal{E}^\sigma_i$),
which very much depends on the chosen
discretisations, is a \textit{first} crucial step
towards ensuring the $L^2$ stability of the scheme.

Let us introduce the notation $\overline{\sigma}_i
=\frac{\sigma_i+\sigma_{i+1}}{2}$ and define
$H_i(\overline{\sigma})=\frac{2}{\frac{1}{\overline{\sigma}_{i-1}}
+\frac{1}{\overline{\sigma}_i}}$, the harmonic mean of
$\overline{\sigma}_{i-1}$ and $\overline{\sigma}_i$.
We we will use the following
discretisation of \eqref{eq:hyp-op-intro}:
\begin{equation*}
	\sqrt{\sigma_i H_i(\overline{\sigma})}\nabla_0u_i \, \d W
	-\frac12 \sigma_i \nabla_+
	\bigl(\overline{\sigma}_{i-1}\nabla_-u_i\bigr)\,\d t,
\end{equation*}
where $\nabla_0$ is the standard central difference operator.
Although this discretisation may seem complex, it possesses a
crucial property of preserving the hyperbolic
character of \eqref{eq:stratonovich_eq}, in the sense that
\begin{equation}\label{eq:scheme2-intro}
	\begin{split}
		\mathcal{E}^\sigma_i
		&:= u_i\sigma_i
		\nabla_+\bigl(\overline{\sigma}_{i-1}\nabla_- u_i\bigr)
		- \sigma_i\nabla_+\!
		\left(\overline{\sigma}_{i-1}\nabla_-
		\!\left(\frac{u_i^2}{2}\right)\right)
		+ \sigma_i H(\overline{\sigma}_i)|\nabla_0u_i|^2\leq 0.
	\end{split}
\end{equation}
If $\sigma\equiv 1$, the
property $\mathcal{E}^\sigma_i \leq 0$ is
dependent upon the discrete commutator formula
$\beta'(u_i)\nabla_-\nabla_+ u_i
-\nabla_-\nabla_+ \beta(u_\alpha)
=-\abs{\nabla_+ u_i}^2-\abs{\nabla_- u_i}^2$
for $\beta(u)=u^2$.
This formula mimics
the continuous identity $\beta'(u)\Delta u-\Delta \beta(u)
=-\beta''(u) |\nabla u|^2$. To summarise:
our specific discretisation of the noise
term implies that the error term
$\mathcal{E}^\sigma_i$ that appears
in the estimate on $\bigl(u_i(t)\bigr)^2$, can be ignored.

At this juncture, we shall leverage a
``regularisation by noise'' idea introduced
in \cite{Attanasio:2011fj} (see also \cite{Beck:2019aa}).
Broadly speaking, regularisation by noise refers
to the observation that introducing noise into a
deterministic system may yield advantageous
effects. These effects may include enhancing
the well-posedness of the system, preventing
solution blow-up, and improving the regularity
of the solution. The study of regularisation by
noise in SDEs and SPDEs has become a well-established
field with a vast amount of literature. For an
introductory understanding of this evolving
area of research, we recommend the book
\cite{Flandoli:2011vn} and the review paper
\cite{Gess:2018ac}.

Building upon the insightful work \cite{Attanasio:2011fj}, we will
develop a Holmgren-type duality approach
to obtaining an energy estimate (a uniform
bound on the $L^2$ norm) for the numerical
approximations. To motivate the analysis
in the upcoming sections, we briefly outline
the continuous duality approach in $\R^d$, which
eventually leads to the conclusion that
$u$ is bounded in $L^2_{\omega,t,x}$.
Multiplying \eqref{eq:stratonovich_eq} by $u$ and
taking the expectation on both sides yields
\begin{align*}
	\partial_t E = - V\cdot \nabla E
	+ \frac12  \sigma \nabla \cdot \bk{\sigma  \nabla E}
\end{align*}
where we have defined the mean energy $E\coloneqq \Ex\bigl(u^2\bigr)$.
Multiplying by a test function $\phi\in C_c^\infty\bigl([0,T]\times\R^d\bigr)$
and integrating over $(t,x)$ yields
\begin{align*}
	\int_0^T \int_{\R^d} \varphi \partial_t E \,\d x \,\d t
	= - \int_0^T \varphi V \cdot \nabla E\,\d x \,\d t
	+ \frac12 \int_0^T \int_{\R^d} \varphi \sigma \nabla
	\cdot \bk{\sigma \nabla E}\,\d x \,\d t.
\end{align*}
By performing integration by parts and conducting
some elementary manipulations, we obtain
\begin{equation}\label{eq:energycomputation}
\begin{split}
	\int_{\R^d} {\varphi(T,x)} E(T,x) \,\d x
	&=  \int_{\R^d} {\varphi(0,x)} E(0,x) \,\d x
	+\int_0^T\int_{\R^d} V \cdot \nabla \varphi \, E\,\d x \,\d t \\
	& \relspace
	+ \int_0^T \int_{\R^d} \biggl(
	\partial_t\varphi + \Div V\,\varphi
	+ \frac12 \sigma^2 \Delta \varphi \bigg) E\,\d x \,\d t\\
	&  \relspace
	+ \int_0^T \int_{\R^d} \biggl( \frac14 \Delta \sigma^2 \varphi
	+ \frac34 \nabla\sigma^2 \cdot
	\nabla \varphi \bigg) E\,\d x \,\d t.
\end{split}
\end{equation}
Let now $\phi$ solve the ``dual'' problem
\begin{align}\label{eq:dual_motivation}
\begin{cases}\displaystyle
	\partial_t\varphi + \frac12 \sigma^2\Delta \varphi
	+ \Div V\,\varphi= 0
	& \text{in } (0,T)\times\R^d\\
	\phi(T)=\psi
\end{cases}
\end{align}
for an arbitrary $\psi\in L^\infty(\R^d)$.
From \eqref{eq:energycomputation} we then get
\begin{equation}\label{eq:energygronwall_cont}
\begin{aligned}
	\int_{\R^d} \psi(x) E(T,x) \,\d x
	&\lesssim \|\phi(0)\|_{L^\infty_x}\int_{\R^d}E(0)\,\d x \\
	&\relspace + \norm{\Delta \sigma^2}_{L^\infty_x}
	\norm{\varphi}_{L^\infty_{t,x}}
	 \int_0^T\int_{\R^d} E(t,x)\,\d x\,\d t \\
	&\relspace + \bigl\|V + \nabla \sigma^2\bigr\|_{L^\infty_x}
	\int_0^T  \|\nabla\phi\|_{L^\infty_{x}} \int_{\R^d}E(t,x)\,\d x\,\d t.
\end{aligned}
\end{equation}
Taking the supremum over $\psi$ with
$\|\psi\|_{L^\infty}\leq 1$ and applying
Gronwall's inequality now yields the
sought-after $L^2$ bound on $u$, \emph{if}
we can bound $\phi$ and $\nabla\phi$ in $L^\infty$,
uniformly in $\psi$. By a careful analysis
of the fundamental solution of the above
dual problem, this is indeed possible as
long as $\Div V\in L^p(\R^d)$ for
some $p>d$ (cf.~Assumption~\ref{ass:main_assumption}).
These estimates will be derived
within the discrete context in Section
\ref{sec:greens_functions}, while the
discrete analogue of the motivated duality
argument will be presented in Section \ref{sec:duality}.

The duality approach will allow us to obtain an
$L^2$-estimate for our scheme that is
robust in the face of the assumption $Z\coloneqq\Div V\in L^p$
with $p<\infty$. Returning to the one-dimensional
case $d=1$, and given two real-valued grid functions
$c_i\approx c(x_i)$
and $Z_i\approx Z(x_i)$ (for ${i\in \Z}$),
consider the semi-discrete equation
\begin{equation}\label{eq:dual_discrete-intro}
	\frac{\d}{\d t}  \varphi_i + \frac12
	c_i \nabla_+\nabla_-  \varphi_i
	= Z_i \varphi_i, \qquad  \varphi_i(T) = 1
\end{equation}
for $t<T$ and $i\in \Z$.
This is a difference scheme for
the backward parabolic equation $\partial_t \varphi
+\frac12c(x) \Delta \varphi=Z(x)\varphi$, with
a variable diffusion coefficient $c(x)>0$
and lower order forcing term $Z\varphi$.
Importantly, it is well-established that parabolic
equations of this nature exhibit $L^\infty$ gradient
bounds in terms of the $L^p$ norm of $Z$, as long
as the exponent $p$ exceeds the spatial dimension
$d$. In simpler terms, the forcing term $Z$ does not
have to belong to $L^\infty$ for the gradient of the
solution $\varphi$ to be in $L^\infty$. In fact, more
general (Weigant--Kazhikov type) estimates are
available for a larger class of linear parabolic
equations \cite[Appendix F]{Lions:1998ga}.

A  notable portion of this paper, specifically
Section \ref{sec:greens_functions}, is
dedicated to an in-depth analysis of (the
multi-dimensional version of) the discrete
dual problem \eqref{eq:dual_discrete-intro}.
Our main focus is on the development of a
\textit{discrete parametrix} approach for
solving semi-discrete heat equations with
variable coefficients. This material holds value in
its own right and could potentially be applied in other contexts.
This analysis will culminate in a solution formula
for \eqref{eq:dual_discrete-intro}, empowering
us to infer the following critical estimate:
\begin{equation}\label{eq:dual_discrete2-intro}
	\sup_{(t,i)}
	\abs{\frac{\varphi_{i+1}(t)-\varphi_i(t)}{\Dx}}
	\lesssim 1+\mathopen\|\Div V\mathclose\|_{L^p},
\end{equation}
provided $p>d$. Employing the
duality argument mentioned before in conjunction with the structural
property \eqref{eq:scheme2-intro} and the
bound \eqref{eq:dual_discrete2-intro}, we are able to
successfully establish the paper's primary ($L^2$ stability) conclusion:
\begin{equation}\label{eq:main-intro}
	\Ex\sum_i \abs{u_i(t)}^2\leq C \sum_i \abs{u_i(0)}^2,
	\quad t\in [0,T],
\end{equation}
where $C>0$ depends on
$\left\|\Div V\right\|_{L^p}$---but \textit{not} $\left\|\Div V\right\|_{L^\infty}$.

This paper---and the $L^2$ estimate
\eqref{eq:main-intro}---represents
one of the first attempts to explore
noise-based regularisation to help
construct consistent and stable (thus convergent)
numerical schemes for stochastic transport equations.

In the subsequent sections, the scheme and its analysis will
be expounded within the context of Cartesian grids in
arbitrary spatial dimensions. To maintain the technical 
complexity at a digestible level, our choice has been to
discretise solely in the space variable. We anticipate
that our reasoning remains valid when applying a suitable
time-discretisation method, such as an implicit one.
While this paper primarily focuses on the stochastic
transport equation \eqref{eq:stratonovich_eq},
the underlying principles can be extended to apply
to the stochastic continuity equation, in which
$\Div (Vu)$ is substituted in for the term
$V \cdot \nabla  u$ in \eqref{eq:stratonovich_eq}.
Additionally, we can also accommodate a time-dependent
velocity $V=V(t,x)$.  At the end of Section \ref{sec:duality},
we explore certain design principles that can be employed
to cultivate alternative discretisation schemes that
showcase $L^2$ stability.

\subsection{Structure of the paper}
Section \ref{sec:motivation_paper} provides the
definition of a weak $L^2$ solution, which
serves as the target for convergence of our
numerical scheme. We also prove uniqueness
of these solutions; existence of solutions
will be a consequence of the convergence
of our numerical method. In Section \ref{sec:nummethod},
we give a detailed exposition
of the difference scheme and state our
main result (Theorem \ref{thm:main_conv}).
Section \ref{sec:greens_functions} offers a
detailed analysis of a class of parabolic
difference schemes. In Section \ref{sec:duality},
we employ the duality approach and build
upon the findings in Section \ref{sec:greens_functions}
to establish the $L^2$ stability of the difference
scheme. In Section \ref{sec:convergence_of_scheme},
we demonstrate the convergence of the
difference scheme, thereby concluding the
proof of the main result presented in Section
\ref{sec:nummethod}. Finally, in the appendix,
we provide some technical calculations
utilised elsewhere in the paper.

\section{Stochastic transport equation}\label{sec:motivation_paper}
In this section we introduce a notion of weak $L^2$ solution
for \eqref{eq:stratonovich_eq}. Throughout the paper, we will
will work on a fixed filtered probability space, which consists
of a complete probability space $(\Omega, \mathcal{F}, \mathbb{P})$
and a complete right-continuous filtration $(\mathcal{F}_{t})_{t \ge 0}$.
The solution $u$ and the $d$-tuple $W=(W_1,\ldots,W_d)$
of independent Wiener processes will be
defined on the probability space and be
adapted to the filtration $(\mathcal{F}_t)_{t\ge 0}$.
For background information on stochastic analysis and SPDEs
such as \eqref{eq:stratonovich_eq}
(with Lipschitz coefficients), we refer the reader
to \cite{Kunita:1990aa} and \cite{Chow:2015aa}.

\subsection{Weak solutions}
Here and elsewhere, we work under
Assumption \ref{ass:main_assumption}.
The It\^o formulation of \eqref{eq:stratonovich_eq} is
\begin{align}\label{eq:ito_eq}
	0 = \d u + V \cdot \nabla u\,\d t
	+ \sigma \nabla u\, \cdot \d W
	- \frac12 \sigma \nabla \cdot
	\bk{\sigma \nabla u} \,\d t
\end{align}
This equation will be used to define weak solutions.
(We show in Lemma \ref{lem:ito-to-stratonovich}
that \eqref{eq:stratonovich_eq} and \eqref{eq:ito_eq}
are indeed equivalent.)

\begin{definition}[Weak solution]
\label{def:wk_sol}
Let $u^0 \in L^2(\R^d)$. A random field
$u$ is a weak $L^2$ solution of
\eqref{eq:stratonovich_eq} with initial data $u^0$ if
\begin{itemize}
	\item[(i)] $u \in L^2(\Omega \times [0,T] \times \R^d)$,

	\item[(ii)] there is a version $\bar{u}$ of $u$ such that
	for every $\varphi \in C^\infty_c(\R^d)$,
	the $\R^d$-valued process
	$t \mapsto \int_{\R^d} \bar{u}(t,x) \varphi(x)\,\d x$
	is a.s. a continuous, adapted semi-martingale, 

	\item[(iii)] for every $\varphi \in C^\infty_c(\R^d)$,
	and all $t \in (0,T]$,
	\begin{align*}
		0 & =  \int_{\R^d} \varphi \bar{u}(t) \,\d x
		- \int_{\R^d} \varphi u^0\,\d x\\*
		&\relspace
		{-}  \int_0^t \int_{\R^d} \nabla \cdot
		\bk{\varphi V} \bar{u} \,\d x \,\d s
		{-}  \int_0^t \int_{\R^d}  \bar{u}
		\nabla \bk{\sigma \varphi} \, \d x \cdot  \,\d W \\*
		&\relspace
		{-} \frac12 \int_0^t \int_{\R^d} \bar{u}\nabla
		\cdot \bk{\sigma \nabla\bk{\sigma \varphi}} \,\d x \,\d s,
	\end{align*}
	$\mathbb{P}$-almost surely.
\end{itemize}
\end{definition}

``Weak" in our context refers to weak solutions 
in the PDE sense --- our solutions are probabilistically strong.
The Cauchy problem for \eqref{eq:stratonovich_eq} has a unique
weak solution per Definition \ref{def:wk_sol} in the class of 
solution adapted to the natural filtration of the Wiener process $W$; see
Theorem \ref{thm:main_wp} and Remark \ref{rem:uniqueness1}.

\subsection{Well-posedness of the stochastic transport equation}
\label{sec:well-posedness}
The existence and uniqueness of weak
solutions for \eqref{eq:stratonovich_eq}
(with $\sigma \equiv 1$) have been studied
extensively under different conditions on $V$,
with emphasis on how these conditions
differ from those required in the deterministic
setting ($\sigma \equiv 0$). We briefly mention
that outside the $\sigma \equiv 1$ context
(or where $\sigma$ is bounded and uniformly positive),
well-posedness has also been extensively
studied (see, e.g., \cite{zha2010} and references contained
there). Results for possibly degenerate $\sigma$
do not see any noise induced regularisation, however,
and typically require ${\rm div} V \in L^\infty$, so
we focus the remainder of our discussion in this
subsection to the non-degenerate $\sigma$ case.
Since an existence
result for \eqref{eq:stratonovich_eq}
can be derived from the convergence of
our proposed scheme (in Section \ref{sec:convergence_of_scheme}),
we concentrate on the uniqueness aspect of the theory in this subsection.
Often $V$ is allowed to depend on time, but as
there are no temporal derivatives on $V$, only
very mild integrability conditions are
demanded in $t$.
As $V$ is time-independent for us,
we review major results on uniqueness
by stating requisite conditions in terms of
the $x$-dependence of $V$ only.

For deterministic transport equations, the
foundational work \cite{DiPerna:1989aa}
showed that $V\in W^{1,p'}_\loc(\R^d;\R^d)$,
$\Div V \in L^\infty(\R^d)$,
$V/\bk{1 + \abs{x}}\in L^1(\R^d;\R^d) + L^\infty(\R^d;\R^d)$
are sufficient conditions for the uniqueness
of weak solutions in $L^\infty([0,T];L^p(\R^d))$ with initial
conditions $u^0\in L^p(\R^d)$, for some $p \in [1,\infty]$.
This result was extended to $BV$ vector
fields in \cite{Ambrosio:2004aa}.

Inclusion of transport noise
($\sigma \equiv 1$) permits a
loosening of the assumptions. The work \cite{Flandoli:2010yq}
established strong uniqueness under the
requirements that $V$ is H\"older continuous
and lies in $L^p(\R^d;\R^d)$ for $p > 2$. Uniqueness
for weak $L^\infty$ solutions was also
established in \cite{Attanasio:2011fj} for $V \in (L^\infty
\cap BV_\loc)(\R^d;\R^d)$ and $\Div V
\in L^1(\R^d)$ such that $V$ decomposes
into $V_1 + V_2$, each satisfying further
integrability conditions.
Finally, \cite[Theorem 1.2, Section 3]{Beck:2019aa}
established path-by-path uniqueness for
weak $L^p$-solutions using the Fredholm
alternative on the dual equation under Prodi--Serrin
type conditions, which, neglecting the
time dependence of $V$, simplify to
$V \in L^p(\R^d;\R^d)$ and
$\Div V \in L^p(\R^d)$, with $p > d$.

The uniqueness result of \cite{Mau2011}
can be applied to \eqref{eq:stratonovich_eq}
with minimal modifications for
$\sigma \in W^{2,\infty}(\R^d)$. In particular,
projection to the $n$th Wiener chaos commutes with
multiplication by any deterministic function $\sigma$.
In order to adapt the proof presented in \cite[Lemma 3.2]{Mau2011},
it is simply a matter of replacing the heat kernel
for the constant $\sigma$ case with a continuous version
of the fundamental solution $\bm{\Gamma}$
from Section \ref{sec:greens_functions}
for the variable coefficient-case $\sigma=\sigma(x)$.
A suitable reference for this
construction can be found in \cite[Section 9.4]{Fri1964}.
Specifically, we can construct a Green's function
$\Gamma$ for the heat operator
$\varphi \mapsto \partial_t \varphi
-\frac12 \sigma^2(x) \Delta \varphi$
using the parametrix method, and
by \cite[Theorem 9.2, eq.~(9.4.1) and
Theorem 9.7, eq.~(9.6.1)]{Fri1964},
\begin{align}\label{eq:continuous_Gamma}
	\bigl|\nabla^{m+ a}_x \nabla^b_y
	\Gamma(t-s,x,y)\bigr|
	 \le \frac{C}{|t - s|^{\bk{d + |m| + |a| + |b|}/2}}
	 \exp\bk{-c\frac{|x-y|^2}{t - s}},
\end{align}
for constants $C , c > 0$ independent of $x,y,t,s$ and
$0 \le |a| + |b| \le r$, $0 \le |m| \le 2$.
We only need this bound for
$m=0,1$ and $a=b=0$.
Using the bounds above, we obtain
(cf.~Lemmas \ref{lem:full_greensfunction} and
\ref{lem:full_greensfunction2}):
\begin{align*}
	&\|\Gamma(t)\|_{L^\infty(\R^d_x;L^{p'}(\R^d_y))}
	\lesssim t^{-d/2 + d/(2p')} ,
	\quad
	\|\nabla_x \Gamma(t)\|_{L^\infty(\R^d_x;L^{p'}(\R^d_y))}
	\lesssim t^{-(d + 1)/2 + d/(2p')},
\end{align*}
where $p'$ denotes the H\"older conjugate of $p>1$.

Using the existence result
presented in Section \ref{sec:convergence_of_scheme},
and the $L^2$ bound \eqref{eq:L2bound2}
of Theorem \ref{thm:disc_energyestm},
we consolidate the preceding
discussion into a theorem:
\begin{theorem}[Well-posedness]\label{thm:main_wp}
Under Assumption \ref{ass:main_assumption}, 
there exists a weak solution (in the sense of
Definition \ref{def:wk_sol}) to the
Cauchy problem for the stochastic transport equation
\eqref{eq:stratonovich_eq} that is unique in the 
class of solutions adapted to the Brownian filtration
(i.e., Wiener uniqueness holds). Moreover, the solution
satisfies the following a priori estimate:
\begin{align}\label{eq:L2bound2}
	\sup_{t \in [0,T]}\Ex \|u(t)\|_{L^2(\R^d)}^2
	& \lesssim_{\sigma,V,T} \|u^0\|_{L^2(\R^d)}^2.
\end{align}
\end{theorem}
In the foregoing and subsequently, we use
$\lesssim_{\sigma, V, T}$ etc., to indicate that the implicit
constant depends on the subscripted quantities,
in this case $\sigma$, $V$, and $T$.

\begin{remark}\label{rem:uniqueness1}
We are able to handle random initial data  
$u^0 \in L^2(\Omega \times \R^d)$ as well. 
In this case the uniqueness result 
of \cite{Mau2011} cannot be directly applied 
as it concerns only the class of solutions 
adapted to the Wiener filtration.

Suppose, in addition to Assumption \ref{ass:main_assumption}, 
that  $V \in W^{1,2}(\R^d)$. 
With further reference to \eqref{eq:continuous_Gamma} 
and to \cite{Beck:2019aa}, the duality method 
itself can be used to establish pathwise 
uniqueness. The  difference $v = u_1 - u_2$ of two solutions
to \eqref{eq:stratonovich_eq} satisfies \eqref{eq:stratonovich_eq}
with initial condition $v(0) \equiv 0$.
The duality argument leading up to \eqref{eq:energygronwall_cont}
above applied to  $v$ can be used to show that
$\Ex \norm{v}_{L^2_x}^2 = 0$, i.e., pathwise uniqueness.
In the continuous setting, it is necessary to
smoothen $\sigma$ with standard mollifiers. This leads
to DiPerna--Lions-type commutator brackets, as well as
second order commutators, both of which are known to vanish
appropriately under our conditions here (see e.g.~\cite{Punshon-Smith:2018aa,Galimberti:2021aa}).

The strong well-posedness for the transport equation
follows in the same vein as the result of \cite{Galimberti:2024aa}
for the stochastic continuity equation (as sketched out above),
and can be applied here to establish
uniqueness with minimal modifications for $V \in W^{1,2}(\R^d)$ and
$\sigma \in W^{3,\infty}(\R^d)$.
\end{remark}

Finally, we briefly discuss the significance
of requirement (ii) of Definition \ref{def:wk_sol}.
This requirement ensures that
the resulting Stratonovich integral
is well-defined, and the conversion of
$\int_0^t \int_{\R^d} u(s)
\nabla \bk{\sigma \varphi}\,\d x\circ \d W$
into an It\^o integral plus a
cross variation term holds rigorously
 \cite[p.~1460 and Proposition~3]{Attanasio:2011fj}.
We show this for the $L^2$ solutions that we consider here.

\begin{lemma}[It\^o-to-Stratonovich conversion]\label{lem:ito-to-stratonovich}
Let $u$ be a weak solution to the stochastic transport
equation \eqref{eq:stratonovich_eq} in the sense of
Definition \ref{def:wk_sol}, replacing (iii) with:
\begin{itemize}
\item[(iii')] for every $t \in [0,T]$ and every
$\varphi \in C^\infty_c( \R^d)$,
\begin{align*}
	0& =\int_{\R^d} \varphi u(t) \,\d x
	- \int_{\R^d} \varphi u^0\,\d x\\
	&\relspace
	{-}\int_0^t \int_{\R^d} u\nabla \cdot \bk{\varphi V} \,\d x \,\d s
	{-}\int_0^t \int_{\R^d}  u \nabla \bk{\sigma \varphi} \, \d x \circ  \,\d W,
\end{align*}
$\mathbb{P}$-almost surely.
\end{itemize}
Then $u$ is a weak
solution to \eqref{eq:stratonovich_eq} in the
sense of Definition \ref{def:wk_sol}, and vice versa.
\end{lemma}

\begin{proof}
For $X(t) \coloneqq \int_{\R^d} u(t)
\nabla \bk{\sigma \varphi}  \,\d  x$,
since $\sigma \in W^{3,\infty}(\R^d)$ is sufficiently
smooth, the lemma follows from the 
Stratonovich-to-It\^o conversion formula,
\[
\int_0^T X \circ \d W = \int_0^T X\,\d W
+ \frac12 [X,W](T).\qedhere
\]
\end{proof}

\section{The numerical scheme}\label{sec:nummethod}

In this section, we introduce our
numerical scheme. More precisely, we introduce
specific discretisation notations, define
the scheme, and present the main theorem
regarding stability and convergence.

\subsection{The discrete setting}\label{sec:discrete_setting}
We start by introducing notation for
finite difference operators and
discussing key discrete calculus rules,
such as the Leibniz product rule and
the chain rule. These tools will be
instrumental in analysing the scheme.

Let $\Dx>0$ denote the discretisation
parameter. For indices $\alpha\in\Z^d$
we denote by $x_\alpha\coloneqq \alpha\Dx$
the grid points and we partition the
spatial domain $\R^d$ into grid cells
$\cell_\alpha \coloneqq x_\alpha
	+ [-\nicefrac{\Dx}{2}, \nicefrac{\Dx}{2})^d$.
We will denote by $L^\infty_\Dx(\R^d)$
the collection of functions that are bounded and
piecewise constant on each cell $\cell_\alpha$,
and we set $L^p_\Dx(\R^d)
	\coloneqq L^\infty_\Dx(\R^d)\cap L^p(\R^d)$
for any $p\in[1,\infty]$. To distinguish
elements of these discrete spaces from
their continuous counterparts, we will
denote elements of $L^p_\Dx(\R^d)$
by bold symbols. By definition then,
any function $\bm{f}\in L^\infty_\Dx(\R^d)$
can be written as
\begin{equation*}
	\bm{f}(x) = \sum_{\alpha\in\Z^d} f_\alpha
	\one{\cell_\alpha}(x), \qquad x\in\R^d
\end{equation*}
for $f_\alpha \in \R$, and it is readily seen that
\[
\|\bm{f}\|_{L^p(\R^d)} = \begin{cases}
\Bigl(\Dx^d\sum_{\alpha\in\Z^d}
|f_\alpha|^p\Bigr)^{1/p} & p<\infty,\\
\sup_{\alpha\in\Z^d}|f_\alpha| & p=\infty.
\end{cases}
\]
Due to this characterisation, we interchangeably
view elements of $L^p_\Dx(\R^d)$ either
as functions of $x\in\R^d$ or as sequences
parametrised by $\alpha\in\Z^d$. The projection
operator $P_\Dx \from L^p(\R^d)\to L^p_{\Dx}(\R^d)$
is defined by
\begin{equation}\label{eq:projection}
	P_\Dx f = \sum_{\alpha\in\Z^d}f_\alpha \one{\cell_\alpha},
	\qquad \text{where } 
	f_\alpha \coloneqq \fint_{\cell_\alpha} f(x)\,\d x,
\end{equation}
where $\fint_c = \frac{1}{|c|}\int_c$,
and $|c|$ denotes the Lebesgue measure
of a measurable set $c\subset\R^d$.

Vector valued versions $L^p_\Dx(\R^d,\R^d)$
of the aforementioned spaces are defined
similarly; vectors in $\R^d$ will be indexed
by a superscript ${}^j$ for $j=1,\dots,d$.
We let $e_1,\dots,e_d\in\R^d$ denote the
canonical basis, $e_i^j = \delta_{i,j}$.

For functions of several variables, e.g.~$f=f_\alpha(t)$ 
or $g=g_{\alpha,\beta}(t)$ (for $t\in[0,T]$ and 
$\alpha,\beta\in\Z^d$) we shall often abuse 
notation somewhat and write things like
\begin{align*}
\|\bm{f}\|_{L^1_tL^p_\alpha} &= \int_0^T \Bigl(\Dx^d\sum_{\alpha\in\Z^d}
|f_\alpha(t)|^p\Bigr)^{1/p}\,\d t, \\
\|\bm{g}\|_{L^1_tL^p_\alpha L^q_\beta} &= \int_0^T \Bigl(\Dx^d\sum_{\alpha\in\Z^d}
\|g_{\alpha,\cdot}(t)\|_{L^q}^p\Bigr)^{1/p}\,\d t.
\end{align*}
In the same vein, we write e.g.~$L^1_tL^p_\alpha$ or $L^1_tL^p_\alpha L^q_\beta$ for the spaces of discrete functions whose corresponding norms are bounded. (The time interval $[0,T]$ on which the functions are defined are implicitly assumed to be fixed, although we will write e.g.~$L^1_t([0,T])$ in place of $L^1_t$ whenever necessary.)

For discrete functions of two variables $F=F_{\alpha,\beta}$, $G=G_{\alpha,\beta}$ we define the discrete convolution
\begin{equation}\label{eq:convolution-def}
    \bigl(\bm{F}\dconv\bm{G}\bigr)_{\alpha,\beta} \coloneqq \sum_{\eta\in\Z^d} F_{\alpha,\eta}G_{\eta,\beta}\,\Dx^d.
\end{equation}
Similarly, for time-dependent functions of two variables we define
\begin{equation}\label{eq:convolution-time-def}
    \bigl(\bm{F}\dconv\bm{G}\bigr)_{\alpha,\beta}(t) \coloneqq \int_0^t \sum_{\eta\in\Z^d} F_{\alpha,\eta}(t-s)G_{\eta,\beta}(s)\,\Dx^d\,\d s.
\end{equation}
The more familiar variants for functions of one spatial variable, $f=f_\alpha$ or $f=f_\alpha(t)$, are obtained by setting $F_{\alpha,\beta} =f_{\alpha-\beta}$ and $G_{\alpha,\beta}=g_\alpha$, leading to
\begin{equation}\label{eq:convolution-def-univariate}
\begin{split}
(\bm{f}\dconv\bm{g})_\alpha &= \sum_{\eta\in\Z^d} f_{\alpha-\eta}g_\eta\,\Dx^d, \\
(\bm{f}\dconv\bm{g})_\alpha(t) &= \int_0^t\sum_{\eta\in\Z^d} f_{\alpha-\eta}(t-s)g_\eta(s)\,\Dx^d\,\d s.
\end{split}
\end{equation}
We shall need the following variant of Young's convolution inequality.
\begin{lemma}\label{lem:young}
Let $p_1,p_2,p_3,q_1,q_2,q_3\in[1,\infty]$ satisfy $\frac{1}{p_3}+\frac{1}{q_2} = 1$.
\begin{enumerate}[label=(\roman*)]
\item If $\bm{F}\in L^{p_2}_\alpha L^{p_3}_\beta$ 
and $\bm{G}\in L^{q_2}_\alpha L^{q_3}_\beta$ then
\[
\|\bm{F}\dconv\bm{G}\|_{L^{p_2}_\alpha L^{q_3}_\beta} 
	\leq \|\bm{F}\|_{L^{p_2}_\alpha L^{p_3}_\beta} 
	\|\bm{G}\|_{L^{q_2}_\alpha L^{q_3}_\beta}.
\]
\item If $\bm{F}\in L^{p_1}_t L^{p_2}_\alpha L^{p_3}_\beta$ 
and $\bm{G} \in L^{q_1}_t L^{q_2}_\alpha L^{q_3}_\beta$ 
for $\frac{1}{p_1}+\frac{1}{q_1}=1+\frac{1}{r_1},$ then
\[
\|\bm{F}\dconv\bm{G}\|_{L^{r_1}_tL^{p_2}_\alpha L^{q_3}_\beta} 
	\leq \|\bm{F}\|_{L^{p_1}_t L^{p_2}_\alpha L^{p_3}_\beta} 
	\|\bm{G}\|_{L^{q_1}_t L^{q_2}_\alpha L^{q_3}_\beta}.
\]
\end{enumerate}
\end{lemma}
The proof is given in the Appendix. For discrete functions 
of one or two variables we have similarly, for example,
\begin{equation}\label{eq:younginequality}
\|\bm{f}\dconv\bm{g}\|_{L^\infty_t L^\infty_\alpha} 
	\leq \|\bm{f}\|_{L^{p_1}_tL^{p_2}_\alpha}
	\|\bm{g}\|_{L^{p_1'}_tL^{p_2'}_\alpha},
\end{equation}
which follows from Lemma \ref{lem:young} by setting 
$F_{\alpha,\beta}(t)=f_{\alpha-\beta}(t)$, 
$G_{\alpha,\beta}(t)=g_\alpha(t)$, and 
$r_1=p_2=q_3=\infty$.

For $\bm{f}\in L^\infty_\Dx(\R^d)$ we define the
(vector-valued) difference operators
\begin{equation}\label{eq:difference_operators}
\begin{gathered}
	\nabla_+^j f_\alpha \coloneqq
	\frac{f_{\alpha + e_j} - f_{\alpha}}{\Delta x},
	\qquad
	\nabla_-^j f_\alpha \coloneqq
	\frac{f_{\alpha} - f_{\alpha-e_j}}{\Delta x}, \\
	\nabla_0^j f_\alpha
	\coloneqq \frac{f_{\alpha + e_j} - f_{\alpha - e_j}}{2\Delta x}
\end{gathered}
\end{equation}
for $\alpha\in\Z^d$ and $j=1,\dots,d$. Throughout this
paper we will abuse notation slightly and write $(\nabla_+ f)_\alpha
= \nabla_+ f_\alpha$, $(\nabla_- f)_\alpha = \nabla_- f_\alpha$ and so on.
When applying difference operators to ``multivariate'' 
maps $(\alpha,\beta)\mapsto f_{\alpha,\beta}\in\R$, 
we will always consider $\beta$ as a fixed
parameter, letting the difference operators
only act on $\alpha$ --- unless otherwise stated.
Given a vector-valued function 
$\bm{V}\in L^\infty_\Dx(\R^d,\R^d)$ we
define the upwind difference operator $D_V$ by
\begin{equation}\label{eq:upwind-scheme}
\begin{split}
	D_V f_\alpha
 	&\coloneqq V_\alpha^+ \cdot \nabla_- f_\alpha
	-V_\alpha^-\cdot \nabla_+ f_\alpha \\
	& = \frac1{\Dx}\sum_{j=1}^d
	\bk{V_\alpha^{j,+} \bk{f_{\alpha} - f_{\alpha - e_j}}
	- V_\alpha^{j,-} \bk{f_{\alpha + e_j} - f_{\alpha}}},
\end{split}
\end{equation}
where $V_\alpha^{j,\pm}\coloneqq \max(\pm V_\alpha^j,0)$.
By extension, we also define the ``dual'' upwind
difference operator
\begin{equation}\label{eq:dual_upwind}
\begin{split}
	D_V'g_\alpha &\coloneqq \nabla_+\cdot (V^+g)_\alpha
	- \nabla_-\cdot(V^-g)_\alpha \\*
	&= \frac{1}{\Dx}\sum_{j=1}^d \bigl(V_{\alpha+e_j}^{j,+}g_{\alpha+e_j}
	- V_\alpha^{j,+}g_\alpha\bigr) - \bigl(V_{\alpha}^{j,-}g_{\alpha}
	- V_{\alpha-e_j}^{j,-}g_{\alpha-e_j}\bigr).
\end{split}
\end{equation}
By a slight abuse of notation we will write
\begin{equation}\label{eq:dual_upwind2}
\begin{split}
	D_V'(1)_\alpha	&= D_V'g_\alpha\bigr|_{g\equiv 1}
	= \frac1{\Dx} \sum_{j=1}^d \bigl(V^{j, + }_{\alpha + e_j}
	- V^{j,+}_\alpha\bigr)
	- \bigl(V^{j, - }_{\alpha} -  V^{j,-}_{\alpha - e_j}\bigr).
\end{split}
\end{equation}

Finally, we introduce the averaging operator
\begin{align}\label{eq:box_defin}
	\square_\pm^j\sigma_\alpha
	= \frac{\sigma_{\alpha\pm e_j}+\sigma_\alpha}{2}.
\end{align}

The following discrete rules of
calculus will be used heavily in this work.

\begin{lemma}[Discrete calculus]\label{lem:discrete_calculus}
For any scalar-valued grid functions $\bm{f},
\bm{g}$ and vector-valued grid function
$\bm{V}$, we have the following rules of calculus:

\begin{enumerate}[label=(\roman*)]
	\item\label{lem:discrete_calculus_product-rule}
	Product rule:	
	\[
	\nabla_\pm^j (f_\alpha g_\alpha)
	=  f_{\alpha \pm e_j} \nabla_\pm^j g_\alpha
	+ g_\alpha \nabla_\pm^j f_\alpha.
	\]

	\item\label{lem:discrete_calculus_upwind-product-rule}
	Product rule for upwind differencing:
	\[
	D_V' g_\alpha = g_\alpha D_V'(1)_\alpha
	+ \frac1{\Dx}\sum_{j=1}^d \bigl(g_{\alpha + e_j} - g_\alpha\bigr)
	V^{j, + }_{\alpha + e_j} - \bigl(g_\alpha - g_{\alpha - e_j}\bigr)
 	V^{j, - }_{\alpha - e_j}
	\]

	\item\label{lem:discrete_calculus_differentiate-convolution}
	Differences of convolutions:
	\[
	\nabla_\pm ({\bm f}\dconv{\bm g}) = (\nabla_\pm {\bm f})\dconv{\bm g}
	= {\bm f}\dconv(\nabla_\pm {\bm g}).
	\]

	\item\label{lem:discrete_calculus_int-by-parts}
	Summation-by-parts:
	\[
	\sum_{\alpha\in\Z^d} f_\alpha \nabla_\pm g_\alpha
	= -\sum_{\alpha\in\Z^d} g_\alpha \nabla_\mp f_\alpha.
	\]

	\item\label{lem:discrete_calculus_upwind-int-by-parts}
	Summation-by-parts for upwind differencing:
	\[
	\sum_{\alpha\in\Z^d} g_\alpha
	 D_V f_\alpha
	= - \sum_{\alpha\in\Z^d} f_\alpha
	D_V'g_\alpha.
	\]

	\item \label{lem:discrete_calculus_chain-rules}
	Chain rules: for any $\beta\in C^2(\R)$ and $\alpha\in \Z^d$,
	\begin{align*}
		& \nabla_+\beta(f_\alpha) = \beta'(f_\alpha)\nabla_+ f_\alpha
		+\frac{\Delta x}{2} \beta''(\xi_\alpha^+)\abs{\nabla_+ f_\alpha}^2,
		\\ &
		\nabla_- \beta(f_\alpha) = \beta'(f_\alpha)\nabla_- f_\alpha
		- \frac{\Delta x}{2} \beta''(\xi_\alpha^-)\abs{\nabla_- f_\alpha}^2,
	\end{align*}
	and
	\begin{equation}\label{eq:2nd-order-commutator-beta}
	\begin{split}
		&\beta'(f_\alpha)\nabla_-\nabla_+ f_\alpha
		-\nabla_-\nabla_+ \beta(f_\alpha)
		\\ & \qquad
		=-\frac{1}{2} \beta''(\xi_\alpha^+)\abs{\nabla_+ f_\alpha}^2
		-\frac{1}{2} \beta''(\xi_\alpha^-)\abs{\nabla_- f_\alpha}^2,
	\end{split}
	\end{equation}
	for some $\xi_\alpha^\pm \in [f_{\alpha }\wedge f_{\alpha \pm e_j},
	f_{\alpha} \vee f_{\alpha \pm e_j}]$.

	\item\label{lem:discrete_calculus_upwind-chain-rule}
	Chain rule for upwind differencing: for any
	$\beta\in C^2(\R)$ and $\alpha\in \Z^d$,
	\begin{align*}
		&\beta'(f_\alpha) D_V f_\alpha
		- D_V \beta(f_\alpha)\\
		& =\frac1{2\Dx}\Bigl( V_\alpha^{j, + }
		{\beta''(\xi_\alpha^-) \bk{f_{\alpha - e_j} - f_\alpha}^2}
		+V_\alpha^{j, - }
		{\beta''(\xi_\alpha^+)\bk{f_{\alpha + e_j} - f_\alpha}^2}\Bigr),
	\end{align*}
	for some $\xi_\alpha^\pm \in [f_{\alpha \pm e_j}  \wedge f_{\alpha} ,
	f_{\alpha \pm e_j}  \vee f_{\alpha}]$.
\end{enumerate}
\end{lemma}
The proofs are straightforward and
we relegate them to Appendix \ref{sec:disccalc}.

\subsection{Difference scheme and main result}\label{sec:discretisation}
With the finite difference notation in place,
we can now define our difference scheme. The
velocity field $V$, the noise function $\sigma$,
and the initial function $u^0$  are discretised by their cell averages,
\begin{align}\label{eq:avg_defin}
	V_\alpha \coloneqq
	\fint_{\cell_\alpha} V(x)\,\d x,
	\quad
	\sigma_\alpha \coloneqq
	\fint_{\cell_\alpha} \sigma(x) \,\d x,
	\quad
	u^0_\alpha \coloneqq
	\fint_{\cell_\alpha} u^0(x) \,\d x,
	\quad \alpha\in\Z^d.
\end{align}
We will study the following finite
difference scheme:
\begin{subequations}\label{eq:scheme}
\begin{equation}\label{eq:scheme1}
\begin{split}
	0 &= \d u_\alpha + D_V u_\alpha\,\d t \\
	&\relspace + S_\alpha\nabla_0u_\alpha \cdot \d W
	-\frac12 \sigma_\alpha \sum_{j=1}^d \nabla_+^j
	\bigl(\average_-^j\sigma_\alpha\nabla_-^j u_\alpha\bigr)\,\d t,
\end{split}
\end{equation}
for $(t,\alpha)\in (0,T]\times \Z^d$, with initial values
$u_\alpha(0)=u^0_\alpha$, $\alpha\in \Z^d$.
In \eqref{eq:scheme1}, $D_V$ is the upwind difference \eqref{eq:upwind-scheme}, and
\begin{equation}\label{eq:scheme2}
	S_\alpha \coloneqq \diag\Biggl(\biggl(\sigma_\alpha
	\frac{\average^j_+\sigma_\alpha\average^j_-\sigma_\alpha}
	{(\average^j_+\sigma_\alpha+\average^j_-\sigma_\alpha)/2}
	\biggr)^\hf\Biggr)_{j=1,\dots,d}
\end{equation}
\end{subequations}
where $\square_\pm^j\sigma_\alpha$ is
defined in \eqref{eq:box_defin}. Since
$\sigma$ is bounded and uniformly positive,
the quantity in the square-root in $S_\alpha$ is bounded
and strictly positive. The motivation
behind the specific form of this difference
scheme can be found in the introduction.

Because of its widespread use, we have opted to use the upwind difference operator $D_V u$ instead of the somewhat simpler Lax--Friedrichs operator
\[
D_V^{\rm{LxF}}u_\alpha = V_\alpha \nabla_0 u_\alpha - \frac{\mu\Dx}{2}\nabla_+\cdot\nabla_- u_\alpha
\]
(for some $\mu \geq \|V\|_{L^\infty}$). The analysis for this alternative scheme is almost identical to that of our scheme. See Section \ref{sec:alternative-discretisation} for more on alternative discretisations.

In Section \ref{sec:duality}, we shall
demonstrate that the semi-discrete difference
scheme \eqref{eq:scheme}
is well-defined by establishing the
existence of a unique global solution
per the following definition.

\begin{definition}[Solution to the difference scheme]
\label{def:discrete_soln}
Let $\bm{u}^0 \in L^2_\Dx(\R^d)$.
We say that $\bm{u}\in L^2(\Omega\times[0,T];L^2_\Dx(\R^d))$
is a strong solution to the SDE system \eqref{eq:scheme}
on the filtered probability space
$(\Omega, \mathcal{F}, (\mathcal{F}_t)_{t \in [0,T]}, \mathbb{P}, W)$ if
\begin{itemize}
	\item[(i)] for every $\alpha \in \Z^d$,
	$u_\alpha (0) = u^0_\alpha$ a.s.,

	\item[(ii)] each $u_\alpha$ is adapted
	to the filtration $(\mathcal{F}_t)_{t \in [0,T]}$,

	\item[(iii)] for every $\alpha \in \Z^d$ and $t \in [0,T]$,
	the following holds a.s.:
	\begin{align*}
		u_\alpha(t) ={}& u_\alpha^0
		- \int_0^t  D_V u_\alpha\,\d s \\*
		&{-} \int_0^t S_\alpha\nabla_0u_\alpha \cdot \d W
		{+}\frac12 \int_0^t \sigma_\alpha \sum_{j=1}^d \nabla_+^j
		\bigl(\average_-^j\sigma_\alpha\nabla_-^j u_\alpha\bigr)\,\d s.
	\end{align*}
\end{itemize}
\end{definition}

Our main result is:

\begin{theorem}[Main theorem]\label{thm:main_conv}
Under Assumption \ref{ass:main_assumption}, the following holds.
\begin{itemize}
	\item[(i)]There exists a unique strong solution
	$\bm{u}$ to the
	semi-discrete system \eqref{eq:scheme} satisfying
	\begin{align*}
 		\Ex\sum_{\alpha \in \Z^d} \abs{u_\alpha(t)}^2 \Dx^d
		&\leq C\sum_{\alpha \in \Z^d}
		 |u^0_\alpha|^2\,\Dx^d, \qquad t \in [0,T].
	\end{align*}

	\item[(ii)] The approximations
	$u_{\Dx} = \sum_{\alpha \in \Z^d}
	u_\alpha \one{\cell_\alpha}$ generated by the
	scheme \eqref{eq:scheme} converge
	weakly in $L^2(\Omega\times[0,T]\times\R^d)$
	as $\Dx\to0$ to the unique weak solution
	of \eqref{eq:stratonovich_eq}
	(in the sense of Definition \ref{def:wk_sol}).
\end{itemize}
\end{theorem}

This theorem stems directly from the
findings presented in Sections \ref{sec:duality}
and \ref{sec:convergence_of_scheme}.

\section{Discrete Green's function estimates}
\label{sec:greens_functions}

In this section, we provide a detailed
derivation of estimates on the Green's function
(fundamental solution) of semi-discrete parabolic
operators with constant and variable
coefficients. We work out the specifics
of a discrete parametrix method, which might be
of independent interest.

\subsection{The constant coefficient case}\label{sec:constant_coeff}
First, we consider the following semi-discrete
heat equation with constant coefficients:
\begin{equation}\label{eq:discrete_heat_const}
	\frac{\d}{\d t} \varphi_\alpha
	= \nabla_+ \cdot \bigl(c\nabla_- \varphi_\alpha\bigr),
\end{equation}
where $c=\diag(c^1,\dots,c^d)$ for constants
$c^1,\dots,c^d > 0$ and
\begin{equation}\label{eq:def_discrete_laplacian}
	\nabla_+\cdot \bigl(c\nabla_- \phi_\alpha\bigr)
	=\sum_{j=1}^d c^j\frac{\phi_{\alpha + e_j} - 2
	\phi_\alpha + \phi_{\alpha - e_j}}{\Dx^2}
\end{equation}
is a discrete anisotropic Laplacian of $\phi$.
We will realise the solution of \eqref{eq:discrete_heat_const}
as $\bm{\phi}(t)=\bm{\phi}(0)\dconv\bm{a}(t)$ for a fundamental
solution $\bm{a}$. Letting $\bm{\delta}\in L^1_\Dx(\R^d)$
denote the discrete Dirac measure
\[
\delta_\alpha
\coloneqq
\begin{cases}\displaystyle
	\Dx^{-d} & \alpha = 0,
	\\ 0 & \text{elsewhere},
\end{cases}
\]
we define the fundamental solution $\bm{a}$ by
\begin{equation}\label{eq:greensfunction}
	a_\alpha(t) = \sum_{n \ge 0}
	\frac{t^n}{n!} \bigl(\nabla_+ \cdot c\nabla_-\bigr)^n{\delta}_\alpha.
\end{equation}
Then---formally speaking for now---$\bm{a}$
satisfies
\begin{align}\label{eq:discheateq}
	\frac{\d }{\d t} a_\alpha(t)
	= \nabla_+ \cdot \bigl(c\nabla_- a_\alpha(t)\bigr),
	\qquad a_\alpha(0) = {\delta}_\alpha,
	\qquad \alpha\in\Z^d,\,t\in\R.
\end{align}
In the following lemma we shall need the 
modified Bessel function of the first kind 
$I=I_n(r)$, defined as
\begin{equation}\label{eq:besselfunc}
	I_n(r) \coloneqq \frac{1}{\pi}\int_0^\pi
	e^{r\cos(z)} \cos(nz)\,\d z
	= \bigl(\tfrac{r}{2}\bigr)^{|n|}\sum_{k=0}^\infty
	\frac{\bigl(\tfrac{r^2}{4}\bigr)^k}{k!(|n|+k)!}
\end{equation}
for $r\geq0$ and $n\in\Z$; see \cite[equation~(9.5.2)]{BW2016}.

\begin{lemma}[Properties of the fundamental solution]
\label{lem:green_estimates}
The series \eqref{eq:greensfunction} converges
for all $t\in \R$ and the limit $\bm{a}$
lies in $C([0,T];L^1_\Dx(\R^d))$ for any $T >0$. Moreover,
\begin{enumerate}[label=(\roman*)]
	\item $a_\alpha(\cdot)$ lies in $C^1((0,\infty))$, satisfies \eqref{eq:discheateq} for all $\alpha\in\Z^d$,
    and has the representation
     \begin{equation}\label{eq:fundamentalsolution}
    	a_\alpha(t) = \frac{1}{\Dx^d}
    	\prod_{j=1}^d e^{-r^j} I_{\alpha^j}(r^j).
    \end{equation}

	\item\label{lem:green_estimates_positivity}
	$a_{\alpha}(t) > 0$ for all $\alpha\in\Z^d$ and $t>0$.

	\item\label{lem:green_estimates_energy-bound}
	$\|\bm{a}(t)\|_{L^2(\R^d)}
	\leq \|\bm{a}(0)\|_{L^2(\R^d)} < \infty$ for all $t>0$.

	\item\label{lem:green_estimates_unit-mass}
	$\sum_{\alpha \in \Z^d} a_\alpha(t) \Dx^d = 1$ for all $t\geq0$.

	\item\label{lem:green_estimates_derivative-bound}
	For $\ell=1,\dots,d$ and $m = 0,1,2$, recall the definition
	of $\nabla_+^\ell$ in \eqref{eq:difference_operators}, and
	let $(\nabla_+^\ell)^m$ denote $m$
	successive applications of $\nabla_+^\ell$.
    Then
\begin{equation}\label{eq:pointwise_da}
	\begin{aligned}
	\abs{(\nabla_+^\ell)^m a_\alpha(t)}
	&\lesssim t^{-(d + m)/2} \prod_{j = 1}^d
		\biggl(1 + \frac{|x_\alpha^j|^2}{4 \overline{c}t}
	+ \frac{|x_\alpha^j|^3}{|4 \overline{c}t|^\thf}\biggr)^{-1},
	\end{aligned}
	\end{equation}
	 where $\overline{c}\coloneqq \max(c^1,\dots,c^d)$.
	In particular, for $p>1$,
	\begin{equation}\label{eq:a_lp}
		\bigl\|(\nabla_+^\ell)^m\bm{a}(t)\bigr\|_{L^{p'} (\R^d)}
		\leq C_{\ell,m,p} t^{-d/2p - m/2}.
\end{equation}
\end{enumerate}
\end{lemma}

\begin{remark}
Comparing $\bm{a}$ to the fundamental solution 
of the heat equation $\xi(x,t)= (4\pi t)^{-\hf}e^{-|x|^2/4t}$, 
properties (ii)--(iv) and \eqref{eq:a_lp} are very natural. 
The pointwise estimate \eqref{eq:pointwise_da} bounds 
$\bm{a}$ and its derivatives in terms of a 
\emph{truncated series expansion} of $\xi$. In fact, 
any attempt to estimate $\bm{a}$ pointwise in terms of 
$\xi$ will fail in the limits $t\to0$ and $|x|\to\infty$.
\end{remark}

\begin{proof}[Proof of Lemma \ref{lem:green_estimates}]
It is readily checked that the operator
$\nabla_+ \cdot c\nabla_-$ is a bounded linear
operator on $L^1_\Dx(\R^d)$. It therefore
follows that the series \eqref{eq:greensfunction}
converges in the supremum norm on
$C([-T,T];L^1_\Dx(\R^d))$, for all $T>0$.
The fact that $a$ satisfies \eqref{eq:discheateq}
follows from standard ODE theory
(see, e.g., \cite[Theorem 1.2]{Paz1983}).

Integrating \eqref{eq:discheateq} over space yields
\begin{align*}
	\frac{\d}{\d t}\sum_{\alpha \in \Z^d} a_\alpha(t) \Dx^d
	= \sum_{\alpha \in \Z^d} \nabla_+ \cdot \bigl(c\nabla_- a_\alpha(t)\bigr) \Dx^d = 0,
\end{align*}
where we have applied summation-by-parts,
cf.~Lemma~\ref{lem:discrete_calculus}\ref{lem:discrete_calculus_int-by-parts}.
This verifies \ref{lem:green_estimates_unit-mass}.
Similarly, multiplying \eqref{eq:discheateq}
through by $a_\alpha$ and integrating over space yields
\begin{align*}
	\frac12\frac\d{\d t} \sum_{\alpha \in \Z^d} a_\alpha^2(t) \Dx^d
	 + \sum_{\alpha \in \Z^d} \underbrace{\nabla_-
	 a_\alpha(t) \cdot c \nabla_-
	 a_\alpha(t)}_{\geq\,0} \Dx^d =0,
\end{align*}
which confirms \ref{lem:green_estimates_energy-bound}.

As in the continuous case, using the
duality between $\Z^d$ and the $d$-torus of unit side
length $\mathbb{T}^d$ (cf.~e.g.~\cite{Katznelson:2004aa}) we set
\[
\widehat{a}(t, \theta) \coloneqq \sum_{\alpha \in \Z^d} a_\alpha(t)
\exp( 2 \pi i \alpha \cdot \theta),
\quad t > 0,\, \theta \in \T^d.
\]
This series is convergent in $\ell^2(\Z^d)$ as
$(a_{\alpha}(t))_{\alpha\in\Z^d,t\geq0}\in L^\infty([0,\infty);\ell^2(\Z^d))$,
given the equivalent condition that
$\bm{a} \in L^\infty([0,\infty); L^2_\Dx(\R^d))$.
Following \cite{Hoff:1985zm}, from the semi-discrete heat
equation \eqref{eq:discheateq}, and
using summation-by-parts, we obtain
\begin{align*}
	\frac{\d}{\d t}\widehat{a}(t,\theta)
	&= \sum_{\alpha \in \Z^d}
	\nabla_+ \cdot \bigl(c\nabla_- a_\alpha(t)\bigr)
	 e^{2 \pi i \alpha \cdot \theta}
	= \sum_{\alpha \in \Z^d}
	a_\alpha(t)\nabla_- \cdot \bigl(c\nabla_+
	 e^{2 \pi i \alpha \cdot \theta}\bigr) \\
	& = \sum_{\alpha \in \Z^d}a_\alpha(t) e^{2 \pi i \alpha \cdot \theta}
	\sum_{j=1}^d c^j\frac{e^{2\pi i\theta^j} - 2
	+ e^{-2\pi i\theta^j}}{\Dx^2}\\
	& = \Biggl(\sum_{j = 1}^d 2c^j\frac{\cos(2\pi\theta^j)-1}{\Dx^2}\Biggr)
	\widehat{a}(t,\theta).
\end{align*}
By the fact that $\widehat{a}(0,\theta)=\Dx^{-d}$, we obtain
\begin{equation}\label{eq:w_aux1}
\begin{split}
	\widehat{a}(t,\theta)
	&= \frac{1}{\Dx^d}\exp\Biggl(\sum_{j = 1}^d
	2c^jt\frac{\cos(2\pi\theta^j)-1}{\Dx^2}\Biggr).
\end{split}
\end{equation}
Thus, inverting the Fourier transform, using \eqref{eq:w_aux1}
and setting $r^j\coloneqq \frac{2c^jt}{\Dx^2}$, we find that for $t > 0$,
\begin{align*}
	a_\alpha(t)
	&= \int_{[-\hf, \hf)^d} \widehat{a}(t,\theta)
	e^{-2 \pi i \alpha\cdot \theta} \,\d\theta\\
	& = \frac{1}{\Dx^d}\int_{[-\hf, \hf)^d}
	\exp\Biggl(\sum_{j = 1}^d r^j
 	\bigl(\cos(2\pi\theta^j)-1\bigr)\Biggr)
 	e^{-2\pi i\alpha\cdot \theta}\,\d\theta \\
	& = \frac{1}{\Dx^d}\prod_{j=1}^d\Biggl(e^{-r^j}\int_{-\hf}^\hf
	e^{r^j\cos(2\pi\theta^j)} e^{-2\pi i\alpha^j\theta^j}\,\d\theta^j\Biggr) \\
	& = \frac{1}{\Dx^d}\prod_{j=1}^d\Biggl(e^{-r^j}\int_{-\hf}^\hf
	e^{r^j\cos(2\pi\theta^j)}\Bigl(\cos\bigl(2\pi\alpha^j\theta^j\bigr)
	- i \sin\bigl(2\pi\alpha^j\theta^j\bigr)\Bigr)\,\d\theta^j\Biggr).
\end{align*}
The function $\theta^j \mapsto
	 \exp\bk{r^j\cos(2\pi\theta^j)}
\sin(2\pi\alpha^j\theta^j)$ is odd, therefore
the imaginary parts of the integrals
vanish, leaving us with
\begin{align*}
	a_\alpha(t) &= \frac{1}{\Dx^d}\prod_{j=1}^d\Biggl(e^{-r^j}
	\int_{-\hf}^\hf e^{r^j\cos(2\pi\theta^j)}
	\cos(2\pi\alpha^j\theta^j)\,\d\theta^j\Biggr)
	\\ &= \frac{1}{\pi^d\Dx^d}\prod_{j=1}^d
	\Biggl(e^{-r^j}\int_0^{\pi} e^{r^j\cos(z^j)}
	\cos\bigl(\alpha^jz^j\bigr)\,\d z^j\Biggr),
\end{align*}
where we have used symmetry of the
integrand and a change of variables. Using the modified Bessel function \eqref{eq:besselfunc}, we get \eqref{eq:fundamentalsolution}.

We recall that $r^j\coloneqq\frac{2c^jt}{\Dx^2}$
and rewrite the above as
\begin{equation*}
	a_\alpha(t) = \prod_{j=1}^d\Xi(\alpha^j,r^j)
	\Psi_{\alpha^j}(r^j),
\end{equation*}
where
\begin{align*}
\Xi(n,r) &\coloneqq
	\frac1{\sqrt{\pi r \Dx^2}}
	\bk{1 + \frac{n^2}{2r} + \frac{|n|^3}{|2r|^{\thf}}}^{-1},\\
	\Psi_n(r) &\coloneqq \frac{e^{-r}I_n(r)}{\Dx \Xi(n,r)} = (2\pi r)^{\hf}
	\biggl(1 + \frac{n^2}{2r} + \frac{|n|^3}{|2r|^{\thf}}\biggr)e^{-r} I_n(r),
\end{align*}
with $\Xi$ being an order $\thf$ truncation of the Taylor series
for the fundamental solution to the
continuous heat equation.

We claim that $\Psi$ is uniformly bounded, and to this end,
we treat the two cases $|n|\leq \sqrt r$ and $|n|>\sqrt r$ separately. 
In the first case, we first note the following 
simple estimate on the Bessel function:
\begin{align*}
|I_n(r)e^{-r}| &\leq \frac1\pi \int_0^\pi e^{r(\cos z - 1)}\,\d z
\leq \frac1\pi \int_0^\pi e^{-rz^2/6}\,\d z
\leq  2 r^{-\hf}.
\end{align*}
Hence, if $|n|\leq \sqrt r$ then
\begin{align*}
\Psi_n(r) \lesssim r^\hf \bk{1 + \frac{n^2}{2r} 
	+ \frac{\abs{n}^3}{\bk{2r}^{\thf}}} r^{-\hf} \lesssim 1.
\end{align*}
Next, in the regime $|n|>\sqrt r$ we use the estimate
\[
I_n(r)e^{-r} \leq C n^{-3} r
\]
for some $C>0$ independent of $n$ and $r$ 
(see \cite[Lemma 4.1(i)]{AGMP2021}), and get
\[
\Psi_n(r) \lesssim r^\hf \frac{n^3}{r^\thf}  n^{-3} r = 1.
\]
This proves \eqref{eq:pointwise_da} for $m=0$.

With the bound $\Psi \lesssim 1$, we get:
\begin{equation}\label{eq:aLpEstimate}
\begin{split}
\|a(t)\|_{L^{p'}(\R^d)} &\lesssim \Biggl(\sum_{\alpha\in\Z^d}
	\biggl|\prod_{j=1}^d\Xi(\alpha^j,r^j)\biggr|^{p'}\Dx^d\Biggr)^{1/p'} \\
&= \prod_{j=1}^d\Biggl(\sum_{n\in\Z}|\Xi(n,r^j)|^{p'}\Dx\Biggr)^{1/p'}.
\end{split}
\end{equation}
Recalling that $r^j\coloneqq 2c^jt/\Dx^2$, 
we can estimate each of the above series as 
\begin{align*}
\sum_{n\in\Z}|\Xi(n,r^j)|^{p'}\Dx &= \bigl(\pi r \Dx^2\bigr)^{-p'/2} 
	\sum_{n\in\Z}\biggl(1 + \frac{n^2}{2r} 
	+ \frac{|n|^3}{|2r|^{\thf}}\biggr)^{-p'}\Dx \\
&\lesssim \bigl(r\Dx^2\bigr)^{-p'/2}\Dx \int_0^\infty 
	\biggl(1+\frac{x^2}{2r}+\frac{x^3}{(2r)^\thf}\biggr)^{-1}\,\d x \\
&= \bigl(r\Dx^2\bigr)^{-p'/2} r^{1/2} \Dx 
	\int_0^\infty \bigl(1+y^2+y^3\bigr)^{-1}\,\d y \\
&\lesssim t^{-p'/2+1/2}.
\end{align*}
Inserting into \eqref{eq:aLpEstimate} yields \eqref{eq:a_lp} for $m=0$.

Taking the $m$th order difference in the $\ell$th direction directly,
\eqref{eq:fundamentalsolution} gives 
\begin{align*}
&(\nabla_+^\ell)^m a_\alpha(t) = \frac{1}{\Dx^{d+m}}
\Biggl(\prod_{j\neq \ell} e^{-r^j} I_{\alpha^j}(r^j)
\Biggr)\Biggl(e^{-r^\ell}\sum_{k=0}^m
C_k I_{\alpha^\ell+k}(r^\ell)\Biggr)\\
&\qquad= \frac1{\Dx^d}\bk{4c^\ell t}^{-m/2}
\prod_{\substack{j=1\\j \not= \ell}}^d\Xi(\alpha^j,r^j)
\Psi_{\alpha^j}(r^j) \bk{r^\ell}^{m/2} \Biggl(e^{-r^\ell}\sum_{k=0}^m
C_k I_{\alpha^\ell+k}(r^\ell)\Biggr),
\end{align*}
for (binomial) coefficients $C_0,\dots,C_m\in\R$ which
only depend on $m$. The pointwise bound 
\eqref{eq:pointwise_da} for $m = 0,1,2$ then follows from 
\cite[Remark 2.5]{AL2021} and \cite[Lemma 4.1, (ii) \& (iii)]{AGMP2021}.
This concludes the proof of 
\ref{lem:green_estimates_derivative-bound}.
\end{proof}

We note in particular that replacing the
repeated forward differencing in
\eqref{eq:pointwise_da} with alternating forward and
backward differences, to second order, we have
\begin{equation}\label{eq:a_F_bound4}
\begin{aligned}
\abs{\nabla_+^\ell\nabla_-^\ell a_\alpha(t)}
	&\lesssim t^{-\bk{d + 2}/2} \prod_{j = 1}^d
		\bk{1 + \frac{\abs{x_\alpha^j}^2}{4 \overline{c}t}
	+ \frac{\abs{x_\alpha^j}^3}{\abs{4 \overline{c}t}^\thf} }^{-1}.
\end{aligned}
\end{equation}

\begin{lemma}[Solution formula]
\label{lem:duhamel_representation}
Fix $p,q$ satisfying $p > d$ and $2/q+d/p<1$.
Let $\bm{f } \in L^q\bigl([0,T];L^p_\Dx(\R^d)\bigr)$
and $\bm{\psi} \in L^\infty_\Dx(\R^d)$. Then the unique
solution $\bm{\varphi}$ of the inhomogeneous
semi-discrete heat equation
\[
\begin{cases}\displaystyle
	\frac{\d}{\d t}  \varphi_\alpha
	-\nabla_+ \cdot \bigl(c\nabla_-\varphi_\alpha\bigr)
	= f_\alpha, & t \in [0,T], \alpha \in \Z^d, \\
	\varphi_\alpha(0) = \psi_\alpha, & \alpha\in\Z^d
\end{cases}
\]
(cf.~\eqref{eq:def_discrete_laplacian}) is given
by the Duhamel formula
\[
\varphi_\alpha(t)
\coloneqq \bk{\psi \dconv a(t)}_\alpha
+ \int_0^t \bk{a(t - s) \dconv f(s)}_\alpha \,\d s, \qquad t > 0,
\,\, \alpha \in \Z^d.
\]
This function satisfies
\begin{equation}\label{eq:non-homogeneous-bounds}
\begin{split}
	\|\bm{\varphi}(t)\|_{L^\infty(\R^d)}
	&\leq \|\bm{\psi}\|_{L^\infty(\R^d)}
	+ C\|\bm{f}\|_{L^q([0,T];L^p (\R^d))},  \\
	\|\nabla_+ \bm{\varphi}(t)\|_{L^\infty (\R^d)}
	&\leq \|\nabla_+ \bm{\psi}\|_{L^\infty(\R^d)}
	+ C\|\bm{f}\|_{L^q([0,T];L^p (\R^d))},
\end{split}
\end{equation}
for some $C>0$ which is independent of $\Dx$.
\end{lemma}

\begin{proof}
Write $\nabla_{\pm,\alpha}$ for the difference
operators $\nabla_\pm$ applied to the $\alpha$ variable,
so that, e.g.,~$\nabla_{+,\alpha}\psi_{\alpha-\beta}
=-\nabla_{-,\beta}\psi_{\alpha-\beta}$.
By the summation-by-parts formula (Lemma
\ref{lem:discrete_calculus}\ref{lem:discrete_calculus_differentiate-convolution}), 
\begin{align*}
	\bigl(\psi \dconv \bigl(\nabla_+ \cdot c\nabla_-a\bigr)\bigr)_\alpha
	&= \Dx^d\sum_{\beta \in \Z^d} \psi_{\alpha-\beta}
	\nabla_{+,\beta}\bigl(c\nabla_{-,\beta} a_\beta\bigr) \\
	&= \Dx^d\sum_{\beta \in \Z^d} \nabla_{+,\beta}
	\bigl(c\nabla_{-,\beta}\psi_{\alpha-\beta}\bigr) a_\beta \\
	&= \Dx^d\sum_{\beta \in \Z^d} \nabla_{-,\alpha}
	\bigl(c\nabla_{+,\alpha}\psi_{\alpha-\beta}\bigr) a_\beta \\
	&=  \nabla_{-,\alpha}\biggl(c\nabla_{+,\alpha}
	\biggl( \Dx^d\sum_{\beta \in \Z^d}\psi_{\alpha-\beta} a_\beta\biggr)\biggr)
	= \nabla_+ \cdot c\nabla_- \bigl(\psi \dconv a\bigr)_\alpha.
\end{align*}
Hence,
\begin{align*}
	\frac{\d}{\d t}  \phi_\alpha(t) &=
	\bigl(\psi \dconv \bigl(\nabla_+ \cdot c\nabla_-a(t)\bigr)\bigr)_\alpha
	+ \int_0^t \bigl(\nabla_+ \cdot c\nabla_- a(t - s)) \dconv f(s)\bigr)_\alpha \,\d s \\
	&\relspace + \bk{f(t) \dconv a(0)}_{\alpha} \\
	&= \nabla_+ \cdot c\nabla_- \biggl((\psi \dconv a(t))_\alpha
	+ \int_0^t (a(t - s) \dconv f(s))_\alpha \,\d s\biggr)
	+ f_{\alpha}(t) \\
	&= \nabla_+ \cdot c\nabla_-\phi_\alpha(t) + f_\alpha(t).
\end{align*}
Here, we used the linearity of the convolution.
To justify exchanging differentiation and convolution
above, we recall from Lemma \ref{lem:green_estimates}
\ref{lem:green_estimates_derivative-bound}
that $a_\alpha(t)$ and its differences
decay sufficiently quickly as $|\alpha| \to \infty$.
Indeed, by Young's convolution
inequality \eqref{eq:younginequality},
we have
\begin{align*}
\|\bm{a}\dconv\bm{f}\|_{L^\infty_t L^\infty_\alpha}
& \le \|\bm{a}\|_{L^{q'}_tL^{p'}_\alpha}\|\bm{f}\|_{L^q_tL^p_\alpha}.
\end{align*}
Using \eqref{eq:a_lp},
in order for $\norm{ \bm{a}}_{L^{q'}_tL^{p'}_\alpha}$ to be
bounded, we need
\begin{align*}
\biggl(-\frac{d}{2p}\biggr)q' > -1
\qquad &\Longleftrightarrow\qquad
{\frac2q+\frac{d}{p}<2}. \\
\intertext{A similar estimate on $\nabla_+\phi_\alpha(t)$ will require}
\biggl(-\frac{d}{2p}-\frac12\biggr)q' > -1
\qquad&\Longleftrightarrow\qquad
{\frac{2}{q}+\frac{d}{p} < 1}.
\end{align*}
The bounds \eqref{eq:non-homogeneous-bounds}
follow from the same kind of estimates.
\end{proof}

\subsection{The variable coefficient case}\label{sec:variable_coeff}
Let $\bm{f} \in L^\infty_\Dx(\R^d)$ and
$\bm{\psi}, \bm{c}^1, \ldots, \bm{c}^d \in L_\Dx^\infty(\R^d)$ be
given grid functions and define ${\bm c}= \text{diag}
(\bm{c}^1, \ldots, \bm{c}^d)$. Assume that $c_\alpha^j \ge \ep >  0$
for all $j,\alpha$, and that
\begin{align}\label{eq:smoothrho}
	\frac{\big|c_\alpha^j - c_\beta^j\big|}
	{\bigl|x_\alpha - x_\beta\bigr|}\le C,
	\quad j=1,\ldots d,
\end{align}
for a constant $C$ independent of $\alpha, \beta, j$.
We keep the convention that the superscript $j = 1, \ldots, d$ of $c_\alpha^j$
are indices and not exponents. We consider here the anisotropic
\emph{variable coefficient} heat equation
\begin{equation}\label{eq:heat_variable_coeff}
\begin{cases}\displaystyle
	\frac{\d}{\d t} \varphi_\alpha - \frac12 \sum_{j = 1}^d
	c_\alpha^j \nabla_+^j \nabla_-^j  \varphi_\alpha
	= f_\alpha, & \alpha\in\Z^d,\ t>0, \\
	\varphi_\alpha(0) = \psi_\alpha, & \alpha\in\Z^d.
\end{cases}
\end{equation}
The parametrix method
(see, e.g., \cite[Ch.~9]{Fri1964}, \cite[p.~56]{Taylor:2011aa}
in the continuous setting) is a technique for
solving heat equations with variable coefficients,
which we implement in a discrete form here.
We do so largely by using
bounds proven in Section
\ref{sec:constant_coeff}, via freezing
the argument $\alpha$ of the
coefficient $c_\alpha^j$
at some fixed index $\beta$.
This gives us a constant coefficient
heat operator as
long as ${\bm c}$ is bounded
away from $0$, and a resultant first
approximation $a_{\alpha,\beta}$
to the fundamental solution of the
full variable coefficient operator.
In pursuit of this ``coefficient freezing''
argument of the parametrix method,
we first establish estimates for
a quantity characterising the error
made by freezing $\alpha$ in
the argument of ${c_\alpha^j}$.

We first establish a rudimentary well-posedness 
result for \eqref{eq:heat_variable_coeff}, 
which will be improved upon later on.
\begin{lemma}\label{lem:basic-existence}
Let $\bm{c}^1,\dots,\bm{c}^d,\bm{f} 
	\in L_\Dx^\infty(\R^d)$. 
Then there exists $\bm{\phi}\in L^\infty_{t,\loc} L^\infty_\alpha$ such that 
$\phi_\alpha(t)$ is $C^1$ in $t$, and such 
that \eqref{eq:heat_variable_coeff} is 
satisfied for every $\alpha\in\Z^d$, $t>0$. 
Moreover, $\bm\phi$ is unique in the class 
of functions that are Lipschitz in $t$ and 
satisfy \eqref{eq:heat_variable_coeff} 
almost everywhere.
\end{lemma}
\begin{proof}
Write \eqref{eq:heat_variable_coeff} as 
$\frac{\d}{\d t}\bm\phi(t) = A(\bm{\phi}(t)) + \bm{f}$. 
Then $A$ is a bounded linear operator on 
$L^\infty_\alpha$ (with operator norm 
$\lesssim \|\bm{c}\|_{L^\infty}/\Dx^2$), 
so by a standard fixed point argument 
there exists a unique integral solution 
of the ODE, and which is continuously differentiable in $t$ (see e.g.~\cite[Theorem 7.3]{Brezis2011}). 

\end{proof}

For ${\bm c}\in L^\infty_\Dx(\R^d)$ as
in \eqref{eq:smoothrho}, we define
$\bigl(a_{\alpha, \beta}\bigr(t))_{\alpha,\beta\in\Z^d}$
as follows: For fixed $\beta\in\Z^d$,
the map $(t,\alpha)\mapsto a_{\alpha,\beta}(t)$
is the fundamental solution of
\[
\frac{\d}{\d t}  \varphi_\alpha
- \frac12\sum_{j = 1}^d c_\beta^j  \nabla_+^j  \nabla_-^j
\varphi_\alpha = 0,
\quad (t,\alpha)\in (0,T]\times \Z^d.
\]
Using the theory from Section \ref{sec:constant_coeff},
we can write
\begin{equation}\label{eq:fund_soln_parametrix}
a_{\alpha,\beta}(t) \coloneqq \sum_{n \ge 0}
\frac{t^n}{n!} \Biggl(\sum_{j = 1}^d
c_\beta^j \nabla_+^j \nabla_-^j \Biggr)^n{\delta}_\alpha.
\end{equation}

\begin{lemma}[Auxiliary estimates]
\label{thm:K_estimates_lemma}
Fix $\beta \in \Z^d$.
Define the auxiliary kernel
\begin{align}\label{eq:K_defin}
	K_{\alpha, \beta}(t )
	&\coloneqq \frac12 \sum_{j = 1}^d \bigl(c_\alpha^j
	-c_\beta^j\bigr)
	\nabla_+^j \nabla_-^j{\, a_{\alpha -\beta, \beta}(t )}
\end{align}
where the difference operators $\nabla_\pm$
act only on the $\alpha$ subscript. Then 
\begin{equation}\label{eq:K_estimate}
\begin{aligned}
|K_{\alpha, \beta}(t)|
	&\lesssim t^{-\bk{d + 1}/2} \prod_{j = 1}^d
	\bk{1 + \frac{|x_\alpha^j- x_\beta^j|^2}{ 4\overline{c}t}}^{-1},
\end{aligned}
\end{equation}
where  $\overline{c} = \norm{\bm c}_{L^\infty(\R^d)}$.
Moreover, for $p > d$, $p' = p/\bk{p - 1}$, 
\begin{align*}
\|\bm{K}(t)\|_{L^\infty_\alpha L^{p'}_\beta}
& \lesssim_{p} t^{-d/2p - 1/2},
\end{align*}
and for a constant $C_p>0$ independent of $\Dx$,
\begin{align}\label{eq:K_L1_tx}
\|\bm{K}\|_{L^1_t L^\infty_\alpha L^{p'}_\beta}
	\le C_p T^{-d/2p+1/2}.
\end{align}
\end{lemma}

\begin{proof}
By Lemma \ref{lem:green_estimates}
\ref{lem:green_estimates_derivative-bound},
and in particular \eqref{eq:pointwise_da} and \eqref{eq:a_F_bound4},
with $\overline{c} := \sup_{j,\alpha} c^j_\alpha$,
\begin{equation*}
\begin{aligned}
	|K_{\alpha, \beta}(t)|
	& = \frac12 \Biggl|\sum_{j = 1}^d\bk{c_\alpha^j
	- c_\beta^j}
	\nabla_+^j \nabla_-^j a_{\alpha - \beta, \beta}\Biggr|\\
	& \le  \frac{|x_\alpha - x_\beta|}{\sqrt{t}} \sqrt{t}
	\max_j \frac{|c_\alpha^j
	- c_\beta^j|}{|x_\alpha - x_\beta|}
	\sum_{j = 1}^d \Big| \nabla_+^j \nabla_-^j
	a_{\alpha - \beta, \beta}\Big|\\
	&\lesssim  \frac{|x_\alpha - x_\beta|}{\sqrt{t}} 
 t^{- \bk{d + 1}/2} \prod_{j = 1}^d
		\bk{1 + \frac{|x_\alpha^j- x_\beta^j|^2}{4 \overline{c}t}
	+ \frac{|x_\alpha^j- x_\beta^j|^3}{\abs{4 \overline{c}t}^\thf} }^{-1}.
\end{aligned}
\end{equation*}
For any $x,y\in\R^d$ and $t>0$ we have
\[
\frac{|x - y|}{\sqrt{t}}
 \prod_{j=1}^d \bk{1 + \frac{|x^j - y^j|^2}{4 \overline{c}t}
	+ \frac{|x^j-y^j|^3}{\abs{4 \overline{c}t}^\thf} }^{-1}
\lesssim \prod_{j=1}^d \bk{1 + \frac{|x^j -y^j|^2}{4 \overline{c}t}  }^{-1}.
\]
The above estimate on $K_{\alpha, \beta}(t)$
therefore reduces to \eqref{eq:K_estimate}.

Next we show that $K_{\alpha, \beta}(t)$ is
bounded in $L^1_t L^\infty_\alpha L^{p'}_\beta$.
From \eqref{eq:K_estimate},
\begin{align*}
\sum_{\beta \in \Z^d}\abs{K_{\alpha, \beta}(t)}^{p'}  \Dx^d
& \le t^{-\bk{d + 1}{p'}/2 + d/2} \prod_{j = 1}^d 
	\Biggl( \sum_{\beta^j \in \Z} \bk{1 
	+ \frac{|x_\alpha^j - x_\beta^j|^2}
		{4 \overline{c} t}}^{-p'} \frac{\Dx}{t^{\hf}}\Biggr)\\
& \lesssim_{\overline{c},d,p} t^{-\bk{d + 1}{p'}/2 + d/2}.
\end{align*}
Therefore,
\begin{align*}
\|\bm{K}\|_{L^1_t L^\infty_\alpha L^{p'}_\beta} &=
\int_0^T \sup_{\alpha \in \Z^d}
    \norm{K_{\alpha, \cdot}(t)}_{L^{p'}_\beta}\,\d t
\lesssim_{\overline{c},d,p} \int_0^T
    t^{-\bk{d + 1}/2 + d/2p'}\,\d t\\
&= \int_0^T
    t^{-d/2p - 1/2}\,\d t
\lesssim T^{1/2 -d/2p}.\qedhere
\end{align*}
\end{proof}

\begin{lemma}\label{lem:auxcalc3}
Let $K^{(1)} \coloneqq K$, and for $m = 2,3,\ldots$, 
let $\bm{K}^{(m)} \coloneqq \bm{K}\dconv \bm{K}^{(m-1)}$, i.e.
\begin{equation}\label{eq:Km_defin}
	K^{(m)}_{\alpha, \beta}(t)
	\coloneqq \int_0^t
	\sum_{\eta \in \Z^d} K_{\alpha, \eta}(t - s)
	 K^{(m - 1)}_{\eta, \beta}(s)\Dx^d\,\d s.
\end{equation}
Then there is a $T_0>0$ such that for each $p > d$, the series
\begin{equation}\label{eq:Phi_defin}
	\bm{\Phi} \coloneqq \sum_{m=1}^\infty \bm{K}^{(m)}
\end{equation}
converges absolutely in
$L^1([0,{T_0}])L^\infty_\alpha L^{p'}_\beta$. 
For all $T \in [0,T_0]$, ${\bm \Phi}$ satisfies
\begin{align}\label{eq:Phi_L1Lp}
\|\bm{\Phi}\|_{L^1_t([0,T]) L^\infty_\alpha L^{p'}_\beta}
\lesssim_{c,p} T^{-d/2p+1/2}.
\end{align}
\end{lemma}

\begin{proof}
By Young's convolution inequality (Lemma \ref{lem:young}), we find that
\[
\|\bm{K}^{(m)}\|_{L^1_t L^\infty_\alpha L^{p'}_\beta} \leq \|\bm{K}\|_{L^1_t L^\infty_\alpha L^1_\beta} \|\bm{K}^{(m-1)}\|_{L^1_t L^\infty_\alpha L^{p'}_\beta}
\]
By induction on $m$ and using \eqref{eq:K_L1_tx} we arrive at
\[
\|\bm{K}^{(m)}\|_{L^1_t L^\infty_\alpha L^{p'}_\beta}
\leq \|\bm{K}\|_{L^1_t L^\infty_\alpha L^1_\beta}^{m-1} \|\bm{K}\|_{L^1_t L^\infty_\alpha L^{p'}_\beta}
\leq \bigl(C_\infty T^{1/2}\bigr)^{m-1} C_p T^{-d/2p+1/2}.
\]
Summing over $m$ and using Weierstrass' M-test, we find that the series \eqref{eq:Phi_defin} converges in $L^1_t L^\infty_\alpha L^{p'}_\beta$ as long as, say, $T \leq T_0 \coloneqq \frac12 C_\infty^{-2}.$ In particular, we can bound
\begin{align*}
\|\Phi\|_{L^1([0,T])L^\infty_\alpha L^{p'}_\beta}
&\leq \sum_{m=1}^\infty \bigl(C_\infty T^{1/2}\bigr)^{m-1} C_p T^{-d/2p+1/2} \\
&= \frac{C_p T^{-d/2p+1/2}}{1 - C_\infty T^{1/2}} \lesssim T^{-d/2p+1/2}.\qedhere
\end{align*}
\end{proof}

\begin{definition}[Fundamental solution]
\label{defin:fundsol_varicoeff}
We say that $\bm{\Gamma}
=\bigl(\Gamma_{\alpha, \beta}(t)\bigr)_{\substack{\alpha, \beta \in \Z^d,
\\ 0 \leq t\leq T}}$
is a {fundamental solution} of \eqref{eq:heat_variable_coeff} on $[0,T]$ if
\begin{equation}\label{eq:fund_soln_variable_heat}
\begin{cases}
	\frac{\d}{\d t}  \Gamma_{\alpha, \beta}(t)
	= \frac12 \sum_{j = 1}^dc_\alpha^j
	\nabla_+^j \nabla_-^j \Gamma_{\alpha, \beta}(t),
	& \alpha,\beta\in\Z^d,\ 0< t\leq T,
	\\ {\Gamma}_{\alpha,\beta}(0)
	= \delta_{\alpha-\beta}, & \alpha,\beta\in\Z^d.
\end{cases}
\end{equation}
\end{definition}

The next result establishes the existence
and regularity of a fundamental solution.

\begin{lemma}[Solution formula]\label{lem:full_greensfunction}
Let ${T_0}$ be as given in Lemma
\ref{lem:auxcalc3}. Then the semi-discrete 
variable coefficient parabolic problem 
\eqref{eq:heat_variable_coeff} has a 
fundamental solution $\bm{\Gamma}$ on $[0,T_0]$ of the form
\begin{align}\label{eq:Gamma_representation}
\Gamma_{\alpha, \beta}(t)
	= a_{\alpha - \beta, \beta}(t)
	+ \int_0^{t} \sum_{\eta \in \Z^d} a_{\alpha - \eta, \eta}(t - s)
	\Phi_{\eta, \beta}(s)\, \Dx^d \,\d s, \qquad \forall\ t\in[0,T_0]
\end{align}
where $\bm{\Phi}=
    \bigl(\Phi_{\alpha, \beta}(t)\bigr)_{\alpha, \beta \in \Z^d,\,t\geq0}$
is defined in \eqref{eq:Phi_defin}. For $p \in (d,\infty]$, we have
\begin{equation}\label{eq:Gamma_bound}
    \bigl\|\bm{\Gamma}(t)\bigr\|_{L^\infty_\alpha L^{p'}_\beta} 
    \lesssim_p  t^{-d/(2p)} \qquad \forall\ t\in(0,T_0],
\end{equation}
where $p' = p/(p - 1)$ is the H\"older conjugate of $p$.
\end{lemma}

\begin{proof}
As in \eqref{eq:fund_soln_parametrix}, we
let $\bm{a}=a_{\alpha,\beta}(t)$ denote the fundamental
solution of the constant-coefficient heat
equation \eqref{eq:discrete_heat_const}
with coefficient $c\equiv c_\beta$.
Following the approach in the continuous
setting in \cite[Ch. 9, \S 4]{Fri1964},
we seek to construct the fundamental
solution $\bm{\Gamma}$ in the
form \eqref{eq:Gamma_representation}.
We split the remainder of the
proof into two steps---derivation
of $\bm{\Phi}$ and estimates on $\bm{\Gamma}$.

\medskip

\noindent \textit{Step 1: Derivation of $\bm{\Phi}$.}
Let $[0,{{T_0}}]$ be a time interval on which the series 
for $\Phi$ converges; see Lemma \ref{thm:K_estimates_lemma}.
Taking the temporal derivative of
\eqref{eq:Gamma_representation} for $\bm{\Gamma}$, we get
\begin{align*}
	\frac{\d}{\d t}  \Gamma_{\alpha, \beta}(t)
	&= \frac{\d}{\d t}  a_{\alpha - \beta, \beta}(t)
	 + \sum_{\eta \in \Z^d} a_{\alpha - \eta, \eta}(0)
	 \Phi_{\eta, \beta}(t) \Dx^d \\*
	&\relspace + \int_0^t \sum_{\eta \in \Z^d} \frac{\d}{\d t}
	a_{\alpha - \eta, \eta}(t - s)
	\Phi_{\eta, \beta}(s) \Dx^d \,\d s \\
	& =\frac12\sum_{j = 1}^d c_\beta^j \nabla_+^j
	\nabla_-^j \,a_{\alpha -\beta, \beta}(t)
	+ \Phi_{\alpha, \beta}(t)\\*
	&\relspace + \frac12\int_0^t \sum_{\eta \in \Z^d}
	\sum_{j = 1}^d c_\eta^j \nabla_+^j
	\nabla_-^j \,a_{\alpha -\eta, \eta}(t - s)
	\Phi_{\eta, \beta}(s) \Dx^d \,\d s,
\end{align*}
and similarly for the difference,
\begin{align*}
	\nqquad\frac12 \sum_{j = 1}^d c_\alpha^j \nabla_+^j
	\nabla_-^j \Gamma_{\alpha, \beta}(t)
    &=   \frac12\sum_{j = 1}^d c_\alpha^j \nabla_+^j
	\nabla_-^j  a_{\alpha - \beta, \beta}(t)\\*
	&\relspace + \frac12  \int_0^t \sum_{\eta \in \Z^d}
	\sum_{j = 1}^dc_\alpha^j \nabla_+^j \nabla_-^j
	a_{\alpha - \eta, \eta}(t - s)\Phi_{\eta, \beta}(s) \Dx^d\,\d s.
\end{align*}
Therefore, $\bm{\Gamma}$ is a fundamental
solution to \eqref{eq:heat_variable_coeff}
if and only if
\begin{align}\label{eq:fundmat_sol}
	\Phi_{\alpha, \beta}(t)
	= K_{\alpha, \beta}(t)
	+ \int_0^t \sum_{\eta \in \Z^d} K_{\alpha, \eta}(t - s)
	\Phi_{\eta, \beta}(s) \Dx^d \,\d s,
\end{align}
where  $K_{\alpha, \beta}$ is defined by \eqref{eq:K_defin}.
The series expansion \eqref{eq:Phi_defin} of $\bm\Phi$
provides a formal solution to
\eqref{eq:fundmat_sol}, as
\begin{align*}
	&\nqquad K_{\alpha, \beta}(t)
	+ \int_0^t \sum_{\eta \in \Z^d} K_{\alpha, \eta}(t - s)
	\Phi_{\eta, \beta}(s) \Dx^d \,\d s\\
	& = K_{\alpha, \beta}(t)
	+ \int_0^t \sum_{\eta \in \Z^d} K_{\alpha, \eta}(t - s)
	\sum_{m= 1}^\infty K^{(m)}_{\eta, \beta}(t) \,\d s\\
	& = K_{\alpha, \beta}(t)
	+\sum_{m= 1}^\infty K^{(m + 1)}_{\alpha, \beta}(t)
	 = \Phi_{\alpha, \beta}(t).
\end{align*}
In the calculation above,
in order to exchange the sums
over $m$ with the sum over $\eta$,
and then exchange the sum over $m$
with the integral, we require
the uniform convergence of the series
$\Phi_{\alpha, \beta}(t)$ with respect to $\alpha, \beta \in \Z^d$
and in $L^1([0,T_0])$.
This is provided by the final
assertion of Lemma \ref{lem:auxcalc3}.

\medskip

\noindent \textit{Step 2: Estimates on $\bm{\Gamma}$.}
Let us denote $F_{\alpha,\beta}(t) = a_{\alpha-\beta,\beta}(t)$, so that we can write
\[
\bm{\Gamma} = \bm{F} + \bm{F}\dconv \bm{\Phi}.
\]
Using the uniform
upper boundedness of $\bm{c}$, set
$\overline{c} = \sup_{\beta, j} c_{\beta}^j$.
From \eqref{eq:pointwise_da}, we then have
\begin{align*}
|F_{\alpha,\beta}(t)|
&\lesssim t^{-d/2} \prod_{j = 1}^d
    \bk{1 + \frac{|x_\alpha^j- x_\beta^j|^2}{4 \overline{c}t}
	+ \frac{|x_\alpha^j- x_\beta^j|^3}{\abs{4 \overline{c}t}^\thf} }^{-1}.
\end{align*}
Therefore, as in the proof of \eqref{eq:a_lp}, we find
\begin{equation}\label{eq:a_alpha_beta_normbounds1}
\begin{aligned}
\|\bm{F}(t)\|_{L^\infty_\alpha L^{p'}_\beta} \lesssim t^{-d/2p},
\end{aligned}
\end{equation}
and
\begin{equation}\label{eq:a_alpha_beta_normbounds2}
\begin{aligned}
\|\bm{F}\|_{L^\infty_t L^\infty_\alpha L^1_\beta} \lesssim 1.
\end{aligned}
\end{equation}
From \eqref{eq:Phi_L1Lp} in Lemma \ref{lem:auxcalc3}, for $p > d$ we also have
\begin{align}\label{eq:Phi_normbound}
\norm{\bm \Phi}_{L^1_t([0,T_0]) L^\infty_\alpha L^{p'}_\beta}
\lesssim 1.
\end{align}
By the triangle inequality and Young's 
convolution inequality (Lemma \ref{lem:young}), 
we therefore find that
\begin{align*}
\|\bm{\Gamma}(t)\|_{L^\infty_\alpha L^{p'}_\beta} 
&\leq \|\bm{F}(t)\|_{L^\infty_\alpha L^{p'}_\beta} 
	+ \|(\bm{F}\dconv\bm{\Phi})(t)\|_{L^\infty_\alpha L^{p'}_\beta} \\
&\leq \|\bm{F}(t)\|_{L^\infty_\alpha L^{p'}_\beta} 
	+ \|\bm{F}\dconv\bm{\Phi}\|_{L^\infty_t([0,t])L^\infty_\alpha L^{p'}_\beta} \\
&\leq \|\bm{F}(t)\|_{L^\infty_\alpha L^{p'}_\beta} 
	+ \|\bm{F}\|_{L^\infty_t([0,t])L^\infty_\alpha L^1_\beta}
	\|\bm{\Phi}\|_{L^1_t([0,t])L^\infty_\alpha L^{p'}_\beta} \\
&\lesssim t^{-d/2p} + 1.\qedhere
\end{align*}
\end{proof}

We now extend uniquely the definition of 
${\bm \Gamma}$ from $[0,{T_0}]$ to all times.

\begin{lemma}
Let ${T_0}$ be as given in Lemma \ref{lem:auxcalc3}.
For $k \in \N$, and $t \in [0,{T_0}]$,
recursively define $\bm{\Gamma}(kT_0+t)
	=\bm{\Gamma}(t)\dconv \bm{\Gamma}(kT_0)$, that is,
\begin{align}\label{eq:Gamma_propagation}
\Gamma_{\alpha, \beta}(k{T_0} + t)
= \sum_{\eta \in \Z^d} \Gamma_{\alpha, \eta}(t)
    \Gamma_{\eta,\beta}(k{T_0})\Dx^d.
\end{align}
Then $\bm{\Gamma}$ is the unique fundamental
solution to \eqref{eq:heat_variable_coeff} on any time interval $[0,T]$.
Moreover,
\begin{align}\label{eq:Gamma_bound2}
\norm{\bm{\Gamma}(k{T_0})}_{L^\infty_\alpha L^{p'}_\beta} 
	\lesssim_k {T_0}^{-d/2p} \qquad \forall\ k\in\N,
\end{align}
and for any $T >0$,
\begin{align}
\|\bm{\Gamma}\|_{L^1_t([0,T])L^\infty_\alpha L^{p'}_\beta} 
&\lesssim_{T,p} 1, \label{eq:Gamma_bound4} \\
\|\bm{\Gamma}\|_{L^\infty_t([0,T])L^\infty_\alpha L^1_\beta} 
&\lesssim_{T,p} 1. \label{eq:Gamma_bound5}
\end{align} 
\end{lemma}

As a consequence of uniqueness 
of the fundamental solution, 
\eqref{eq:Gamma_propagation} will in fact hold for ${\bm \Gamma}$ 
with any value of $t, T_0 \in [0,\infty)$, 
with $t$ not necessarily smaller than $T_0$.

\begin{proof}
Let us first establish the estimates showing
that the series converges and is bounded as stated. 
We apply Young's convolution inequality 
(Lemma \ref{lem:young}) repeatedly and 
use \eqref{eq:Gamma_bound} to obtain
\begin{align*}
\|\bm{\Gamma}(kT_0)\|_{L^\infty_\alpha L^{p'}_\beta} 
&\leq \|\bm{\Gamma}(T_0)\|_{L^\infty_\alpha L^1_\beta}	
	\|\bm{\Gamma}((k-1)T_0)\|_{L^\infty_\alpha L^{p'}_\beta} \\
&\leq \cdots 
	\leq \|\bm{\Gamma}(T_0)\|_{L^\infty_\alpha L^1_\beta}^{k-1}
	\|\bm{\Gamma}(T_0)\|_{L^\infty_\alpha L^{p'}_\beta}
\lesssim_k T_0^{-d/2p}.
\end{align*}
Moreover, \eqref{eq:Gamma_bound} with 
$\infty$ in place of $p$ yields 
$\|\bm{\Gamma}(t)\|_{L^\infty_\alpha L^1_\beta} 
	\lesssim 1$. Therefore, for any $k\in\N_0$,
\begin{align*}
\|\bm{\Gamma}\|_{L^1_t([kT_0,(k+1)T_0]),L^\infty_\alpha L^{p'}_\beta} 
&= \int_{0}^{T_0}\|\bm{\Gamma}(t)\dconv
	\bm{\Gamma}(kT_0)\|_{L^\infty_\alpha L^{p'}_\beta}\,\d t \\
&\leq \int_{0}^{T_0}\|\bm{\Gamma}(t)\|_{L^\infty_\alpha L^1_\beta} 
	\|\bm{\Gamma}(kT_0)\|_{L^\infty_\alpha L^{p'}_\beta}\,\d t \\
&\leq T_0 \cdot T_0^{-d/2p}.
\end{align*}
Splitting an arbitrary time interval $[0,T]$ into 
subintervals of length at most $T_0$ then yields 
\eqref{eq:Gamma_bound4}. A similar estimate will yield
\[
\|\bm{\Gamma}\|_{L^\infty_t([0,kT_0])L^\infty_\alpha L^1_\beta}
\leq \|\bm{\Gamma}\|_{L^\infty_t([0,T_0])L^\infty_\alpha L^1_\beta} 
\|\bm{\Gamma}(T_0)\|_{L^\infty_\alpha L^1_\beta}^k \lesssim 1,
\]
again due to \eqref{eq:Gamma_bound}.

Next we show that the extension of 
$\bm{\Gamma}$ is a fundamental solution on $[0,T]$.
We already know that $\Gamma_{\alpha, \beta}(0) = \delta_{\alpha - \beta}$.
Taking a temporal derivative on both sides of
\eqref{eq:Gamma_propagation}, for $t \in (0,T_0)$, we find:
\begin{align*}
&\frac{\d}{\d t} \Gamma_{\alpha, \beta}(k {T_0} + t) \\
& = \sum_{\eta \in \Z^d} \frac{\d}{\d t} \Gamma_{\alpha, \eta}(t)
	\Gamma_{\eta, \beta}(k {T_0})\Dx^d
= \sum_{\eta \in \Z^d} \frac12 \sum_{j = 1}^d
    c_{\alpha}^j \nabla_+^j \nabla_-^j \Gamma_{\alpha, \eta}(t)
    \Gamma_{\eta, \beta}(k {T_0})\Dx^d\\
& =\frac12 \sum_{j = 1}^d c_{\alpha}^j \nabla_+^j \nabla_-^j
    \sum_{\eta \in \Z^d} \Gamma_{\alpha, \eta}(t)
    \Gamma_{\eta, \beta}(k {T_0})\Dx^d
 = \frac12 \sum_{j = 1}^d c_{\alpha}^j \nabla_+^j \nabla_-^j
    \Gamma_{\alpha, \beta}(k {T_0} + t).
\end{align*}
Hence, the extension of $\bm\Gamma$ satisfies the 
ODE on any interval of the form $\big(kT_0,(k+1)T_0\big)$. 
But the extended function $\bm\Gamma$ is also locally 
Lipschitz in $t$, so by the uniqueness result in Lemma 
\ref{lem:basic-existence}, it must coincide with the 
unique $C^1_t$ solution of the ODE. 
This concludes the proof.
\end{proof}

Next, we proceed to establish bounds on the spatial differences
of $\bm{\Gamma}$, similar to those
for $\bm{a}$ in Lemma \ref{lem:green_estimates}. 
Recall in particular the definition
of $\nabla_+^\ell$ in \eqref{eq:difference_operators}.

\begin{lemma}[Fundamental solution estimates]
\label{lem:full_greensfunction2}
For any $p > d$, we have
\begin{equation*}
\|\nabla_\pm^\ell\bm{\Gamma}\|_{L^1_t([0,T])L^\infty_\alpha L^{p'}_\beta} \lesssim_{T,p} 1.
\end{equation*}
\end{lemma}
\begin{proof}
Using \eqref{eq:Gamma_representation},
we take a difference
in direction $\ell$ to get 
\begin{align*}
\nabla_\pm^\ell \Gamma_{\alpha, \beta}(t)
&= \nabla_\pm^\ell  a_{\alpha - \beta, \beta}(t)
+ \int_{0}^t \sum_{\eta \in \Z^d}\nabla_\pm^\ell
a_{\alpha - \eta, \eta}(t - s)\Phi_{\eta, \beta}(s) \Dx^d \,\d s.
\end{align*}
Controlling $|\nabla_\pm^\ell a_{\alpha, \beta}(t)|$ by
\eqref{eq:pointwise_da} and  $\Phi$ by \eqref{eq:Phi_normbound},
we can calculate as in Step 2 of the proof of Lemma
\ref{lem:full_greensfunction} to get
\begin{align*}
\norm{\nabla_\pm^\ell \bm{\Gamma}(t)}_{L^\infty_\alpha L^{p'}_\beta}
    \lesssim t^{-d/2p - 1/2}.
\end{align*}
Using \eqref{eq:Gamma_propagation}, and keeping in mind
that the finite differences act only on the index $\alpha$,
we have
\begin{align*}
\nabla_\pm^\ell \Gamma_{\alpha, \beta}(k{T_0} + t)
= \sum_{\eta \in \Z^d} \nabla_\pm^\ell  \Gamma_{\alpha, \eta}(t)
    \Gamma_{\eta,\beta}(k{T_0})\Dx^d 
= \bigl(\bigl(\nabla_\pm^\ell \bm{\Gamma}(t)\bigr)
	\dconv \bm{\Gamma}(kT_0)\bigr)_{\alpha,\beta}.
\end{align*}
From this we find, using \eqref{eq:Gamma_bound2} as before, that
\begin{align*}
\|\bm{\Gamma}\|_{L^1([kT_0,(k+1)T_0])L^\infty_\alpha L^{p'}_\beta}
\lesssim \int_0^{T_0} t^{-d/2p - 1/2}\,\d t \lesssim 1
\end{align*}
since $p>d$. Summing over $k$ yields the desired 
bound over an arbitrary time interval $[0,T]$.
\end{proof}

Similar to the Duhamel representation formula
established in Lemma \ref{lem:duhamel_representation}, we have
the following formula for solutions of \eqref{eq:heat_variable_coeff}.

\begin{lemma}[Duhamel's formula]
\label{lem:duhamel_representation2}
Fix $p,q$ satisfying $p > d$ and $1/q+d/2p<1$.
Let $\bm{f} \in L^{q}\bigl([0,T];L^p_\Dx(\R^d)\bigr)$
and $\bm{\psi}  \in L^\infty_\Dx(\R^d)$.
Then the solution $\bm{\varphi}$ of the
non-homogeneous semi-discrete parabolic
equation \eqref{eq:heat_variable_coeff} is given by
\begin{equation}\label{eq:duhamel_representation2}
	\varphi_\alpha(t)
	= \sum_{\eta \in \Z^d} \Gamma_{\alpha, \eta}(t) \psi_\eta \Dx^d
	+ \int_0^t \sum_{\eta \in \Z^d}
	\Gamma_{\alpha, \eta}(t - s) f_{\eta}(s) \Dx^d \,\d s.
\end{equation}
\end{lemma}
\begin{proof}
Recall that the fundamental solution $\bm{\Gamma}$
satisfies \eqref{eq:fund_soln_variable_heat}.
Therefore, we can compute as in the proof
of Lemma \ref{lem:duhamel_representation} to obtain
\begin{align*}
	\frac{\d}{\d t}  \phi_\alpha(t) & =\sum_{\eta \in \Z^d}\frac{\d}{\d t}
	\Gamma_{\alpha, \eta}(t)  \psi_\eta \Dx^d
	+ f_\alpha(t) + \int_0^t  \sum_{\eta \in \Z^d} \frac{\d}{\d t}
	\Gamma_{\alpha, \eta}(t - s) f_{\eta}(s) \Dx^d \,\d s\\
	& = \frac12 \sum_{\eta \in \Z^d}\sum_{j = 1}^d
	c_\alpha^j \nabla_+^j \nabla_-^j
	\Gamma_{\alpha , \eta}(t)  \psi_\eta\Dx^d
	\\*
 &\relspace + f_\alpha(t) + \frac12 \int_0^t  \sum_{\eta \in \Z^d} \sum_{j = 1}^d
	 c_\eta^j \nabla_+^j \nabla_-^j\Gamma_{\alpha, \eta}(t - s)
	 f_{\eta}(s) \Dx^d \,\d s\\
	& = f_\alpha(t) + \frac12 \sum_{j = 1}^d
	c_\alpha^j \nabla_+^j \nabla_-^j\varphi_\alpha (t).
\end{align*}
Here, we used the fact that ${\bm \Gamma}$ is a
fundamental solution, so $\sum_{\eta \in \Z^d}
\Gamma_{\alpha, \eta}(0) f_\eta = f_\alpha$.
By this same fact, as $t \to 0$,
we find $\varphi_\alpha(0) = \psi_\alpha$.
The choice of exponents $p,q$ ensures that
each series above is absolutely convergent;
cf.~the proof of Lemma \ref{lem:duhamel_representation}.
\end{proof}

\subsection{Estimates for the dual scheme}
In this subsection, we prove bounds on
the solutions to the discrete reaction-diffusion equation
\begin{equation}\label{eq:heat_variable_coeff_reaction}
\begin{cases}\displaystyle
	\frac{\d}{\d t} \varphi_\alpha - \frac12 \sum_{j = 1}^d
	c_\alpha^j \nabla_+^j \nabla_-^j  \varphi_\alpha
	+ Z_\alpha\phi_\alpha = 0, & \alpha\in\Z^d,\ t>0, \\
	\varphi_\alpha(0) = 1, & \alpha\in\Z^d,
\end{cases}
\end{equation}
where $Z_{\alpha}$ is given; this equation can be
seen as \eqref{eq:heat_variable_coeff} with a non-homogeneity
of the form $f_\alpha = -Z_\alpha \varphi_\alpha$.
This more closely mirrors the dual scheme
to which we shall apply the results of this section.

We first prove that this modified equation is well-posed and
satisfies reasonable bounds (cf.~\eqref{eq:non-homogeneous-bounds}).

\begin{lemma}[Well-posedness of dual scheme]
\label{lem:dual_phi_bounds}
Fix $p > d$ and ${\bm Z} \in L^p_\Dx(\R^d)$. 
Then there exists
a unique solution to the Cauchy problem 
\eqref{eq:heat_variable_coeff_reaction}.
The solution $\bm{\phi}$ is non-negative 
on $[0,\infty)\times\R^d$ and satisfies
\begin{subequations}\label{eq:phi-bounds}
\begin{align}
	\|\bm{\phi}\|_{L^\infty([0,T]\times \R^d)}
	&\leq C, \label{eq:phi-bound-1}\\*
	\|\nabla_\pm \bm{\varphi}\|_{L^\infty([0,T]\times \R^d)}
	&\leq C, \label{eq:phi-bound-2}
\end{align}
\end{subequations}
where $C$ only depends on $\|\bm Z\|_{L^p_\Dx}$, 
$T$ and on $\|{\bm c}\|_{L^\infty_\Dx}$.
\end{lemma}

\begin{proof}
As in the proof of Lemma \ref{lem:basic-existence} 
we can view \eqref{eq:heat_variable_coeff_reaction} 
as an abstract ODE $\frac{\d}{\d t}\bm{\phi}(t) 
	= A(\bm{\phi}(t)) - \bm{Z}\bm{\phi}(t)$. 
The right-hand side is a bounded linear operator 
(with operator norm $\lesssim \|\bm{c}\|_{L^\infty}\Dx^{-2} 
	+ \|\bm{Z}\|_{L^p}\Dx^{-d/p'}$), so the 
equation has a unique solution for all times. 
In order to obtain an expression for $\bm{\phi}$ 
with which to obtain better estimates, we carry 
out a fixed point argument. In particular, the 
solution will be unique, by the above argument.

Let $\bm{\psi}\in L^\infty_\alpha$ and 
let ${\bm \varphi}^{\langle 0\rangle}$ 
be given by the Duhamel representation
$$
\varphi^{\langle 0\rangle}_\alpha(t)
	= \sum_{\eta \in \Z^d} \Gamma_{\alpha, \eta}(t)\psi_\eta \Dx^d
	+ \int_0^t \sum_{\eta \in \Z^d}
	\Gamma_{\alpha, \eta}(t - s) Z_\eta \Dx^d \,\d s,
$$
so that by Lemma \ref{lem:duhamel_representation2},
it solves the non-homogeneous equation
$$
\frac{\d}{\d t} \varphi^{\langle 0\rangle}_\alpha
    - \frac12 \sum_{j = 1}^d
    	c_\alpha^j \nabla_+^j \nabla_-^j
     \varphi^{\langle 0\rangle}_\alpha
	+ Z_\alpha = 0,
$$
with initial condition $\varphi^{\langle 0\rangle}_\alpha(0) 
= \psi_\alpha$. For $n \in \N$, 
let ${\bm \varphi}^{\langle n \rangle}$ be given by
\begin{align*}
\varphi^{\langle n\rangle}_\alpha(t)
	= \sum_{\eta \in \Z^d} \Gamma_{\alpha, \eta}(t)\psi_\eta \Dx^d
	+ \int_0^t \sum_{\eta \in \Z^d}
	\Gamma_{\alpha, \eta}(t - s) Z_\eta \
    \varphi^{\langle n - 1\rangle}_\eta(s) \Dx^d \,\d s,
\end{align*}
so that it solves
$$
\frac{\d}{\d t} \varphi^{\langle n\rangle}_\alpha
    - \frac12 \sum_{j = 1}^d
	   c_\alpha^j \nabla_+^j \nabla_-^j
        \varphi^{\langle n \rangle}_\alpha
	+ Z_\alpha \varphi^{\langle n-1\rangle}_\alpha = 0
$$
with initial condition $\varphi^{\langle n\rangle}_\alpha(0) 
= \psi_\alpha$. For any $m, n \in \N$, and $T'>0$,
we find by H\"older's inequality that
\begin{align*}
&\bigl\|\bm{\varphi}^{\langle n\rangle} - \bm{\varphi}^{\langle m\rangle}\bigr\|_{L^\infty([0,T']\times\R^d)} \\
&\qquad =
\sup_{\substack{t \in [0,T']\\
    \alpha \in \Z^d}}
    \abs{\int_0^t \sum_{\eta \in \Z^d}
	\Gamma_{\alpha, \eta}(t - s) Z_\eta
    \bigl(\varphi^{\langle n - 1\rangle}_\eta(s)
    - \varphi^{\langle m - 1\rangle}_\eta(s)\bigr) \Dx^d \,\d s}\\
&\qquad \le
    \norm{{\bm Z}}_{L^p(\R^d)} \norm{\bm{\Gamma}}_{L^1_t([0,T'])L^\infty_\alpha L^{p'}_\beta}
        \bigl\|{\bm \varphi}^{\langle n - 1 \rangle} -
            {\bm \varphi}^{\langle m - 1\rangle}\bigr\|_{L^\infty([0,T']\times \R^d)}.
\end{align*}

Choose now $T'>0$ such that
\begin{align}\label{eq:existtime}
\norm{{\bm Z}}_{L^p(\R^d)} \norm{\bm{\Gamma}}_{L^1_t([0,T'])L^\infty_\alpha L^{p'}_\beta} < 1
\end{align}
(which is possible due to \eqref{eq:Gamma_bound4}). Then on the metric space $C([0,T'];L^\infty_\Dx(\R^d))$, the sequence $\big({\bm \varphi}^{\langle n\rangle}\big)_{n=1}^\infty$ converges to a unique
fixed point ${\bm \varphi}$ satisfying
\begin{align}\label{eq:duhamel2}
\varphi_\alpha(t)
	= \sum_{\eta \in \Z^d} \Gamma_{\alpha, \eta}(t)\psi_\eta\Dx^d
	+ \int_0^t \sum_{\eta \in \Z^d}
	\Gamma_{\alpha, \eta}(t - s) Z_\eta
    \varphi_\eta(s) \, \Dx^d \,\d s,
\end{align}
so that it solves \eqref{eq:heat_variable_coeff_reaction} with initial data $\phi_\alpha(0)=\psi_\alpha$.

The condition \eqref{eq:existtime} on $T'$
is independent of the initial condition ${\bm \psi}$.
Therefore, we may run the entire argument again with
$\psi_\alpha = \varphi_\alpha(T')$ over the interval
$[T', 2T']$, an so on. This proves that on each interval $[kT', (k+1)T']$, a solution formula of the form \eqref{eq:duhamel2} (but with $\bm\phi(kT')$ in place of $\bm\psi$) indeed holds for the unique solution to \eqref{eq:heat_variable_coeff_reaction}. Iteratively inserting $\bm\phi(kT')$ and using the propagation formula \eqref{eq:Gamma_propagation} now shows that \eqref{eq:duhamel2} holds \emph{for all} times $t>0$.

Using \eqref{eq:duhamel2}, \eqref{eq:Gamma_bound4} 
and \eqref{eq:Gamma_bound5} we get
\begin{align*}
\|\bm{\varphi}(t)\|_{L^\infty_\alpha}
&\le \|\bm{\Gamma}(t)\|_{L^\infty_\alpha L^1_\beta} \|\bm{\psi}\|_{L^\infty}
+ \int_0^t \|\bm{\Gamma}(t - s)\|_{L^\infty_\alpha L^{p'}_\beta} \,
\|{\bm Z}\|_{L^p}
\|\bm{\varphi}(s)\|_{L^\infty} \,\d s \\
&\lesssim \|\bm{\psi}\|_{L^\infty} + \|{\bm Z}\|_{L^p} \int_0^t \|\bm{\Gamma}(t-s)\|_{L^\infty_\alpha L^{p'}_\beta} \|\bm{\varphi}(s)\|_{L^\infty} \,\d s \\
\intertext{(\textit{Young's convolution inequality})}
&\leq \|\bm{\psi}\|_{L^\infty} + \|{\bm Z}\|_{L^p} \|\bm{\Gamma}\|_{L^1_t([0,t])L^\infty_\alpha L^{p'}_\beta} \int_0^t \|\bm{\varphi}(s)\|_{L^\infty} \,\d s.
\end{align*}
By \eqref{eq:Gamma_bound4} we know that $t \mapsto \|\bm{\Gamma}(t-s)\|_{L^\infty_\alpha L^{p'}_\beta}$ lies in $L^1_t$. Therefore, Gronwall's inequality gives
\begin{align*}
\|\bm\phi(t)\|_{L^\infty_\alpha} 
&\leq \|\bm\psi\|_{L^\infty}\Biggl(\begin{aligned}[t]&1 
    + \|{\bm Z}\|_{L^p} \|\bm{\Gamma}\|_{L^1_t([0,t])L^\infty_\alpha L^{p'}_\beta} \\
&\times 
\int_0^t \exp\biggl(\|{\bm Z}\|_{L^p} 
    \int_s^t \|\bm{\Gamma}\|_{L^1_t([0,\tau])L^\infty_\alpha L^{p'}_\beta}\,\d \tau \biggr)\,\d s\Biggr)\end{aligned}
\end{align*}
which implies \eqref{eq:phi-bound-1}.

In order to prove \eqref{eq:phi-bound-2} we define 
$\tilde{\bm{\phi}} \coloneqq \bm\phi - 1$, which solves
\begin{align*}
	0 &= \frac{\d}{\d t}  \tilde{\varphi}_\alpha
	+ Z_\alpha  \bk{\tilde{\varphi}_\alpha  +1}
	- \frac12 \sum_{j = 1}^d c_\alpha^j
	\nabla_+^j \nabla_-^j \tilde{\varphi}_\alpha,
\end{align*}
with initial datum $\tilde{\varphi}_\alpha(0) \equiv 0$. 
According to \eqref{eq:duhamel2}, we have
\begin{align*}
	\tilde{\varphi}_\alpha(t)
	& = \int_0^t \sum_{\eta \in \Z^d} 
	\Gamma_{\alpha, \eta}(t - s) Z_\eta 
	\bk{\tilde{\varphi}_\eta(s) + 1} \,\Dx^d\,\d s.
\end{align*}
Upon applying a difference operator in $\alpha$ 
and inserting $\phi=\tilde\phi+1$, we obtain
\begin{align*}
	\nabla_\pm \tilde{\varphi}_\alpha(t)
	&= \int_0^t \sum_{\eta \in \Z^d}\nabla_\pm
	\Gamma_{\alpha, \eta}(t - s) \,
	Z_\eta \varphi_\eta(s) \,\Dx^d\,\d s.
\end{align*}
Let $p' = p/(p - 1)$ be the H\"older dual
of $p$. By Lemma \ref{lem:full_greensfunction2}, 
and \eqref{eq:phi-bound-1}, we get
\begin{align*}
	\norm{\nabla_\pm {\bm \varphi}(t)}_{L^\infty_\alpha}
	&= \norm{\nabla_\pm {\bm \tilde{\bm \varphi}}(t)}_{L^\infty_\alpha}\\
	&\le \norm{\bm{\varphi}}_{L^\infty_t([0,T]) L^\infty_\alpha}
	\norm{\bm Z}_{L^p_\alpha}
	\int_0^t \norm{\nabla_\pm {\bm\Gamma}(t - s)}_{L^\infty_\alpha L^{p'}_\beta} \,\d  s\\
	&\lesssim  \norm{\bm Z}_{L^p_\alpha}
	\norm{\nabla_\pm {\bm\Gamma}}_{L^1_tL^\infty_\alpha L^{p'}_\beta} \lesssim 1,
\end{align*}
the last inequality following from Lemma \ref{lem:full_greensfunction2}.
This proves \eqref{eq:phi-bound-2}. 

To show that the solution is non-negative, take $\beta_n$ ($n\in\N$)
to be a smooth, convex, uniformly convergent approximation of $r \mapsto {r_-}$,
satisfying $\beta_n'(r) r = \beta_n(r) + o(1)$ uniformly as $n\to\infty$.
Using the commutator formula
\eqref{eq:2nd-order-commutator-beta},  we multiply 
\eqref{eq:heat_variable_coeff_reaction} by
$\beta_n'(\varphi_\alpha)$ to arrive at
\begin{align*}
	0 \ge \frac{\d}{\d t} \beta_n(\varphi_\alpha)
	+ Z_\alpha \abs{\beta_n(\varphi_\alpha) + o(1)}
	- \frac12 \sum_{j = 1}^d c_\alpha^j
	\nabla_-^j \nabla_+^j \beta_n(\varphi_\alpha).
\end{align*}
By taking the limit $n\to\infty$ and summing
over $\alpha \in \Z^d$, we get
\[
\frac{\d}{\d t}  \sum_{\alpha \in \Z^d} {\bk{\varphi_\alpha}_-}
\le  \norm{\bm{Z}}_{L^\infty(\R^d)}
\sum_{\alpha \in \Z^d} {\bk{\varphi_\alpha}_-}
+ \frac{2d}{\Dx^2} \norm{{\bm c}}_{L^\infty(\R^d)}
\sum_{\alpha \in \Z^d} {\bk{\varphi_\alpha}_-} .
\]
Gronwall's inequality now yields non-negativity.
\end{proof}

\subsection{The backward dual problem}
We close this section with the simple remark that
Lemma \ref{lem:dual_phi_bounds} also applies to the backward equation
\begin{equation}\label{eq:dual_discrete}
\begin{cases}\displaystyle
	\frac{\d}{\d t}  \varphi_\alpha + \frac12 \sum_{j = 1}^d
	c_\alpha^j \nabla_+^j \nabla_-^j  \varphi_\alpha
	- Y_\alpha\phi_\alpha=0, & \alpha\in\Z^d,\ t \in [0,T), \\
	\varphi_\alpha(T) = 1, & \alpha\in\Z^d
\end{cases}
\end{equation}
under the transformation $t \mapsto T-t$.
In particular, the formula
\eqref{eq:duhamel_representation2} for
\eqref{eq:dual_discrete} becomes
\begin{equation}\label{eq:duhamel_representation_backward}
	\varphi_\alpha(t)
	= \sum_{\eta \in \Z^d} \Gamma_{\alpha, \eta}(T - t) \Dx^d
	- \int_t^T \sum_{\eta \in \Z^d} \Gamma_{\alpha, \eta}(T - s)
	Y_{\eta} \varphi_\eta (s) \Dx^d \,\d s.
\end{equation}

\section{Well-posedness and stability of the scheme}\label{sec:duality}
In this section we show that our numerical scheme \eqref{eq:scheme}
is well-defined and satisfies an $L^2$ stability estimate.

\subsection{Well-posedness}
We begin by verifying that our difference scheme is well-defined.

\begin{proposition}[Solution of the difference scheme]
Under Assumption \ref{ass:main_assumption}, there exists a unique
strong solution $\bm{u}$ to the semi-discrete system \eqref{eq:scheme},
in the sense of Definition \ref{def:discrete_soln}.
\end{proposition}

\begin{proof}
We apply the well-posedness theorem of
Liu and R\"ockner \cite[Theorem 4.2.4]{Liu:2015vb}
for infinite dimensional SDEs (see also
\cite[Theorem I.3.1]{Krylov:1981aa} and
\cite[Theorem 4.2.4]{Prevot:2007aa}). Given a
Gel'fand triple of Hilbert spaces
$\mathbb{V} \subset H \subset \mathbb{V}^*$, where $\mathbb{V}$
is reflexive and $\mathbb{V} \hookrightarrow H$
continuously and densely, \cite[Theorem 4.2.4]{Liu:2015vb} yields
existence and uniqueness of SDEs
 \[
 \d X = A(t,X)\,\d t + B(t,X)\,\d W,
 \]
for a cylindrical Wiener process $W$
with values in $C([0,T];\U)$ where
$\U$ is an abstract (real) Hilbert space,
and progressively measurable maps
$A\from \Omega \times [0,T] \times \mathbb{V}  \to \mathbb{V}^*$ and
$B\from \Omega \times [0,T] \times \mathbb{V} \to L_2(\U,H)$.
Moreover, $A$ and $B$ must satisfy:
\begin{itemize}
	\item[(i)] \textit{Hemicontinuity.}
	For all $u,v,w, \in \mathbb{V}$, $t \in[0,T]$,
	$\lambda \mapsto
 	\langle A(t, u + \lambda v),
 	w \rangle_{\mathbb{V}^*, \mathbb{V}}$ is continuous,
	$\mathbb{P}$-almost surely.

	\item[(ii)] \textit{Weak monotonicity.}
	Uniformly in $u,v \in \mathbb{V}$,
 	\begin{align*}
  		&2\langle A(t,u) - A(t, v) , u - v \rangle_{\mathbb{V}^*,\mathbb{V}}
 		+ \|B(t,u) - B(t,v)\|_{L^2(\U,H)}^2 \\
		&\qquad \lesssim \|u - v\|_H^2
	 \end{align*}
	$\mathbb{P}\otimes \d t$-almost surely.

	\item[(iii)] \textit{Coercivity.}
	There exists $r >1$,
	$C_1,C_2 >0$, and $h \in L^1(\Omega
	\times [0,T])$ such that,
 	uniformly over $v \in \mathbb{V}$ and $t \in [0,T]$,
	\[
	2 \langle A(t, v),v\rangle_{\mathbb{V}^*,\mathbb{V}} + \|B(t,v)\|_{L^2(\U,H)}^2
	\le C_1 \|v\|_{H}^2 - C_2 \|v\|_\mathbb{V}^r + h(t)
	\]
	$\mathbb{P}$-almost surely.

	\item[(iv)] \textit{Boundedness.} With $r$ as in (iii),
	there exists $C_3 \geq0$, and an
	adapted $g \in L^{r/(r - 1)}
	(\Omega \times [0,T])$ such that for
	all $v \in \mathbb{V}$ and $t \in [0,T]$,
	$\|A(t,v)\|_{\mathbb{V}^*} \le g(t) +
	C_3 \|v\|_{\mathbb{V}}^{r-1}$, $\mathbb{P}$-almost surely.
\end{itemize}

For our Gel'fand triple
we choose $\mathbb{V} = H = \mathbb{V}^* = L^2_\Dx(\R^d)$.
The map $A$ for us is
\[
A(t,\bm{u})
: \alpha \mapsto - D_V u_\alpha(t)
+ \frac12 \sigma_\alpha \sum_{j=1}^d \nabla_+^j
	\bigl(\average_-^j\sigma_\alpha\nabla_-^j u_\alpha\bigr).
\]
Letting $(e_1, e_2, \ldots )$ denote
an orthonormal basis for $\U$, our map $B$ acts as follows: 
\[
B(t, \bm{u}) : \alpha \mapsto 
- \sum_{j=1}^d S_\alpha^j\nabla_0^ju_\alpha e_j
\]
By linearity, hemicontinuity (i) is immediately satisfied.
By linearity, (ii) and (iii) are equivalent. Picking $r = 2$,
condition (iii) can be verified as follows:
\begin{align*}
	\langle A(t, \bm{u}), \bm{u}\rangle_{L^2(\R^d)}
	& \le \frac1{\Dx} \|\bm{{V}}\|_{L^\infty(\R^d)}
	 \|\bm{u}(t)\|_{L^2(\R^d)}^2
	 + \frac{1}{\Dx^2}\|\bm{\sigma}\|_{L^\infty(\R^d)}^2\|\bm{u}\|_{L^2(\R^d)}^2, \\
	\|B(t,\bm{u})\|_{L^2(\R^d)}^2
	& \le \frac1\Dx \|\bm{\sigma}\|_{L^\infty(\R^d)}
	\|\bm{u}(t)\|_{L^2(\R^d)}^2.
\end{align*}
For us, $h$ and $g$ in conditions (iii) and (iv),
respectively, are simply the zero function.
As $\mathbb{V}^* = L^2_{\Dx}(\R^d)$,
we find that (iv) is verified by the calculations
foregoing. This establishes the proposition.
\end{proof}

\subsection{Stability of the scheme}
In Section \ref{sec:tech_outline}, we outlined a
Holmgren-type duality approach for
establishing the $L^2$ stability of solutions
to \eqref{eq:stratonovich_eq}. In Section \ref{sec:greens_functions},
we analysed a dual difference scheme,
culminating with Lemma \ref{lem:dual_phi_bounds}
and that its solutions exhibit
spatial differences that are bounded independently of $\Dx$.
We now bring these ingredients
together, showing that the $L^2$ norm
of the difference approximations
\eqref{eq:scheme} can be estimated by
only an $L^p$ bound ($p>d$) on $\Div V$,
rather than the typically assumed $L^\infty$ bound.
The main result of this section is the following, which
proves part (i) of the main
theorem (Theorem~\ref{thm:main_conv}).

\begin{theorem}[$L^2$ stability]\label{thm:disc_energyestm}
If Assumption \ref{ass:main_assumption} holds, and if $\bm{u}$ is
computed by the scheme \eqref{eq:scheme}, then for every $t \in [0,T]$,
\begin{align}\label{eq:energy_estimate1}
	\Ex\sum_{\alpha \in \Z^d} \abs{u_\alpha(t)}^2 \Dx^d
	\leq C \sum_{\alpha \in \Z^d} |u^0_\alpha|^2\Dx^d,
\end{align}
where $C>0$ depends only on $T$, $\norm{\sigma}_{W^{2,\infty}(\R^d)}$,
$\norm{V}_{L^\infty(\R^d)}$, and $\mathopen\|\Div V\|_{L^p(\R^d)}$.
\end{theorem}

\begin{proof}
Recall the formulation \eqref{eq:scheme}
of our numerical scheme. By the It\^o
formula applied to $u_\alpha^2$, we find
\begin{equation}\label{eq:ito_eq2}
\begin{split}
	\d u_\alpha^2
	&+ 2u_\alpha D_V u_\alpha\,\d t
	+ 2u_\alpha S_\alpha\nabla_0u_\alpha \cdot \d W\\
	&= u_\alpha\sigma_\alpha
	\sum_{j=1}^d \nabla_+^j\bigl(\average_-^j\sigma_\alpha
	\nabla_-^ju_\alpha\bigr)\,\d t
	+ |S_\alpha\nabla_0u_\alpha|^2 \,\d t \\
	&=\frac12\sigma_\alpha \sum_{j=1}^d
	 \nabla_+^j\bigl(\average_-^j\sigma_\alpha
	 \nabla_-^j u_\alpha^2\bigr)\, \d t
	 + \mathcal{E}_\alpha^\sigma \,\d t,
\end{split}
\end{equation}
where
\begin{align*}
	\mathcal{E}^\sigma_\alpha &\coloneqq u_\alpha\sigma_\alpha
	\sum_{j=1}^d \nabla_+^j\bigl(\average_-^j\sigma_\alpha
	\nabla_-^j u_\alpha\bigr)
	-\frac12\sigma_\alpha \sum_{j=1}^d \nabla_+^j
	\bigl(\average_-^j\sigma_\alpha\nabla_-^j
	u_\alpha^2\bigr)
	+ |S_\alpha\nabla_0u_\alpha|^2 \\
	&= |S_\alpha\nabla_0u_\alpha|^2 \\
	&\relspace + \sigma_\alpha\sum_{j=1}^d
	\biggl( u_\alpha \frac{\average_+^j\sigma_\alpha\nabla_+^j u_\alpha
	- \average_-^j\sigma_\alpha\nabla_-^j u_\alpha}{\Dx}
	- \frac{\average_+^j\sigma_\alpha\nabla_+^j u_\alpha^2
	- \average_-^j\sigma_\alpha \nabla_-^j u_\alpha^2}{2\Dx}\biggr) \\
	& = |S_\alpha\nabla_0u_\alpha|^2\\
	&\relspace +\sigma_\alpha\sum_{j=1}^d
	\biggl( \frac{\average_+^j\sigma_\alpha
	\bigl(u_\alpha\nabla_+^j u_\alpha
	-\tfrac12\nabla_+^j u_\alpha^2\bigr)
	- \average_-^j\sigma_\alpha\bigl(u_\alpha\nabla_-^j u_\alpha
	- \tfrac12 \nabla_-^j u_\alpha^2\bigr)}{\Dx}\biggr)  \\
	&= -\frac{\sigma_\alpha}{2}\sum_{j=1}^d
	\biggl( \average_+^j\sigma_\alpha\bigl(\nabla_+^j u_\alpha\bigr)^2
	+ \average_-^j\sigma_\alpha\bigl(\nabla_-^j u_\alpha\bigr)^2\biggr)
	+ |S_\alpha\nabla_0u_\alpha|^2 \\
	&= -\frac{\sigma_\alpha}{2}\sum_{j=1}^d
	\Biggl( \average_+^j\sigma_\alpha\bigl(\nabla_+^j u_\alpha\bigr)^2
	+ \average_-^j\sigma_\alpha\bigl(\nabla_-^j u_\alpha\bigr)^2
	- 4\frac{\average^j_+\sigma_\alpha\average^j_-\sigma_\alpha}
	{\average^j_+\sigma_\alpha+\average^j_-\sigma_\alpha}
	\bigl(\nabla_0^j u_\alpha\bigr)^2\Biggr).
\end{align*}
Writing $\nabla_0^j u_\alpha
=\tfrac12(\nabla_+^ju_{\alpha}+\nabla_-^j u_\alpha)$,
we apply the elementary inequality
$(a + b)^2 \le (1 + \ep) a^2 + (1 + 1/\ep) b^2$
for any $\ep > 0$ to get
\begin{align*}
	& 4\frac{\average^j_+\sigma_\alpha\average^j_-\sigma_\alpha}
	{\average^j_+\sigma_\alpha+\average^j_-\sigma_\alpha}
	\bigl(\nabla_0^j u_\alpha\bigr)^2
	\\ & \qquad
	\leq \frac{\average^j_+\sigma_\alpha\average^j_-\sigma_\alpha}
	{\average^j_+\sigma_\alpha+\average^j_-\sigma_\alpha}
	\Bigl(\bigl(1+\eps\bigr)\bigl(\nabla_+^j u_\alpha\bigr)^2
	+ \bigl(1+\eps^{-1}\bigr)\bigl(\nabla_-^j u_\alpha\bigr)^2 \Bigr) \\
	\intertext{(\textit{setting} $\eps =
	\frac{\average_+^j\sigma_\alpha}{\average_-^j\sigma_\alpha}$)}
	& \qquad
	= \average_+^j\sigma_{\alpha}\bigl(\nabla_+^j u_\alpha\bigr)^2
	+ \average_-^j\sigma_\alpha\bigl(\nabla_-^j u_\alpha\bigr)^2.
\end{align*}
Inserting this into the computation of $\mathcal{E}_\alpha^\sigma$
and using the fact that $\sigma$ is strictly positive, we conclude that
\begin{align}\label{eq:error_main}
\mathcal{E}^\sigma_\alpha \leq 0.
\end{align}

Taking expectation in \eqref{eq:ito_eq2},
we find (in differential form) the ODE
\begin{align*}
	\frac{\d}{\d t}  \Ex u_\alpha^2
	&= -\Ex\bigl( D_V (u_\alpha^2)\bigr) + \frac12\sigma_\alpha \sum_{j=1}^d
	 \nabla_+^j\bigl(\average_-^j\sigma_\alpha\nabla_-^j \Ex u_\alpha^2\bigr)
	 +\Ex \mathcal{E}^\sigma_\alpha + \Ex \mathcal{E}^V_\alpha
\end{align*}
where
\begin{align}\label{eq:error_3}
	\mathcal{E}^V_\alpha
	&\coloneqq D_V (u_\alpha^2)
	- 2u_\alpha D_V u_\alpha
	\leq 0,
\end{align}
the inequality following from Lemma
\ref{lem:discrete_calculus} \ref{lem:discrete_calculus_upwind-chain-rule}.
Multiplying by a non-negative
${\bm \varphi} \in C^1([0,T];L^\infty_\Dx(\R^d))$ and
then integrating in time and summing over $\alpha$, we get
\begin{align*}
	&\nqquad - \int_0^T \sum_{\alpha \in \Z^d}
	\frac{\d}{\d t} \varphi_\alpha \Ex\bigl(u_\alpha^2\bigr) \,\d t
	+ \sum_{\alpha \in \Z^d}
	\varphi_\alpha \Ex\bigl(u_\alpha^2\bigr)\bigg|_{t=0}^T \\
	&\le - \int_0^T \sum_{\alpha \in \Z^d}
	\varphi_\alpha D_V \Ex\bigl(u_\alpha^2\bigr) \,\d t\\
	&\relspace + \frac12 \int_0^T\sum_{\alpha \in \Z^d}
	\varphi_\alpha \sigma_\alpha \sum_{j=1}^d
	\nabla_+^j\bigl(\average_-^j\sigma_\alpha
	\nabla_-^j \Ex \bigl( u_\alpha^2\bigr) \bigr)\, \d t.
\end{align*}
Applying the summation by parts
rules in Lemma \ref{lem:discrete_calculus}
\ref{lem:discrete_calculus_int-by-parts},
\ref{lem:discrete_calculus_upwind-int-by-parts}
and using \eqref{eq:error_main}
and \eqref{eq:error_3}, we obtain
\begin{equation}\label{eq:duality1}
\begin{split}
	\sum_{\alpha \in \Z^d} \varphi_\alpha
	\Ex\bigl(u_\alpha^2\bigr) \bigg|_0^T \leq  \int_0^T\sum_{\alpha \in \Z^d}
	\bigg(\begin{aligned}[t]&\frac{\d}{\d t}  \varphi_\alpha
	+ D_V' \varphi_\alpha \\
	&+\frac12  \sum_{j=1}^d \nabla_+^j \bigl(
	\average_-^j\sigma_\alpha \nabla_-^j
	\bigl(\sigma_\alpha \varphi_\alpha  \bigr)\bigr) \bigg)
	\Ex\bigl(u_\alpha^2\bigr)  \,\d t,\end{aligned}
\end{split}
\end{equation}
where $D_V'$ is the dual upwind
operator \eqref{eq:dual_upwind}. From Lemma \ref{lem:discrete_calculus}~\ref{lem:discrete_calculus_product-rule}, we have
\begin{align*}
	\nabla_+^j \bigl(
	\average_-^j\sigma_\alpha \nabla_-^j
	\bigl(\sigma_\alpha \varphi_\alpha  \bigr)\bigr)
	& = \nabla_+^j \bigl(
	\average_-^j\sigma_\alpha \nabla_-^j
	\sigma_\alpha \varphi_\alpha  \bigr)
	+ \nabla_+^j \bigl(\sigma_{\alpha - e_j}
	\average_-^j\sigma_\alpha
	 \nabla_-^j\varphi_\alpha \bigr) \\
	 &=  \nabla_+^j \bigl(
	\average_-^j\sigma_\alpha \nabla_-^j \sigma_\alpha
	\bigr)\varphi_\alpha
	 +  \average_+^j\sigma_\alpha \nabla_+^j
	 \sigma_\alpha \nabla_+^j \varphi_\alpha \\
	 &\quad\,\,	+ \nabla_+^j \bigl(\sigma_{\alpha - e_j}
	 \average_-^j \sigma_\alpha \bigr) \nabla_-^j \varphi_\alpha
	+  \sigma_\alpha \average_+^j\sigma_\alpha
	 \nabla_+^j  \nabla_-^j\varphi_\alpha .
\end{align*}
(This is a discrete version of the product rule
applied to $\partial_{x^j}(\sigma\partial_{x^j}(\sigma\phi))$.)
Putting this into \eqref{eq:duality1} then yields
\begin{equation}\label{eq:duality2}
\begin{split}
	&\sum_{\alpha \in \Z^d} \varphi_\alpha
	\Ex\bigl(u_\alpha^2\bigr) \bigg|_0^T \\
	&\leq  \int_0^T\sum_{\alpha \in \Z^d} \bigg(\frac{\d}{\d t}  \varphi_\alpha
	+ D_V'(1)_\alpha\varphi_\alpha
	+ \frac12\sum_{j = 1}^d \sigma_\alpha \average_+^j\sigma_\alpha
	\nabla_+^j  \nabla_-^j \varphi_\alpha\bigg)
	\Ex\bigl(u_\alpha^2\bigr) \,\d t  \\
	&\relspace
	+ \int_0^T \sum_{\alpha \in \Z^d}
	\biggl(\begin{aligned}[t]&\bigl(D'_V(\phi_\alpha)
	- \phi_\alpha D_V'(1)_\alpha\bigr) +\frac12 \sum_{j = 1}^d\nabla_+^j \bigl(
	\average_-^j\sigma_\alpha \nabla_-^j \sigma_\alpha
	\bigr)\varphi_\alpha \\
	&+\frac12 \sum_{j = 1}^d
	\bigr( \average_+^j\sigma_\alpha \nabla_+^j
	 \sigma_\alpha \nabla_+^j \varphi_\alpha  + \nabla_+^j
	 \bigl(\sigma_{\alpha - e_j}
	 \average_-^j \sigma_\alpha \bigr) \nabla_-^j \varphi_\alpha\bigl)
	\biggr)\Ex\bigl(u_\alpha^2\bigr) \,\d t.\end{aligned}
\end{split}
\end{equation}

We now pick ${\bm \phi}$ to satisfy
the dual difference scheme
\begin{equation*}
\begin{cases}
\displaystyle
\frac{\d}{\d t}  \varphi_\alpha
	+ D_V'(1)_\alpha \varphi_\alpha
	+ \frac12\sum_{j = 1}^d \sigma_\alpha \average_+^j\sigma_\alpha
	\nabla_+^j  \nabla_-^j \varphi_\alpha=0, & \alpha\in\Z^d,\ t \in [0,T), \\
	\varphi_\alpha(T) = 1, & \alpha\in\Z^d,
\end{cases}
\end{equation*}
that is, \eqref{eq:duhamel_representation_backward} with
$Z = -D_V'(1)_\alpha$ and $c_\alpha^j
=\sigma_\alpha \average_+^j\sigma_\alpha$. Such a function ${\bm \varphi}$
exists and is non-negative by Lemma \ref{lem:dual_phi_bounds}. We obtain
\begin{equation*}
\begin{aligned}
	&\sum_{\alpha \in \Z^d} \varphi_\alpha \Ex\bigl(u_\alpha^2\bigr) \bigg|_0^T
	\\ & \qquad
	\le  \int_0^T \begin{aligned}[t]&\sum_{\alpha \in \Z^d} \Biggl(\bigl(D'_V\phi_\alpha
	- \phi_\alpha D_V'(1)_\alpha\bigr)+\frac12 \sum_{j = 1}^d
	 \nabla_+^j \bigl(\average_-^j\sigma_\alpha \nabla_-^j
	 \sigma_\alpha\bigr)\varphi_\alpha \\
	 & +\frac12 \sum_{j = 1}^d
	 \bigg( \average_+^j\sigma_\alpha \nabla_+^j
	 \sigma_\alpha \nabla_+^j \varphi_\alpha  + \nabla_+^j
	 \bigl(\sigma_{\alpha - e_j}
	 \average_-^j \sigma_\alpha \bigr) \nabla_-^j \varphi_\alpha\bigg)
	\Biggr)\Ex\bigl(u_\alpha^2\bigr) \,\d t.\end{aligned}
\end{aligned}
\end{equation*}
Using the product rule in Lemma
\ref{lem:discrete_calculus}
\ref{lem:discrete_calculus_upwind-product-rule}, we can estimate 
\begin{align*}
\bigl|D'_V\phi_\alpha - \phi_\alpha D_V'(1)_\alpha\bigr|
	&\lesssim \|\bm{V}\|_{L^\infty(\R^d)} \|\nabla_+ \bm{\varphi}\|_{L^\infty(\R^d)} ,
\end{align*}
where $\|\nabla {\bm \varphi}\|_{L^\infty (\R^d)} \lesssim C$ for a constant $C$ dependent on $\|D'_V(1)\|_{L^p(\R^d)}$ and $\bigl\|D_V'(1) - \Div V\|_{L^p(\R^d)} \to 0$ by the upcoming Lemma \ref{lem:V_DV_discretelimits1}.
Similarly,
\begin{align*}
 	\abs{ \sum_{j = 1}^d \bk{\average_+^j\sigma_\alpha \nabla_+^j
	 \sigma_\alpha \nabla_+^j \varphi_\alpha  + \nabla_+^j
	 \bigl(\sigma_{\alpha - e_j}
	 \average_-^j \sigma_\alpha \bigr) \nabla_-^j \varphi_\alpha}}
	 \lesssim \|\sigma\|_{W^{1,\infty}(\R^d)} 
     \|\nabla_+ \bm{\varphi}\|_{L^\infty(\R^d)},
\end{align*}
and, by the first bound in
Lemma \ref{lem:dual_phi_bounds},
\begin{align*}
	\sum_{j = 1}^d\nabla_+^j \bigl(
	\average_-^j\sigma_\alpha \nabla_-^j \sigma_\alpha
	\bigr)\varphi_\alpha
	\lesssim \|\sigma\|_{W^{2,\infty}(\R^d)}
	 \|\nabla_+ \bm{\varphi}\|_{L^\infty(\R^d)}.
\end{align*}
All together, these estimates yield
\begin{align*}
\Ex  \sum_{\alpha \in \Z^d} \varphi_\alpha(T)  u_\alpha(T)^2 \Dx^d
&\leq \sum_{\alpha \in \Z^d} \varphi_\alpha(0)  u_\alpha(0)^2 \Dx^d
+ C \Ex \sum_{\alpha \in \Z^d} \int_0^T u_\alpha(s)^2 \Dx^d \,\d s,
\end{align*}
where the constant $C$ depends only on $\sigma$
and $V$. Recalling that $\phi_\alpha(T)\equiv 1$ and that $\|\phi\|_{L^\infty} \lesssim 1$,
we can apply Gronwall's lemma and conclude the proof.
\end{proof}

\subsection{Alternative discretisation schemes}\label{sec:alternative-discretisation}
We have chosen the discretisation
scheme \eqref{eq:scheme} as one in a range
of possible ones. Let us discuss other possibilities
for discretising the noise term $\sigma \nabla u \circ \d W$.
To effectively derive an $L^2$ stability result,
a crucial aspect of any discretisation is that
the correction term arising from the It\^o
formula in the computation of the energy
$u^2_{\alpha}$ must be controlled effectively
by the dissipation provided by the discretised
Stratonovich-to-It\^o conversion term---the
last term in \eqref{eq:scheme1}.

Consider two ``abstract'' operators
$D_{\sigma, \Dx}^{(1)}$ and $D_{\sigma,\Dx}^{(2)}$
that approximate the first and second order
differential operators $\sigma \nabla u$  and
$\sigma \nabla \cdot \bk{\sigma \nabla u}$,
respectively, so that $\sigma \nabla u \cdot \d W
- \frac12 \sigma \,\nabla \cdot \bk{\sigma \nabla u}\,\d t$
in \eqref{eq:ito_eq} is approximated by
\begin{align}\label{eq:new-scheme}
	D_{\sigma, \Dx}^{(1)} u_\alpha \,\d W
	- \frac12 D_{\sigma,\Dx}^{(2)} u_\alpha\,\d t,
\end{align}
that is, the two final terms of the numerical
scheme \eqref{eq:scheme1}
is replaced by \eqref{eq:new-scheme}.
We now outline conditions---or design guidelines---that need
to be imposed on $D_{\sigma, \Dx}^{(1)}$
and $D_{\sigma,\Dx}^{(2)}$ to ensure
the $L^2$ stability of the resulting scheme.

The terms in the equation for $\d u_{\alpha}^2$
in which $\sigma$ is involved are given by
$$
u_\alpha D_{\sigma,\Dx}^{(1)} u_\alpha\,\d W
+  \frac12 \abs{D_{\sigma,\Dx}^{(1)} u_\alpha}^2\,\d t
- \frac12 u_{\alpha }D_{\sigma,\Dx}^{(2)} u\,\d t,
$$
where the middle term is the It\^o correction.
Let us also introduce an ``abstract'' dual
operator $D^{(2),*}_{\sigma,\Dx}$, formally defined by
\[
\sum_{\alpha \in \Z^d}
f_\alpha D^{(2), *}_{\sigma,\Dx} g_{\alpha}
=
\sum_{\alpha \in \Z^d}
g_\alpha D^{(2)}_{\sigma,\Dx} f_{\alpha}
\qquad \forall\ \bm{f},\bm{g}\in L^2_\Dx(\R^d).
\]
In the duality formulation, a test function satisfying
a dual difference scheme is introduced,
see \eqref{eq:duality1} and \eqref{eq:duality2}.
With the new abstract discretisation operators,
we require that for a discrete, non-negative function $\bm{\phi}$,
\begin{align}\label{eq:dual_condition1}
	\sum_{\alpha\in \Z^d}\varphi_{\alpha} \biggl(
	\abs{D_{\sigma,\Dx}^{(1)} u_\alpha}^2
	+ u_{\alpha }D_{\sigma,\Dx}^{(2)} u_\alpha
	- \frac12  D_{\sigma,\Dx}^{(2)}
	u_{\alpha}^2 \biggr) \le 0.
\end{align}
The continuous analogue of the foregoing
is identically zero for sufficiently smooth
functions. In the stability proof, we choose $\bm{\phi}$ to solve
the dual scheme that arises, which is a
(backward) parabolic difference scheme of the form
\begin{align}\label{eq:dual_condition2}
	0=\frac{\d}{\d t} \varphi_{\alpha}
	+\frac12 \tilde{D}_{\sigma,\Dx}^{(2),*} \varphi_{\alpha}
	+V_{\alpha} \cdot \nabla'_u \varphi_{\alpha},
	\quad \varphi_{\alpha}(T)=1, \quad \alpha\in \Z^d,
\end{align}
where $\tilde{D}_{\sigma,\Dx}^{(2),*}$
(with a tilde) is a possibly simpler operator
than $D_{\sigma,\Dx}^{(2),*}$; in \eqref{eq:duality2},
this simpler operator was obtained by
``differencing out'' the $\sigma$-related
coefficient from the second-order difference
operator in \eqref{eq:duality1}. More
generally, the simplified operator
$\tilde{D}_{\sigma,\Dx}^{(2),*}$
must be chosen so that the solution
$\varphi_{\alpha}$ of \eqref{eq:dual_condition2}
is non-negative; see Lemma \ref{lem:dual_phi_bounds}.
It may be that \eqref{eq:dual_condition1}
only holds for a particular class of
sequences, such as those that are non-negative,
or are uniformly bounded
away from zero. In this case, one would need to
choose $\tilde{D}_{\sigma,\Dx}^{(2)}$ so that
the solution to \eqref{eq:dual_condition2}
belongs to this class. We also require that
the simplified operator approximates the full
operator $D_{\sigma,\Dx}^{(2),*}$ when
applied to the solution $\bm{\varphi}$ of
\eqref{eq:dual_condition2}, in the sense that
$$
\norm{D_{\sigma,\Dx}^{(2), *}
{\bm \varphi }-\tilde{D}_{\sigma,\Dx}^{(2),*}
{\bm \varphi}}_{L^1([0,T];L^\infty(\R^d))}
\lesssim 1.
$$

The expression in the brackets in \eqref{eq:dual_condition1}
with $\sigma \equiv 1$, when set to nought, is
a discretised version of the fundamental commutator formula
\begin{align*}
	\beta'(u)\Delta u-\Delta \beta(u) = - \beta''(u)\abs{\nabla u}^2,
\end{align*}
where e.g.~$\beta(u)=\frac12 u^2$.
With $\sigma \equiv 1$,
this is easy to  replicate at the discrete level
for $\Delta\approx \nabla_-\nabla_+$
(both derivatives being in the same direction)
and $\nabla\approx \nabla_\pm$.
Here $\beta''(u) \abs{\nabla u}^2$ is the It\^o correction
term representing energy pumped into
the system by noise, and $\beta'(u)\Delta u$ is $\beta'(u)$ times the
Stratonovich-to-It\^o conversion term $\Delta u$. The
term $\Delta \beta(u)$ above integrates to nought
and does not enter the energy balance.
Specifying $\beta(u)=\frac12 u^2$ in
the formula \eqref{eq:2nd-order-commutator-beta}
gives exact cancellation as in the continuous setting:
\begin{equation*}
	u_\alpha\nabla_-\nabla_+ u_\alpha
	-\nabla_-\nabla_+ \left(\frac12 u_\alpha^2\right)
	=-\frac{1}{2}\abs{\nabla_+ u_\alpha}^2
	-\frac{1}{2}\abs{\nabla_- u_\alpha}^2.
\end{equation*}
For $\sigma \in W^{2,\infty}$ more generally,
our discretisation \eqref{eq:scheme} with
$$
D^{(1)}_{\sigma,\Dx} u_\alpha
= S_{\alpha} \nabla_\alpha u_{\alpha},
\qquad
D^{(2)}_{\sigma, \Dx} u_{\alpha}
=\frac12 \sigma_\alpha \sum_{j=1}^d \nabla_+^j
\bigl(\average_-^j\sigma_\alpha\nabla_-^j u_\alpha\bigr)
$$
achieves \eqref{eq:dual_condition1}, thanks to
\eqref{eq:error_main} and the non-negativity of
the solution $\varphi_\alpha$ to the dual
scheme (Lemma \ref{lem:dual_phi_bounds}).

Other methods for controlling the
It\^o correction term include incorporating
an artificial viscosity that diminishes
as the grid size $\Dx$ approaches zero.
This can be achieved through various discretisations
of $\nabla u \cdot \dot W-\frac12 \Delta u$, such as
\begin{align*}
	&\sqrt{1-\ep_\Dx} \nabla_0 u_\alpha \cdot \dot W
	 - \frac12\nabla_0^2 u_\alpha\\
	\intertext{or}
	&\nabla_0 u_\alpha \cdot \dot W -
	\frac{1}{8}\left(\nabla_+^2
	+2\bigl(1+2\ep_\Dx\bigr)\nabla_+\nabla_-
	+\nabla_-^2 \right)u_\alpha,
\end{align*}
where $\ep_\Dx$ is a positive parameter that vanishes
as $\Dx \to 0$.
Of course, in place of the central difference
$\nabla_0$, it is also possible retain $\nabla_\pm$
(and $\nabla_\pm \nabla_\mp$ for the discrete Laplacian).
It seems that in these schemes for which the simplest choice
$\tilde{D}^{(2),*}_{\sigma, \Dx} = D^{(2),*}_{\sigma, \Dx}$
is made, the solution of the dual scheme
\eqref{eq:dual_condition2} must be
uniformly bounded away from zero
on $[0,T] \times \Z^d$ in order to achieve
\eqref{eq:dual_condition1}.  Establishing a uniform lower bound
is a challenging yet attainable task,
requiring a weak Harnack inequality for difference schemes.

\subsection{Time discretisation}

Time discretisation falls outside
the scope of our present investigation, but
it naturally extends the discussions here.
The simplest time discretisation is the
Euler(--Maruyama) explicit approximation. In
the abstract discretisation notation and using
superscript to denote discrete times, the
discretisation $u^{n+1}_{\alpha}-u^n_{\alpha}$
of the It\^{o} differential $\d u$
includes a ``stochastic part'' composed of the following terms:
$$
D_{\sigma,\Dx}^{(1)} u_\alpha^n \bk{W^{n + 1}-W^n}
+ \frac12 D_{\sigma,\Dx}^{(2)} u_{\alpha}^n \Delta t,
$$
where $\Dt$ is the time step. We require that  the $\R^d$-valued
increments $W^{n+1}-W^n$ be independent and
centred Gaussian variables with variance
$\Ex\abs{ W^{n+1}-W^n}^2=d \Delta t$.
Implicit and explicit higher order methods
can involve multiple It\^o integrals, and can give higher
orders of convergence.  We firmly expect that the $L^2$
analysis remains applicable to (first-order) explicit
and implicit time discretisations.
An avenue for future research would be to explore
the potential of using the arguments of this paper
to establish $L^2$ stability for higher-order temporal
and spatial discretisations.

\section{Convergence of the scheme}\label{sec:convergence_of_scheme}

The Lax--Richtmyer theorem for
PDEs teaches us to expect that
consistent and stable numerical schemes
will converge to the exact solution.
We have successfully established the
stability criterion (the hard part).
The consistency requirement is straightforward, as illustrated in
the subsequent two lemmas. The remaining portion of
the section (see Lemma \ref{lem:weak_compactness}
and Proposition \ref{thm:equation_wklim}) is dedicated to
substantiating the convergence assertion.

\subsection{Consistency lemmas}

We present two lemmas showing the convergence
of specific discretised functions
towards their continuous counterparts.

\begin{lemma}[Consistency of velocity approximations]
\label{lem:V_DV_discretelimits1}
Suppose $V\in L^p(\R^d;\R^d)$ for $p\in [1,\infty)$.
\begin{enumerate}[label=(\roman*)]
	\item Let $\bm{V}_\Dx\coloneqq P_\Dx V$,
	with the projection operator $P_\Dx$
	defined in \eqref{eq:projection}. Then
	\[
	\hspace{10000pt minus 1fil}
	\bm{V}_\Dx \to V \qquad \text{pointwise a.e.~and
	in $L^p$ as $\Dx\to0$.} \hfilneg
	\]

	\item Assume also that $\partial_{x^j}V^j \in L^p(\R^d)$ for every
	$j=1, \ldots, d$, where $\partial_{x^j} V^j$ is the
	weak derivative of $V^j$ in the $j$th direction. Then
	\[
	\hspace{10000pt minus 1fil}
	 \sum_{\alpha \in \Z^d}D_V'(1)_\alpha \one{\cell_\alpha} \to \Div V \,\,
	\text{pointwise a.e.~and in $L^p$ as $\Dx\to0$,}\hfilneg
	\]
	where $D_V'(1)$ is
	given by \eqref{eq:dual_upwind2}.
	\item With the assumptions of (ii), for $\varphi \in C^1_c(\R^d)$,
	\[
	\hspace{10000pt minus 1fil}
	\sum_{\alpha \in \Z^d} D'_V(P_\Dx \varphi)_\alpha \one{\cell_\alpha}
		\to \varphi \Div V + V \cdot \nabla \varphi \,\,
	\text{pointwise a.e.~and in $L^p$ as $\Dx\to0$.}\hfilneg
	\]
\end{enumerate}
\end{lemma}

\begin{proof}[Proof of (i)]
The fact that $\bm{V}_\Dx\to V$ a.e.~follows
from Lebesgue's differentiation theorem
(see, e.g.,~\cite[Theorem 3.21]{Folland:1999aa}).
Let $\cell(x)$ be the cell in which a point $x\in\R^d$ lies. Since $\cell(x)\subset B_r(x)$, where $r=\sqrt{d}\Dx$, we get
\begin{align*}
\|\bm{V}_\Dx-V\|_{L^p}^p &\leq \int_{\R^d}\fint_{\cell(x)}|V(y)-V(x)|^p\,\d y\,\d x \\
&\leq \frac{1}{|\cell_0|}\int_{\R^d}\int_{B_r(x)}|V(y)-V(x)|^p\,\d y\,\d x \\
&= \frac{1}{|\cell_0|}\int_{B_r(0)}\int_{\R^d}|V(x+z)-V(x)|^p\,\d x\,\d z
\\
&\lesssim \fint_{B_r(0)} \|V(\cdot+z)-V\|_{L^p(\R^d)}^p\,\d z,
\end{align*}
which vanishes as $r\simeq \Dx \to 0$.
\end{proof}

\begin{proof}[Proof of (ii)]
We have
\[
\frac{V^j(x + \Dx e_j) - V^j(x)}{\Dx}
= \fint_0^\Dx \partial_{x^j} V^j(x + h e_j) \,\d h.
\]
Taking the average of the above we find
\begin{align*}
	\nabla_+^j V^j_\alpha
	=\frac{V^j_{\alpha+ e_j} - V^j_{\alpha }}{\Dx}
	=\fint_{\cell_\alpha} \fint_0^\Dx \partial_{x^j} V^j(y + h e_j)\,\d h\,\d y.
\end{align*}
By an argument as in part (i) of the theorem,
$P_{\Dx } \partial_{x^j} V^j
\to \partial_{x^j} V^j$ a.e.~and in $L^p(\R^d)$.
For an $x\in\R^d$, let $\cell_\alpha$ be the unique box containing $x \in \R^d$. Then
\begin{align*}
	&\abs{P_{\Dx} \partial_{x^j} V^j(x) - \nabla_+^j V^j_\alpha(x) }\\
	&\qquad
 	=\abs{ \fint_{\cell_\alpha} \fint_0^\Dx
 	\partial_{x^j} V^j (y + he_j) - \partial_{x^j} V^j(y)\,\d h\,\d y}
	\le  I_1+I_2,
\end{align*}
where
\begin{align*}
	I_1 &\coloneqq\abs{\fint_{0}^\Dx \fint_{\cell_\alpha}
	 \partial_{x^j} V^j (y + h e_j) - \partial_{x^j} V^j (x ) \,\d y\,\d h}, \\
	I_2&\coloneqq\abs{\fint_{\cell_\alpha}
	 \partial_{x^j} V^j (x) - \partial_{x^j} V^j (y) \,\d y}.
\end{align*}
These two integrals tend separately to
nought. Indeed, by the Lebesgue differentiation
theorem, for a.e.~$x$ and in $L^p(\R^d)$, $I_2 \to 0$
as $\Dx \to 0$. In $I_1$, as $y$ varies over
$\cell_\alpha$, $z = y - x$ varies
over $\cell_{\alpha} - x$, which is contained in the box
$\mathcal{Q}$ of size $3^d \Dx^{d}$ consisting of the union of
all boxes whose closure intersect the closure of $\cell_0$.
Thus
\begin{align*}
	I_1 &\le \fint_0^{\Dx} \fint_{\mathcal{Q}} \abs{\partial_{x^j} V^j(x + z + he_j)
	-\partial_{x^j} V^j(x)}\,\d z\,\d h
	\\ & \le \underbrace{\fint_0^\Dx \,\d h}_{ = 1}
	\fint_{\mathcal{Q} \cup \mathcal{Q} + \Dx e_j}
	\abs{\partial_{x^j} V^j(x + z) - \partial_{x^j} V^j(x)}\,\d z.
\end{align*}
The second inequality holds as $z + he_j$
is contained in $\mathcal{Q} \cup \mathcal{Q} + \Dx e_j$ as
$z$ varies over $\mathcal{Q}$ and $h$ varies over $[0,\Dx]$.
By the Lebesgue differentiation theorem,
this final integral tends to nought for a.e.~$x$ and in $L^p(\R^d)$.
Therefore $\nabla_+^j V^j_\alpha$ converges
to $\partial_{x^j} V^j$
pointwise a.e.~and in $L^p(\R^d)$.

We shall show now that
\begin{equation}\label{eq:diffquot_upwind}
	\sum_{\alpha \in \Z^d} \frac{V^{j,\pm}_{\alpha \pm e_j}
	-V^{j,\pm}_\alpha}\Dx
	\one{\cell_\alpha}
	\to \pm \partial_{x^j} V^j \one{\{\pm V^j \ge 0\}}
	\quad \text{a.e.~in }\R^d.
\end{equation}
Let us only prove \eqref{eq:diffquot_upwind} for the ``$+$'' part,
that is, $\nabla_+^j V^{j,+}_\alpha \to \partial_{x^j} V^j \one{\{V^j \geq0\}}$.
Since $|\nabla_+^j V^{j,+}_\alpha| \leq |\nabla_+^j V^j_\alpha|$ pointwise,
and the latter converges in $L^p$, the pointwise
convergence \eqref{eq:diffquot_upwind} will imply $L^p$ convergence.
We decompose
\begin{equation*}
\begin{aligned}
	\frac{V^{j,+}_{\alpha + e_j} - V^{j, + }_\alpha}{\Dx}
	&=  \nabla_+^j V^j_\alpha\,
	\one{\{V^j_{\alpha + e_j}, V^j_\alpha \ge 0\}} \\
	&\relspace+ \frac{ V^j_{\alpha + e_j}}{\Dx}
	\one{\{V^j_{\alpha + e_j}\ge 0, \,V^j_\alpha < 0\}}
	- \frac{V^j_\alpha}{\Dx}
	\one{\{V^j_{\alpha + e_j} < 0, \,V^j_\alpha \ge 0\}} \\
	&\eqqcolon E^1_\alpha + E^2_\alpha + E^3_\alpha.
\end{aligned}
\end{equation*}
We claim that $E^1 \to \partial_{x^j} V^j \one{\{V_j\geq0\}}$
pointwise a.e.~as $\Dx\to0$, and that $E^2,E^3\to0$.

If $\bar x$ is a Lebesgue point for both $V^j$ and $\partial_{x^j} V^j$
then $V_\alpha^j, V_{\alpha+e_j}^j \to V^j(\bar{x})$
and $\nabla_+^j V^j_\alpha =
\frac{V_{\alpha+e_j}^j-V_\alpha^j}{\Dx} \to \partial_{x^j} V^j(\bar{x})$
as $\Dx\to0$, where $\alpha\in\Z^d$ is such
that $\bar{x}\in \cell_\alpha$. In particular, either $V^j(\bar{x})=0$,
or $V^j(\bar{x}) \neq 0$, in which case
$\one{\{V^j_{\alpha + e_j}\ge 0,\, V^j_\alpha < 0\}}=0$
for sufficiently small $\Dx$.
Hence,
\begin{align*}
	0 &\leq E^2(\bar{x}) = \sum_{\alpha \in \Z^d}
	\one{\cell_\alpha}(\bar{x}) E^2_\alpha
	\leq \sum_{\alpha \in \Z^d} \one{\cell_\alpha}(\bar{x})
	\Bigl|\frac{V_{\alpha+e_j}^j-V_\alpha^j}{\Dx}
	\Bigr| \one{\{V^j_{\alpha + e_j}\ge 0,\, V^j_\alpha < 0\}} \\
	&\leq |\partial_{x^j}V^j(\bar{x})| \one{\{V^j=0\}}(\bar{x}) + o(1)
\end{align*}
as $\Dx\to0$. But $\partial_{x^j} V^j = 0$ a.e.~on $\{V^j = 0\}$
(see e.g.~\cite[Theorem 3.1]{Bouchut:2001zr}), so we conclude
that $E^2 \to 0$ a.e.~as $\Dx\to0$. In a similar fashion
we find that $E^3 \to 0$, and that
$E^1 \to \partial_{x^j} V^j \one{\{V^j\geq0\}}$.
\end{proof}

\begin{proof}[Proof of (iii)]
By Lemma \ref{lem:discrete_calculus}
\ref{lem:discrete_calculus_upwind-product-rule}
we can decompose $D_V' \varphi_\alpha$ as
\begin{align*}
D_V'\varphi_\alpha
& = \varphi_\alpha D_V'(1)_\alpha + \frac1\Dx \sum_{j = 1}^d
V_{\alpha + e_j}^{j,+} \bk{\varphi_{\alpha + e_j} - \varphi_\alpha}
- V_{\alpha - e_j}^{j, - } \bk{\varphi_\alpha - \varphi_{\alpha - e_j}}.
\end{align*}
It is obvious that $P_\Dx \varphi \to \varphi$ in
$L^\infty(\R^d)$ and $D'_V(1)$ converges in $L^p(\R^d)$
by {\em (ii)}. Therefore, we can concentrate on the sum. We will prove
convergence for the positive part. For each fixed
$j = 1, \ldots, d$, we have the decomposition
\begin{align*}
	& \nqquad\Biggl\|{\frac1\Dx} \sum_{\alpha \in \Z^d}
	V_{\alpha + e_j}^{j,+} \bk{\varphi_{\alpha + e_j}
	- \varphi_\alpha} \one{\cell_\alpha}
	- V^{j,+} \partial_{x_j} \varphi\Biggr\|_{L^p(\R^d)}\\
	& \le  \Biggl\|\frac1\Dx \sum_{\alpha \in \Z^d}
	\bk{	V_{\alpha + e_j}^{j,+} - V^{j,+}}\bk{\varphi_{\alpha + e_j}
	- \varphi_\alpha} \one{\cell_\alpha}\Biggr\|_{L^p(\R^d)} \\*
	&\relspace +  \Biggl\| \sum_{\alpha \in \Z^d}
	V^{j,+} \bk{\frac1\Dx \bk{\varphi_{\alpha + e_j}
	- \varphi_\alpha} - \partial_{x^j} \varphi} \one{\cell_\alpha}\Biggr\|_{L^p(\R^d)}\\
	& \le \Biggl\|\sum_{\alpha \in \Z^d}
	\bk{	V_{\alpha }^{j,+} - V^{j,+}}\one{\cell_\alpha}\Biggr\|_{L^p(\R^d)}
	\Biggl\|\sum_{\alpha \in \Z^d} \frac{\varphi_{\alpha + e_j}
	- \varphi_\alpha}\Dx \one{\cell_\alpha}\Biggr\|_{L^\infty(\R^d)}\\*
	&\relspace + \Biggl\|\sum_{\alpha \in \Z^d}
	\bk{	V_{\alpha + e_j}^{j,+} - V_{\alpha}^{j,+}}\one{\cell_\alpha}\Biggr\|_{L^p(\R^d)}
	\Biggl\|\sum_{\alpha \in \Z^d} \frac{\varphi_{\alpha + e_j}
	- \varphi_\alpha}\Dx \one{\cell_\alpha}\Biggr\|_{L^\infty(\R^d)}\\*
	 &\relspace + \norm{V^{j,+}}_{L^p(\R^d)}
 	\Biggl\| \sum_{\alpha \in \Z^d}
	 \bk{\frac{\varphi_{\alpha + e_j}
	- \varphi_\alpha}\Dx - \partial_{x^j} \varphi} \one{\cell_\alpha}\Biggr\|_{L^\infty(\R^d)}.
\end{align*}
The second summand tends to nought by the proof of {\em (ii)}, and the
last summand tends to nought as $\partial_{x^j}\varphi \in C_b(\R^n)$.
It remains to show that the first factor of the
first summand converges. By the continuity of the positive part function,
taking pointwise limits $\Dx\to0$ yields
\begin{align*}
\lim_{\Dx \to 0}\bigl(P_\Dx V^j\bigr)^+
&= \lim_{\Dx \to 0}
\Biggl(\sum_{\alpha \in \Z^d}\one{\cell_\alpha}\fint_{\cell_{\alpha}} V^j\,\d x\Biggr)^+ \\
&=
\Biggl(\lim_{\Dx \to 0}
\sum_{\alpha \in \Z^d}\one{\cell_\alpha}\fint_{\cell_{\alpha}} V^j\,\d x\Biggr)^+
=  V^{j,+}.
\end{align*}
The convergence of the summand to zero
then follows from an application of the dominated convergence theorem.
\end{proof}

Lastly, we record the
following elementary convergences,
provided that $\sigma$ and $\varphi$
possess sufficient regularity. Their proofs are omitted for the sake of brevity.

\begin{lemma}[Consistency of $\sigma$-related approximations]
\label{lem:V_DV_discretelimits2}
Let $\sigma \in W^{2,\infty}(\R^d)$, and $\varphi \in C^2_c(\R^d)$.
Define $\varphi_\alpha \coloneqq \fint_{\cell_\alpha} \varphi\,\d x$ and
$\sigma_\alpha \coloneqq \fint_{\cell_\alpha} \sigma\,\d x$.
We have the following convergences pointwise and
in $L^1(\R^d)$ as $\Dx \to 0$:
\begin{align*}
	\sum_{\alpha \in \Z^d} \varphi_\alpha  \one{\cell_\alpha}
	& \to \varphi ,\\
	\sum_{\alpha \in \Z^d} \nabla_0
	\varphi_\alpha  \one{\cell_\alpha}
	&\to \nabla \varphi ,\\
	\sum_{\alpha \in \Z^d} \nabla_0
	\bk{S_\alpha \varphi_\alpha}  \one{\cell_\alpha}
	&\to \nabla \bk{\sigma \varphi},\\
	\sum_{\alpha \in \Z^d}   \sum_{j  = 1}^d
 	\nabla_+^j \bk{\average_-^j\sigma_\alpha \nabla_-^j
 	\bk{\sigma_\alpha \varphi_\alpha}}  \one{\cell_\alpha}
	& \to \nabla\cdot\bk{\sigma \nabla \bk{\sigma \varphi}}.
\end{align*}
Moreover, for $u^0 \in L^2(\R^d)$ and
$u^0_\alpha = \fint_{\cell_\alpha} u^0\,\d x$,
\[
\sum_{\alpha \in \Z^d} u^0_\alpha
\varphi_\alpha \one{\cell_\alpha}
\to u^0 \varphi \quad \text{in}
	\quad (L^1 \cap L^2)(\R^d)\quad \text{ as $\Dx \to 0$}.
\]
\end{lemma}

\subsection{Convergence and the limit equation}

Using the $\Dx$-independent stability estimate for the finite
difference approximations (Theorem \ref{thm:disc_energyestm}),
we can establish weak $L^2$ convergence (compactness)
of the solutions to the scheme \eqref{eq:scheme}.
This weak compactness, along with the consistency
Lemmas \ref{lem:V_DV_discretelimits1}
and \ref{lem:V_DV_discretelimits2}, implies that the
difference approximations converge to the
unique solution of the stochastic transport equation
\eqref{eq:stratonovich_eq}, thereby
confirming Theorem \ref{thm:main_conv}.

Recall that $\cell_\alpha$ denotes 
the cube $\alpha\Dx + [-\Dx/2,\Dx/2)^d$.
From the $L^2$ estimate \eqref{eq:energy_estimate1},
the Banach--Alaoglu theorem readily implies that:

\begin{lemma}[Convergence]
\label{lem:weak_compactness}
Assume that Assumption \ref{ass:main_assumption} holds,
and let $u_{\Dx} \in L^\infty([0,T]; L^2(\Omega\times \R^d)))$ be
the $x$-piecewise constant function defined by
\[
u_{\Dx}(\omega,t,x) = u_\alpha(\omega, t),
\quad\text{for}\,\, (\omega, t, x)
\in \Omega \times [0,T] \times \cell_\alpha,
\]
i.e., $ u_{\Dx} = \sum_{\alpha \in \Z^d}
u_\alpha \one{\cell_\alpha}$,
where $u_\alpha$
is computed by \eqref{eq:scheme}.
Then there exists an adapted $\overline{u}$
such that $u_{\Dx} \rightharpoonup \overline{u}$ in
$L^2(\Omega \times [0,T] \times \R^d)$
along a (non-relabelled) subsequence
as $\Dx \to 0$.
\end{lemma}

The assertion of adaptedness in
Lemma \ref{lem:weak_compactness}
follows from the weak convergence
of adapted processes \cite[Remark 5.2.17]{Liu:2015vb}.
Once we have verified that the
limit $\bar{u}$ is a weak solution
of \eqref{eq:stratonovich_eq}, as
demonstrated in the proposition
below, the uniqueness of the weak
solution guarantees that the entire
sequence $(u_\Dx)_{\Delta x > 0}$
converges, not just a subsequence.
We are now in a position to prove part (ii) of Theorem \ref{thm:main_conv}.

\begin{proposition}[Limit equation]
\label{thm:equation_wklim}
Under Assumption \ref{ass:main_assumption}, there is
a version of the weak limit $\overline{u}$ from
Lemma \ref{lem:weak_compactness}, still called
$\overline{u}$, such that for any
$\varphi \in C^2_c(\R^d)$, and any $t \in [0,T]$,
\begin{equation}\label{eq:wk_limiteq}
\begin{aligned}
	0 &= \int_{\R^d} \bar{u}(t) \varphi\,\d x
	-  \int_{\R^d} {u}^0 \varphi\,\d x
	- \int_0^t \int_{\R^d} \bar{u} \varphi\nabla \cdot V
	+ \bar{u}V\cdot\nabla\varphi \,\d x \,\d s \\
	&\relspace
	- \int_0^t \int_{\R^d}  \bar{u} \nabla
	 \bk{\sigma \varphi} \,\d x \cdot \d W
	- \frac12 \int_0^t \int_{\R^d}
	 \bar{u} \nabla \cdot\bk{\sigma\,
	 \nabla \bk{\sigma \varphi}}\,\d x \,\d s,
\end{aligned}
\end{equation}
$\mathbb{P}$-almost surely. Moreover,
$t \mapsto \int_{\R^d} \overline{u}(t) \varphi \,\d x$ is continuous a.s.
\end{proposition}

\begin{proof}
Recall from \eqref{eq:avg_defin} that
$u^0_\alpha = \fint_{\cell_\alpha} u^0\,\d x$.
Multiplying \eqref{eq:scheme1} through by
the piecewise constant approximation
$\varphi_\Dx=\sum_{\alpha \in \Z^d}
\varphi_\alpha \one{\cell_\alpha}$
of $\varphi \in C^2_c(\R^d)$, we have:
\begin{align*}
	0 &= \int_{\R^d}\sum_{\alpha \in \Z^d} u_\alpha (t)
 	\varphi_\alpha \one{\cell_\alpha}\,\d x
	-\int_{\R^d}\sum_{\alpha \in \Z^d} u_\alpha^0
 	\varphi_\alpha \one{\cell_\alpha}\,\d x\\
	&\relspace
 	-\int_0^t \int_{\R^d} \sum_{\alpha \in \Z^d}
	D_V' \varphi_\alpha
	 u_\alpha  \one{\cell_\alpha}\,\d x\,\d s\\
	 & \relspace -\int_0^t \int_{\R^d} \sum_{\alpha \in \Z^d}
 	 u_\alpha \one{\cell_\alpha}
 	\nabla_0 \bk{\varphi_\alpha S_\alpha }
 	\,\d x \cdot\d W\\
	& \relspace - \frac12 \int_0^t \int_{\R^d}
	\sum_{\alpha \in \Z^d} \sum_{j  = 1}^d
	u_\alpha \nabla_+^j \bk{\average_-^j\sigma_\alpha \nabla_-^j
 	\bk{\sigma_\alpha \varphi_\alpha}}\one{\cell_\alpha} \,\d x\,\d s\\
	& \eqqcolon I_1 + I_2 + I_3 + I_4 + I_5.
\end{align*}
In order to take weak limits in the
product space $\Omega \times [0,T] \times \mathbb{R}^d$, we
multiply this equation by an arbitrary random
variable $X \in L^2(\Omega)$ and take an expectation.

Based on the consistency results established in
Lemmas \ref{lem:V_DV_discretelimits1} and \ref{lem:V_DV_discretelimits2},
we have 
\begin{align*}
	\bigl(D_V'(\phi_\alpha) - \phi_\alpha D_V'(1)_\alpha\bigr)
	+ \varphi_\alpha D_V'(1)_\alpha
	&\to  V \cdot \nabla \varphi + \varphi \Div V &&\text{in } L^p_x, \\
	\sum_{j = 1}^d
	\nabla_+^j \bk{\average_-^j\sigma_\alpha \nabla_-^j
 	\bk{\sigma_\alpha \varphi_\alpha}}
	&\to \nabla \cdot \bk{\sigma \nabla \bk{\sigma \varphi}} &&\text{in } L^\infty_x
\end{align*}
as $\Dx\to0$.
Along with Lemma \ref{lem:weak_compactness},
these ensure the convergences of $\Ex \big[X I_3 \big]$,
and $\Ex \big[ X I_5\big]$ to their
continuous counterparts. The term $\Ex \big[X I_2\big]$ converges by Lemma~\ref{lem:V_DV_discretelimits2}.

As the convergence $u_\Dx \to \bar{u}$
is weak in $(\omega, t, x)$, we can use a
weak convergence argument
of \cite[pp.~15--16]{Frid:2021us} to get
\[
\int_0^T \bk{\int_{\R^d}
\phi_\Dx  S_{\Dx}
\nabla_0 u_\Dx \,\d x}
\cdot \,\d W
\rightharpoonup
-\int_0^T \bk{\int_{\R^d}
\bar{u}\nabla \bk{\phi \sigma}  \,\d x}
\cdot \,\d W
\]
in $L^2(\Omega)$, for any test function
$\phi \in L^\infty([0,T];H^1(\R^d))$
with $\phi_\Dx=\sum_{\alpha \in \Z^d}
\varphi_\alpha \one{\cell_\alpha}$ and
$S_\Dx=\sum_{\alpha \in \Z^d}
S_\alpha \one{\cell_\alpha}$.
This ensures the convergence of  $\Ex \big[X I_4\big]$.

Finally we consider $I_1$.
Since $I_1= -I_2-I_3-I_4-I_5$,
by the convergences established
for the right-hand side, it holds that
$\Ex \Big[ XI_1 \Big](t)$ converges to
a continuous function $F_{X, \varphi}(t)$.
On the other hand, using the convergence of $u_\Dx$
provided by Lemma \ref{lem:weak_compactness},
it holds that for any $\psi \in L^\infty([0,T])$,
\[
\Ex  \int_0^T X \psi I_1\,\d t
\to \Ex  \int_0^T\int_{\R^d}
X \psi\bar{u}  \varphi\,\d x \,\d t,
\]
whence $F_{X, \varphi}(t) = \Ex \int_{\R^d}
X  \bar{u}(t) \varphi\,\d x$
for a.e.~$t \in [0,T]$. The null set of times for
which the equality does not hold is dependent
on $(X, \varphi) \in L^2(\Omega) \times C^2_c(\R^d)$,
which is separable. This permits us to redefine
$\bar{u}$ on a universal null set of times
by doing so only for a countable dense
set of pairs $(X, \varphi)$. This establishes \eqref{eq:wk_limiteq}.
Since \eqref{eq:wk_limiteq} holds a.s.~for
all $t \in [0,T]$, and each temporal integral
in \eqref{eq:wk_limiteq} is time-continuous,
it holds that $\int_{\R^d} \bar{u}(t) \varphi\,\d x$
is also a.s.~time-continuous.
\end{proof}

\appendix

\section{Convolution estimates and discrete calculus} 
\label{sec:disccalc}
This section is dedicated to the proofs
of Lemma \ref{lem:young} and \ref{lem:discrete_calculus}.

\begin{proof}[Proof of Lemma \ref{lem:young}]
We realise each Lebesgue norm in terms of its dual formulation, e.g.
\[
\|\bm{f}\|_{L^p} = \sup_{\substack{\psi\in L^{p'}_\Dx \\ 
\|\psi\|_{L^{p'}} = 1}} \sum_{\alpha\in\Z^d} f_\alpha\psi_\alpha\,\Dx^d
\]
for $\bm{f}\in L^p_\Dx$. For 
$\bm{F}\in L^{p_2}_\alpha L^{p_3}_\beta$ 
and $\bm{G}\in L^{q_2}_\alpha L^{q_3}_\beta$ we have
\begin{align*}
\|(\bm{F}\dconv\bm{G})\|_{L^{p_2}_\alpha L^{q_3}_\beta}
&= \sup_{\|\psi\|_{L^{p_2'}} = 1} \sum_{\alpha \in \Z^d} \psi_\alpha 
	\Biggl(\sup_{\|\tilde\psi\|_{L^{p_3'}} = 1} 
	\sum_{\beta \in \Z^d} \tilde\psi_\beta \sum_{\eta \in \Z^d} 
	F_{\alpha,\eta} G_{\eta,\beta} \Biggr)\,\Dx^{3d} \\
&\leq \sup_{\|\psi\|_{L^{p_2'}} = 1} \sum_{\alpha \in \Z^d} \psi_\alpha 
	\sum_{\eta \in \Z^d} F_{\alpha,\eta} \Biggl(\sup_{\|\tilde\psi\|_{L^{p_3'}} = 1} 
	\sum_{\beta \in \Z^d} \tilde\psi_\beta G_{\eta,\beta} \Biggr)\,\Dx^{3d} \\
&= \sup_{\|\psi\|_{L^{p_2'}} = 1} \sum_{\alpha \in \Z^d} \psi_\alpha 
	\sum_{\eta \in \Z^d} F_{\alpha,\eta} \|G_{\eta,\cdot}\|_{L^{q_3}_\beta}\,\Dx^{2d} \\
\intertext{(\textit{by H\"older's inequality})}
&\leq \sup_{\|\psi\|_{L^{p_2'}} = 1} \sum_{\alpha \in \Z^d} 
	|\psi_\alpha| \|F_{\alpha,\cdot}\|_{L^{p_3}_\beta} 
	\|G\|_{L^{p_3'}_\alpha L^{q_3}_\beta}\,\Dx^d
\leq \|F\|_{L^{p_2}_\alpha L^{p_3}_\beta} 
	\|G\|_{L^{p_3'}_\alpha L^{q_3}_\beta},
\end{align*}
which proves (i) since $q_2=p_3'$. To prove (ii) 
we apply first Minkowski's inequality, then (i), 
and finally Young's convolution inequality in time to get
\begin{align*}
\|\bm{F}\dconv\bm{G}\|_{L^{r_1}_tL^{p_2}_\alpha L^{q_3}_\beta}
&= \Biggl(\int_0^T \biggl\|\int_0^t \bm{F}(t-s)\dconv 
	\bm{G}(s)\,\d s \biggr\|_{L^{p_2}_\alpha L^{q_3}_\beta}^{r_1}\,\d t \Biggr)^{1/r_1} \\
&\leq \Biggl(\int_0^T \biggl(\int_0^t \bigl\|
	\bm{F}(t-s)\dconv \bm{G}(s) \bigr\|_{L^{p_2}_\alpha L^{q_3}_\beta} 
	\,\d s \biggr)^{r_1}\,\d t \Biggr)^{1/r_1} \\
&\leq \Biggl(\int_0^T \biggl(\int_0^t \|\bm{F}(t-s)\|_{L^{p_2}_\alpha L^{p_3}_\beta} 
	\|\bm{G}(s) \bigr\|_{L^{q_2}_\alpha L^{q_3}_\beta} 
	\,\d s \biggr)^{r_1}\,\d t \Biggr)^{1/r_1} \\
&\leq \|\bm{F}\|_{L^{p_1}_t L^{p_2}_\alpha L^{p_3}_\beta} 
	\|\bm{G}\|_{L^{q_1}_t L^{q_2}_\alpha L^{q_3}_\beta},
\end{align*}
which holds since $1/p_1+1/q_1=1+1/r_1$.
\end{proof}

\begin{proof}[Proof of  Lemma
\ref{lem:discrete_calculus}
\ref{lem:discrete_calculus_product-rule}]
We show this for the forward difference:
\begin{align*}
	\nabla_+^j (h_\alpha g_\alpha)
    &= h_{\alpha + e_j} \bk{g_{\alpha + e_j} - g_{\alpha}}
    + \bk{h_{\alpha + e_j} - h_\alpha} g_\alpha\\
    &=  h_{\alpha + e_j} \nabla_+ g_\alpha
	+ g_\alpha \nabla_+ h_\alpha. \qedhere
\end{align*}
\end{proof}

\begin{proof}[Proof of  Lemma
\ref{lem:discrete_calculus}
\ref{lem:discrete_calculus_upwind-product-rule}]
From the definition \eqref{eq:dual_upwind}
of the dual upwind difference $D'_V g_\alpha$,
we have
\begin{align*}
	D_V'g_\alpha
	=&\ \frac{1}{\Dx}\sum_{j=1}^d \bigl(V_{\alpha+e_j}^{j,+}g_{\alpha+e_j}
	- V_\alpha^{j,+}g_\alpha\bigr) - \bigl(V_{\alpha}^{j,-}g_{\alpha}
	- V_{\alpha-e_j}^{j,-}g_{\alpha-e_j}\bigr)\\
    =&\ \frac{1}{\Dx}\sum_{j=1}^d \bigl(V_{\alpha+e_j}^{j,+}\bk{g_{\alpha+e_j} - g_{\alpha}}
	- V_\alpha^{j,+}g_\alpha\bigr) + V_{\alpha+e_j}^{j,+}g_\alpha\\
    &-  \frac{1}{\Dx}\sum_{j=1}^d \bigl(V_{\alpha}^{j,-}g_{\alpha}
	+ V_{\alpha-e_j}^{j,-}\bk{g_\alpha - g_{\alpha-e_j}}\bigr)
    - V_{\alpha-e_j}^{j,-}g_\alpha \\
    =&\ \frac{1}{\Dx}\sum_{j=1}^d \bigl(V_{\alpha+e_j}^{j,+}\bk{g_{\alpha+e_j} - g_{\alpha}}
	   +\bigl(V_{\alpha+e_j}^{j,+} - V_\alpha^{j,+}\bigr)g_\alpha\bigr)\\
    & -  \frac{1}{\Dx}\sum_{j=1}^d \bigl(
	V_{\alpha-e_j}^{j,-}\bk{g_\alpha - g_{\alpha-e_j}}\bigr)
    - \bigl(V_{\alpha-e_j}^{j,-} - V_{\alpha}^{j,-}\bigr)g_\alpha\bigr) \\
    =&\  g_\alpha D_V'(1)_\alpha
	+ \frac1{\Dx}\sum_{j=1}^d \bigl(g_{\alpha + e_j} - g_\alpha\bigr)
	V^{j, + }_{\alpha + e_j} - \bigl(g_\alpha - g_{\alpha - e_j}\bigr)
 	V^{j, - }_{\alpha - e_j}.
\end{align*}

\end{proof}
\begin{proof}[Proof of  Lemma
\ref{lem:discrete_calculus}
\ref{lem:discrete_calculus_differentiate-convolution}]
Using the definition \eqref{eq:convolution-def-univariate}
of the discrete convolution,
\begin{align*}
\nabla_\pm^j (\bm{f}\dconv \bm{g})_\alpha
&= \Dx^{d-1}\Biggl(\sum_{\beta\in\Z^d} f_{\alpha\pm e_j} g_\beta
-  \sum_{\beta\in\Z^d} f_{\alpha-\beta} g_\beta\Biggr)\\
&= \Dx^{d-1}\sum_{\beta\in\Z^d} \Bigl( f_{\alpha\pm e_j}
- f_{\alpha-\beta} \Bigr)g_\beta = \bigl((\nabla_\pm^j \bm{f})\dconv \bm{g}\bigr)_\alpha.
\end{align*}
Since $\bm{f}\dconv \bm{g} = \bm{g}\dconv \bm{f}$,
the other equality also holds.
\end{proof}

\begin{proof}[Proof of  Lemma
\ref{lem:discrete_calculus}
\ref{lem:discrete_calculus_int-by-parts}]
This follows from summing Lemma
\ref{lem:discrete_calculus}
\ref{lem:discrete_calculus_product-rule}
with boundary values ${\bm h \bm g} = 0$
at infinity.
\end{proof}
\begin{proof}[Proof of  Lemma
\ref{lem:discrete_calculus}
\ref{lem:discrete_calculus_upwind-int-by-parts}]
Again using the definition  \eqref{eq:dual_upwind}
of the dual upwind difference $D'_V g_\alpha$,
by shifting indices of summation, we find
\begin{align*}
	&\sum_{\alpha \in \Z^d} f_\alpha D_V'g_\alpha\\
	&= \frac{1}{\Dx}\sum_{\alpha \in \Z^d} f_\alpha
    \sum_{j=1}^d \bigl(V_{\alpha+e_j}^{j,+}g_{\alpha+e_j}
	- V_\alpha^{j,+}g_\alpha\bigr) - \bigl(V_{\alpha}^{j,-}g_{\alpha}
	- V_{\alpha-e_j}^{j,-}g_{\alpha-e_j}\bigr)\\
    &= \frac{1}{\Dx}\sum_{\alpha \in \Z^d}
    \sum_{j=1}^d \bigl(V_{\alpha}^{j,+}g_{\alpha}f_{\alpha - e_j}
	- V_\alpha^{j,+}g_\alpha f_\alpha\bigr) - \bigl(V_{\alpha}^{j,-}f_{\alpha}g_{\alpha}
	- V_{\alpha}^{j,-}f_{\alpha + e_j}g_{\alpha}\bigr)\\
     &= \frac{1}{\Dx} g_\alpha \sum_{\alpha \in \Z^d}
    \sum_{j=1}^d \bigl(V_{\alpha}^{j,+}f_{\alpha - e_j}
	- V_\alpha^{j,+}f_\alpha\bigr) - \bigl(V_{\alpha}^{j,-}f_{\alpha}
	- V_{\alpha}^{j,-}f_{\alpha + e_j}\bigr) \\
      &= \frac{1}{\Dx} g_\alpha \sum_{\alpha \in \Z^d}
    \sum_{j=1}^d \bigl(V_{\alpha}^{j,+}\bk{f_{\alpha - e_j}
	- f_\alpha} - V_{\alpha}^{j,-} \bk{f_{\alpha}
	- f_{\alpha + e_j}}\bigr)
   \overset{\eqref{eq:upwind-scheme}}{=}
   -\sum_{\alpha \in \Z^d} g_\alpha D_V f_\alpha.
\end{align*}
\end{proof}
\begin{proof}[Proof of Lemma
\ref{lem:discrete_calculus}
\ref{lem:discrete_calculus_chain-rules}]
By Taylor's formula, for some
$\xi^+_\alpha \in [f_\alpha \wedge f_{\alpha + e_j},
f_\alpha \vee f_{\alpha + e_j}]$,
\begin{align*}
\beta(f_{\alpha + e_j}) - \beta(f_{\alpha})
= \beta'(f_\alpha) \bk{f_{\alpha + e_j} - f_\alpha}
	+ \frac12 \beta''(\xi^+_\alpha) \abs{f_{\alpha + e_j} - f_\alpha}^2.
\end{align*}
Dividing the entire equation through
by $\Dx$ yields the chain rule for
$\nabla_+$. The chain rule for $\nabla_-$
can be shown in a similar way.

Subtracting the chain rule for the
forward difference from that for the
backward difference, we get
\begin{align*}
	\nabla_-\nabla_+ \beta(f_\alpha)
	& =\frac{\nabla_+\beta(f_\alpha)-\nabla_-\beta(f_\alpha)}{\Delta x}
	\\ & =  \beta'(f_\alpha)\frac{\nabla_+ f_\alpha-\nabla_- f_\alpha}{\Delta x}
	\\ & \qquad
	+\frac{\Delta x}{2\Delta x} \beta''(\xi_\alpha^+)\abs{\nabla_+ f_\alpha}^2
	+ \frac{\Delta x}{2\Delta x} \beta''(\xi_\alpha^-)\abs{\nabla_- f_\alpha}^2
	\\ & =  \beta'(f_\alpha)\nabla_-\nabla_+ f_\alpha
	\\ & \qquad
	+\frac{1}{2} \beta''(\xi_\alpha^+)\abs{\nabla_+ f_\alpha}^2
	+ \frac{1}{2} \beta''(\xi_\alpha^-)\abs{\nabla_- f_\alpha}^2,
\end{align*}
for some numbers $\xi_\alpha^\pm$ between
$u_{i\pm 1}$ and $f_\alpha$. In other words,
\begin{equation*}
	\beta'(f_\alpha)\nabla_-\nabla_+ f_\alpha
	-\nabla_-\nabla_+ \beta(f_\alpha)
	=-\frac{1}{2} \beta''(\xi_\alpha^+)\abs{\nabla_+ f_\alpha}^2
	-\frac{1}{2} \beta''(\xi_\alpha^-)\abs{\nabla_- f_\alpha}^2.
\end{equation*}
This is the commmutator rule
\eqref{eq:2nd-order-commutator-beta}

\end{proof}
\begin{proof}[Proof of  Lemma
\ref{lem:discrete_calculus}
\ref{lem:discrete_calculus_upwind-chain-rule}]

From the definition \eqref{eq:upwind-scheme}
of the upwind difference $D_V$,
\begin{align*}
\beta'(f_\alpha) D_V f_\alpha
& = \frac{\beta'(f_\alpha)}{\Dx}\sum_{j = 1}^d
	\bk{V^{j,+}_\alpha \bk{f_\alpha - f_{\alpha - e_j}}
	- V^{j,-}_\alpha \bk{f_{\alpha + e_j} - f_\alpha}}, \\
D_V \beta (f_\alpha)
& = \frac1\Dx \sum_{j = 1}^d
	\bk{V^{j,+}_\alpha \bk{\beta(f_\alpha) - \beta(f_{\alpha - e_j})}
	- V^{j,-}_\alpha \bk{\beta(f_{\alpha + e_j}) - \beta(f_\alpha)}}.
\end{align*}
Subtracting the second from the first and
using Lemma \ref{lem:discrete_calculus}
\ref{lem:discrete_calculus_chain-rules}
proves the assertion.
\end{proof}



\begin{thebibliography}{10}



\bibitem{AL2021}
L.~Abadias and M.~De León-Contreras. 
\newblock{Discrete Hölder spaces and their characterization via semigroups associated with the discrete Laplacian and kernel estimates.}
\newblock{\em J. Evol. Equations}, 22:Article 91, 2021.

\bibitem{AGMP2021}
L.~Abadias, J.~Gonz\'alez-Camus, P.J.~Miana, and J.C.~Pozo.
\newblock{Large time behaviour for the heat equation on $\Z$, moments and decay rates}
\newblock{\em  J. Math. Anal. Appl.},  500:125137, 2021.

\bibitem{Ambrosio:2004aa}
L.~Ambrosio.
\newblock Transport equation and {C}auchy problem for {$BV$} vector fields.
\newblock {\em Invent. Math.}, 158(2):227--260, 2004.

\bibitem{Attanasio:2011fj}
S.~Attanasio and F.~Flandoli.
\newblock Renormalized solutions for stochastic transport equations and the
  regularization by bilinear multiplication noise.
\newblock {\em Comm. Partial Differential Equations}, 36(8):1455--1474, 2011.

\bibitem{Bauzet:2016aa}
C.~Bauzet, J.~Charrier, and T.~Gallou\"et.
\newblock Convergence of flux-splitting finite volume schemes for hyperbolic
  scalar conservation laws with a multiplicative stochastic perturbation.
\newblock {\em Math. Comp.}, 85(302):2777--2813, 2016.

\bibitem{Bauzet:2016ab}
C.~Bauzet, J.~Charrier, and T.~Gallou\"et.
\newblock Convergence of monotone finite volume schemes for hyperbolic scalar
  conservation laws with multiplicative noise.
\newblock {\em Stoch. Partial Differ. Equ. Anal. Comput.}, 4(1):150--223, 2016.

\bibitem{Banas:2014aa}
L.~Ba\v{n}as, Z.~Brze\'{z}niak, M.~Neklyudov, and A.~Prohl.
\newblock {\em Stochastic ferromagnetism}, volume~58 of {\em De Gruyter Studies
  in Mathematics}.
\newblock De Gruyter, Berlin, 2014.
\newblock Analysis and numerics.

\bibitem{BW2016}
R.~Beals and R.~Wong.
\newblock {\em Special functions and orthogonal polynomials}, volume 153 of
  {\em Cambridge Studies in Advanced Mathematics}.
\newblock Cambridge University Press, Cambridge, 2016.

\bibitem{Beck:2019aa}
L.~Beck, F.~Flandoli, M.~Gubinelli, and M.~Maurelli.
\newblock Stochastic {ODE}s and stochastic linear {PDE}s with critical drift:
  regularity, duality and uniqueness.
\newblock {\em Electron. J. Probab.}, 24:Paper No. 136, 72, 2019.

\bibitem{Ben-Belgacem:2019aa}
F.~Ben~Belgacem and P.-E. Jabin.
\newblock Convergence of numerical approximations to non-linear continuity
  equations with rough force fields.
\newblock {\em Arch. Ration. Mech. Anal.}, 234(2):509--547, 2019.

\bibitem{Bouchut:2001zr}
F.~Bouchut.
\newblock Renormalized solutions to the {V}lasov equation with coefficients of
  bounded variation.
\newblock {\em Arch. Ration. Mech. Anal.}, 157(1):75--90, 2001.

\bibitem{Boyer:2012ab}
F.~Boyer.
\newblock Analysis of the upwind finite volume method for general initial- and boundary-value transport problems.
\newblock {\em IMA J. Numer. Anal.}, 32(4):1404--1439, 2012.

\bibitem{Brezis2011}
H.~Brezis.
\newblock \textit{Functional Analysis, {S}obolev Spaces and Partial Differential Equations.}
\newblock Springer-Verlag New York, 2011.

\bibitem{Chow:2015aa}
P.-L. Chow.
\newblock {\em Stochastic partial differential equations}.
\newblock Advances in Applied Mathematics. CRC Press, Boca Raton, FL, second
  edition, 2015.

\bibitem{CCDPS2019}
C.~Cotter, D.~Crisan, D.D.~Holm, W.~Pan, and I.~Shevchenko.
\newblock{Numerically modeling stochastic Lie transport in fluid dynamics.}
\newblock{\em SIAM J. Numer. Anal.}, 17(1):192--232, 2019.

\bibitem{Delarue:2011aa}
F.~Delarue and F.~Lagouti\`ere.
\newblock Probabilistic analysis of the upwind scheme for transport equations.
\newblock {\em Arch. Ration. Mech. Anal.}, 199(1):229--268, 2011.

\bibitem{Despres:2004aa}
B.~Despr\'{e}s.
\newblock An explicit a priori estimate for a finite volume approximation of
  linear advection on non-{C}artesian grids.
\newblock {\em SIAM J. Numer. Anal.}, 42(2):484--504, 2004.

\bibitem{DiPerna:1989aa}
R.~J. DiPerna and P.-L. Lions.
\newblock Ordinary differential equations, transport theory and {S}obolev
  spaces.
\newblock {\em Invent. Math.}, 98(3):511--547, 1989.

\bibitem{Dotti:2020aa}
S.~Dotti and J.~Vovelle.
\newblock Convergence of the finite volume method for scalar conservation laws
  with multiplicative noise: an approach by kinetic formulation.
\newblock {\em Stoch. Partial Differ. Equ. Anal. Comput.}, 8(2):265--310, 2020.

\bibitem{Eymard:2000fr}
R.~Eymard, T.~Gallou{\"e}t, and R.~Herbin.
\newblock Finite volume methods.
\newblock In {\em Handbook of numerical analysis, Vol. VII}, Handb. Numer.
  Anal., VII, pages 713--1020. North-Holland, Amsterdam, 2000.

\bibitem{Flandoli:2011vn}
F.~Flandoli.
\newblock {\em Random perturbation of {PDE}s and fluid dynamic models}, volume
  2015 of {\em Lecture Notes in Mathematics}.
\newblock Springer, Heidelberg, 2011.
\newblock Lectures from the 40th Probability Summer School held in Saint-Flour,
  2010.

\bibitem{Flandoli:2010yq}
F.~Flandoli, M.~Gubinelli, and E.~Priola.
\newblock Well-posedness of the transport equation by stochastic perturbation.
\newblock {\em Invent. Math.}, 180(1):1--53, 2010.

\bibitem{Folland:1999aa}
G.~B. Folland.
\newblock {\em Real analysis}.
\newblock Pure and Applied Mathematics (New York). John Wiley \& Sons, Inc.,
  New York, second edition, 1999.

\bibitem{Frid:2021us}
H.~Frid, K.~H. Karlsen, and D.~Marroquin.
\newblock Homogenization of stochastic conservation laws with multiplicative
  noise.
\newblock {\em J. Funct. Anal.}, 283(9):63 pp., 2022.

\bibitem{Fri1964}
A.~Friedman.
\newblock {\em Partial differential equations of parabolic type.}
\newblock Prentice-Hall, Inc., Englewood Cliffs, N.J., 1964.

\bibitem{Funaki:2018aa}
T.~Funaki, Y.~Gao, and D.~Hilhorst.
\newblock Convergence of a finite volume scheme for a stochastic conservation
  law involving a {$Q$}-{B}rownian motion.
\newblock {\em Discrete Contin. Dyn. Syst. Ser. B}, 23(4):1459--1502, 2018.


\bibitem{Galimberti:2021aa}
L.~Galimberti and K.H.~Karlsen.
\newblock{Renormalization of stochastic continuity equations on Riemannian
manifolds.}
\newblock{\em Stoch. Processes their Applications}, 142:195-–244, 2021.


\bibitem{Galimberti:2024aa}
L.~Galimberti and K.H.~Karlsen.
\newblock{Well-posedness of stochastic continuity equations on
Riemannian manifolds.}
\newblock{\em Chin. Ann. Math. Ser. B}, 45(1):81--122, 2024.

\bibitem{Gess:2018ac}
B.~Gess.
\newblock Regularization and well-posedness by noise for ordinary and partial
  differential equations.
\newblock In {\em Stochastic partial differential equations and related
  fields}, volume 229 of {\em Springer Proc. Math. Stat.}, pages 43--67.
  Springer, Cham, 2018.

\bibitem{Hoel:2018aa}
H.~Hoel, K.~H. Karlsen, N.~H. Risebro, and E.~B. Storr{\o}sten.
\newblock Numerical methods for conservation laws with rough flux.
\newblock {\em Stoch. Partial Differ. Equ. Anal. Comput.}, 8(1):186--261, 2020.

\bibitem{Hoff:1985zm}
D.~Hoff and J.~Smoller.
\newblock Error bounds for finite-difference approximations for a class of
  non-linear parabolic systems.
\newblock {\em Math. Comp.}, 45(171):35--49, 1985.

\bibitem{Holden:2015aa}
H.~Holden and N.~H. Risebro.
\newblock {\em Front tracking for hyperbolic conservation laws}, volume 152 of
  {\em Applied Mathematical Sciences}.
\newblock Springer, Heidelberg, second edition, 2015.

\bibitem{Jentzen:2011aa}
A.~Jentzen and P.~E. Kloeden.
\newblock {\em Taylor approximations for stochastic partial differential
  equations}, volume~83 of {\em CBMS-NSF Regional Conference Series in Applied
  Mathematics}.
\newblock Society for Industrial and Applied Mathematics (SIAM), Philadelphia,
  PA, 2011.

\bibitem{Karper:2013aa}
T.~K. Karper.
\newblock A convergent {FEM}-{DG} method for the compressible {N}avier-{S}tokes
  equations.
\newblock {\em Numer. Math.}, 125(3):441--510, 2013.

\bibitem{Katznelson:2004aa}
Y.~Katznelson.
\newblock {\em An introduction to harmonic analysis}.
\newblock Cambridge Mathematical Library. Cambridge University Press,
  Cambridge, third edition, 2004.

\bibitem{Kroker:2012fk}
I.~Kr\"oker and C.~Rohde.
\newblock Finite volume schemes for hyperbolic balance laws with multiplicative
  noise.
\newblock {\em Appl. Numer. Math.}, 62(4):441--456, 2012.

\bibitem{Kroner:1997lq}
D.~Kr{\"o}ner.
\newblock {\em Numerical schemes for conservation laws}.
\newblock Wiley-Teubner Series Advances in Numerical Mathematics. John Wiley \&
  Sons Ltd., Chichester, 1997.

\bibitem{Krylov:1981aa}
N.~V. Krylov and B.~L. Rozovskii.
\newblock Stochastic evolution equations.
\newblock {\em Journal of Mathematical Sciences}, 16:1233--1277, 1981.

\bibitem{Kunita:1990aa}
H.~Kunita.
\newblock {\em Stochastic flows and stochastic differential equations},
  volume~24 of {\em Cambridge Studies in Advanced Mathematics}.
\newblock Cambridge University Press, Cambridge, 1990.

\bibitem{LL2022}
K.~L\^e and C.~Ling.
\newblock{Taming singular stochastic differential equations: A numerical method.}
\newblock{\em  arXiv:2110.01343v5 [math.PR]}, 1 -- 69, 2022.

\bibitem{Li:2020aa}
Y.~Li, C.-W. Shu, and S.~Tang.
\newblock A discontinuous {G}alerkin method for stochastic conservation laws.
\newblock {\em SIAM J. Sci. Comput.}, 42(1):A54--A86, 2020.

\bibitem{Lions:1998ga}
P.-L. Lions.
\newblock {\em Mathematical topics in fluid mechanics. {V}ol. 2: Compressible
  models.}
\newblock Oxford University Press, New York, 1998.

\bibitem{Liu:2007fe}
C.~Liu and N.~J. Walkington.
\newblock Convergence of numerical approximations of the incompressible
  {N}avier-{S}tokes equations with variable density and viscosity.
\newblock {\em SIAM J. Numer. Anal.}, 45(3):1287--1304 (electronic), 2007.

\bibitem{Liu:2015vb}
W.~Liu and M.~R\"{o}ckner.
\newblock {\em Stochastic partial differential equations: an introduction}.
\newblock Universitext. Springer, Cham, 2015.

\bibitem{Lord:2014aa}
G.~J. Lord, C.~E. Powell, and T.~Shardlow.
\newblock {\em An introduction to computational stochastic {PDE}s}.
\newblock Cambridge Texts in Applied Mathematics. Cambridge University Press,
  New York, 2014.

\bibitem{Majee:2018aa}
A.~K. Majee.
\newblock Convergence of a flux-splitting finite volume scheme for conservation
  laws driven by {L}\'{e}vy noise.
\newblock {\em Appl. Math. Comput.}, 338:676--697, 2018.

\bibitem{Mau2011}
M.~Maurelli.
\newblock Wiener chaos and uniqueness for stochastic transport equation.
\newblock {\em C. R. Math. Acad. Sci. Paris}, 349(11-12):669--672, 2011.

\bibitem{Merlet:2007aa}
B.~Merlet and J.~Vovelle.
\newblock Error estimate for finite volume scheme.
\newblock {\em Numer. Math.}, 106(1):129--155, 2007.

\bibitem{Ondrejat:2022aa}
M.~Ondrej{\'a}t, A.~Prohl, and N.~Walkington.
\newblock Numerical approximation of non-linear {SPDE's}.
\newblock {\em Stoch PDE: Anal Comp}, 2022.

\bibitem{Paz1983}
A.~Pazy.
\newblock {\em Semigroups of linear operators and applications to partial
  differential equations}, volume~44 of {\em Applied Mathematical Sciences}.
\newblock Springer-Verlag, New York, 1983.

\bibitem{Perthame:2002qy}
B.~Perthame.
\newblock {\em Kinetic formulation of conservation laws}, volume~21 of {\em
  Oxford Lecture Series in Mathematics and its Applications}.
\newblock Oxford University Press, Oxford, 2002.

\bibitem{Prevot:2007aa}
C.~Pr{\'e}v{\^o}t and M.~R{\"o}ckner.
\newblock {\em A concise course on stochastic partial differential equations},
  volume 1905 of {\em Lecture Notes in Mathematics}.
\newblock Springer, Berlin, 2007.


\bibitem{Punshon-Smith:2018aa}
S.~Punshon-Smith and S.~Smith.
\newblock On the {B}oltzmann equation with stochastic kinetic transport: global
  existence of renormalized martingale solutions.
\newblock {\em Arch. Ration. Mech. Anal.}, 229(2):627--708, 2018.


\bibitem{Schlichting:2017aa}
A.~Schlichting and C.~Seis.
\newblock Convergence rates for upwind schemes with rough coefficients.
\newblock {\em SIAM J. Numer. Anal.}, 55(2):812--840, 2017.

\bibitem{Schlichting:2018aa}
A.~Schlichting and C.~Seis.
\newblock Analysis of the implicit upwind finite volume scheme with rough
  coefficients.
\newblock {\em Numer. Math.}, 139(1):155--186, 2018.

\bibitem{Taylor:2011aa}
M.~E. Taylor.
\newblock {\em Partial differential equations {II}. {Q}ualitative studies of
  linear equations}, volume 116 of {\em Applied Mathematical Sciences}.
\newblock Springer, New York, second edition, 2011.

\bibitem{Vila:2003aa}
J.-P. Vila and P.~Villedieu.
\newblock Convergence of an explicit finite volume scheme for first order
  symmetric systems.
\newblock {\em Numer. Math.}, 94(3):573--602, 2003.

\bibitem{Walkington:2005jl}
N.~J. Walkington.
\newblock Convergence of the discontinuous {G}alerkin method for discontinuous
  solutions.
\newblock {\em SIAM J. Numer. Anal.}, 42(5):1801--1817 (electronic), 2005.

\bibitem{zha2010}
X.~Zhang.
\newblock{Stochastic flows of SDEs with irregular coefficients and
stochastic transport equations.}
\newblock{\em Bull. Sci. math.}, 134:340--378, 2010.

\end{thebibliography}

\end{document}